\documentclass[11pt,letter]{amsart}
\usepackage{amssymb,amsmath,epsfig,graphics,mathrsfs,enumerate,verbatim}
\usepackage[pagebackref,colorlinks=false,linkcolor=blue,citecolor=blue]{hyperref}

\DeclareFontFamily{U}{mathx}{\hyphenchar\font45}
\DeclareFontShape{U}{mathx}{m}{n}{
      <5> <6> <7> <8> <9> <10>
      <10.95> <12> <14.4> <17.28> <20.74> <24.88>
      mathx10
      }{}
\DeclareSymbolFont{mathx}{U}{mathx}{m}{n}
\DeclareFontSubstitution{U}{mathx}{m}{n}
\DeclareMathAccent{\widecheck}{0}{mathx}{"71}
\usepackage{fancyhdr}
\usepackage{hyperref}
\hypersetup{hidelinks,
 colorlinks=false,
 urlcolor     = blue,
 linkcolor    = blue,
 citecolor   = red ,
 bookmarksopen=true
}

\usepackage{amsmath}
\usepackage{amsfonts}
\usepackage{amssymb}
\usepackage{amsthm}
\usepackage{epsfig,graphics,mathrsfs}
\usepackage{graphicx}

\makeatletter
\@namedef{subjclassname@2020}{%
  \textup{2020} Mathematics Subject Classification}
\makeatother

\usepackage[usenames, dvipsnames]{color}

\usepackage{cite}

\usepackage[a4paper,
bindingoffset=0in,
left=0.8in,
right=0.8in,
top=1.3in,
bottom=1.3in,
footskip=.25in]{geometry}

\def \phi {\varphi}
\def \RNu {\mathbb{R}^{d+1}}
\def \RN {\mathbb{R}_+^{d+1}}
\def \R {\mathbb{R}}

\def \G{\Gamma}
\newcommand{\Ba}{\mathscr B_z^{(a)}}
\newcommand{\BaN}{\mathscr B_z^{(a),\mathrm N}}

\def \vf{\varphi}

\newcommand{\Rd}{\mathbb R^d}

\newcommand{\p}{\partial}

\newcommand{\la}{\lambda}

\numberwithin{equation}{section}

\newcommand{\beq}{\begin{equation}}
\newcommand{\bea}[1]{\begin{array}{#1} }
\newcommand{\eeq}{ \end{equation}}
\newcommand{\ea}{ \end{array}}

\newcommand{\ve}{\varepsilon}

\newcommand{\sa}{\langle}
\newcommand{\da}{\rangle}

\newcommand{\C}{\mathbb{C}}

\newcommand{\M}{\mathcal M}
\newcommand{\Le}{\mathcal L}

\newcommand{\ps}{P_\sigma^{\Le}}

\newcommand{\RNm}{\Rd\times (-\infty,0]}
\newcommand{\Sm}{\mathscr S_-(\RNu)}
\newcommand{\Dom}{\operatorname{Dom}}

\newtheorem{theorem}{Theorem}[section]
\newtheorem{lemma}[theorem]{Lemma}
\newtheorem{proposition}[theorem]{Proposition}
\newtheorem{corollary}[theorem]{Corollary}
\newtheorem{remark}[theorem]{Remark}
\newtheorem{definition}[theorem]{Definition}

\numberwithin{equation}{section}
\begin{document}


\title[]{Fully fractional Schr\"odinger equations with memory: well-posedness via extension} 
	


\date{}

\begin{abstract}
We study the nonlinear Schr\"odinger equation with memory
$\Le^s u=\mu|u|^{p-1}u$ for $t>0$, where
$\Le^s=(\partial_t-i\Delta_x)^s$, $0<s<1$, and the solution
is prescribed for all $t\le0$. We identify a Hilbert space
of histories and prove local well-posedness for
$1<p\le p_*(s)=1+4s/(d+2-2s)$, with global solutions for
sufficiently small histories at the critical exponent.
We also establish the passage from this memory problem to
the classical nonlinear Schr\"odinger equation as $s\uparrow1$.
For every fixed Schwartz history and $1<p\le1+4/d$, including
the classical mass critical exponent, the memory solutions
converge in $C([0,T];L^2\cap L^{p+1})$ to the classical
solution with initial value $u_0(\cdot,0)$, for every finite
$T$ below its maximal existence time. Thus the limiting
evolution depends on the prescribed past only through its
value at time zero.
Our construction uses an oscillatory Poisson extension and
an exact identity expressing the weighted $L^2$ norm of the
lifted initial datum in terms of the history. This identity
allows us to apply the boundary estimates of our earlier
work to the memory problem. We recover $\Le^s$ as a weighted
Dirichlet-to-Neumann map and establish the corresponding
extension bounds on $L^2$ for $0<s\le\frac12$, and from
$\operatorname{Dom}(\Le^s)$ to $L^2$ for $\frac12<s<1$.
\end{abstract}	

\author{Nicola Garofalo}
\address{School of Mathematical and Statistical Sciences\\ Arizona State University}\email[Nicola Garofalo]{nicola.garofalo@asu.edu}

\author{Gigliola Staffilani}
\address{Department of Mathematics\\ Massachusetts Institute of Technology\\
Cambridge, Massachusetts, 02141\\USA}\email[Gigliola Staffilani]{gigliola@math.mit.edu}

\thanks{G. Staffilani is funded in part by the NSF grant DMS-2306378 and the Simons Foundation through the Simons
Collaboration Grant on Wave Turbulence.}
	
\maketitle

\tableofcontents

\section{Introduction}\label{S:intro}

Consider the free Schr\"odinger operator $\Le = \p_t - i \Delta$ in $\Rd\times \R$, with associated unitary group $S(t) = e^{i t\Delta}$ on $L^2(\Rd)$, and denote by 
\begin{equation}\label{fullyfrac}
\Le^s=(\partial_t-i\Delta_x)^s,\qquad 0<s<1,
\end{equation}
its fractional powers. The present work studies well-posedness in the nonlinear Cauchy problem with memory \begin{equation}\label{cpmem}
\begin{cases}
\mathcal L^s u(x,t) = \mu |u(x,t)|^{p-1} u(x,t),\ \ \ x\in \Rd, t>0,
\\
u(x,t) = u_0(x,t),\ \ \ \ \ \ \ \ \ \ \ \ \ \ \ \ \ \ \ \ \ x\in \Rd, t\le 0,
\end{cases}
\end{equation}
where $\mu\in\mathbb C$, $p>1$, and the datum
$u_0:\mathbb R^d\times(-\infty,0]\longrightarrow\mathbb C$
represents the prescribed past history.

This paper is a continuation of our previous work on Schr\"odinger equations associated with the Bessel operator. In \cite{GS1} we studied the Schr\"odinger evolution associated with
the Bessel operator on the half-line. 
In \cite{GS2} we established Strichartz estimates and nonlinear boundary results for the corresponding half-space equation; see the introduction of that paper for a discussion of the literature on Schr\"odinger equations with boundary and interface interactions.
The extension constructed in the present work connects that
half-space theory developed in \cite{GS2} to the prescribed-history problem \eqref{cpmem}.

Before discussing our results further, we note that the operator $\Le^{1/2}=(\partial_t-i\Delta)^{1/2}$ appears in physics in at least two independent ways. First, up to a constant phase factor, it is the Dirichlet-to-Neumann operator of the free Schr\"odinger equation on a half-space, just as $(\partial_t-\Delta)^{1/2}$ is the Dirichlet-to-Neumann operator of the heat equation.
Second, it describes a pair of quantum particles coupled by a contact interaction tuned to resonance:
up to constants and rescaling, the inverse of the propagator of the pair field is the symbol of $\Le^{1/2}$. We next describe both occurrences in more detail. 

\subsection{Physics background} Eliminating the exterior of a flat boundary for the free Schr\"odinger equation gives
a Dirichlet-to-Neumann operator proportional to
$(\partial_t-i\Delta_\Gamma)^{1/2}$, where $\Delta_\Gamma$ acts in
the tangential variables. Its nonlocality in space and time is central
to the construction of transparent boundary conditions; see
\cite[Section~4.2.1]{AABES08}, \cite[Theorem~6]{AB01}, and, for a rectangular computational domain, \cite[equation~(13)]{YV24}. For the parabolic wave equation of paraxial optics, in which the propagation coordinate plays the role of time, the corresponding exact condition on a rectangular domain is derived in \cite{FP11}; see also \cite{FP21} for the Schr\"odinger equation on hyperrectangular domains in arbitrary dimension.

The second occurrence comes from the unitary Fermi gas. In a dilute ultracold gas, low-energy scattering between two
particles can be modeled by a contact interaction. In the
$s$-wave channel, this zero-range model is determined by the
scattering length $a_{\rm s}$; see \cite[Section~1]{NiSo12}.
The relative wave function $\psi$ solves the free equation
$-\Delta\psi=k^2\psi$ away from the origin. In three dimensions,
the contact interaction is imposed through the Bethe-Peierls
condition 
\[
\psi(r)=A\left(\frac1r-\frac1{a_{\rm s}}\right)+o(1)
\qquad(r\to0),
\]
where $A\in\mathbb C$ depends on the state.
\smallskip

This describes a one-parameter family of self-adjoint extensions of $-\Delta$ on $\mathbb R^3\setminus\{0\}$, labelled by $a_{\rm s}^{-1}$; see \cite[Chapter~I.1, Theorem~1.1.1]{AGHH}. With a magnetic field one can tune $a_{\rm s}^{-1}$ through a Feshbach resonance, and the unitary limit is $a_{\rm s}\to\infty$ \cite{NiSo12}, corresponding to $\psi(r)=A/r+o(1)$, so the contact condition contains no length scale.
Nishida and Son developed
an $\epsilon$ expansion for this system and studied its nonrelativistic
conformal symmetry in \cite{NiSo06,NiSo07}; see also \cite{NiSo12}.
In their analysis the spatial dimension $d$ is treated as a continuous parameter, the $\epsilon$ expansion being performed around $d=4$ and $d=2$; the physical case is $d=3$. For equal fermion masses $m$ and $2<d<4$,  the vacuum pair-field propagator has,
up to a nonzero constant, the form
\[
G_\phi(\omega,k)\ \propto\
\left[-\omega+\frac{|k|^2}{4m}-i0\right]^{1-d/2};
\]
see \cite[Section~3.3.1, equation~(78) with $m_\uparrow=m_\downarrow=m$, arXiv version]{NiSo12}; the same power occurs in the vacuum two-body scattering amplitude \cite[equation~(5)]{NiSo12}.
In $d=3$ the inverse pair propagator\footnote{With Fourier transform kernel $e^{i\omega t-ik\cdot x}$ and the principal branch, $\left[-\omega+\frac{|k|^2}{4m}-i0\right]^{1/2}=e^{-i\pi/4}\left[-i\left(\omega-\frac{|k|^2}{4m}\right)\right]^{1/2}$. Under the same identification $-\omega+|k|^2/(4m)$ becomes $\mathfrak p(\xi,\tau)$, so the prescription $-i0$ is the one used in \eqref{E:limiting-absorption} below.} is therefore, after a spatial rescaling and multiplication by a constant, the symbol of $\Le^{1/2}$.
The connection is also visible in physical variables. Schr\"odinger
symmetry fixes the time-ordered vacuum two-point function of a primary field of scaling
dimension $\Delta$ and mass $M$ to a constant multiple of
$t^{-\Delta}\exp(iM|x|^2/(2t))$ for $t>0$, and to zero for $t<0$ ($\hbar=1$); see \cite[equation~(62)]{NiSo12}.
For the pair field at unitarity, $\Delta=2$ and $M=2m$;
see \cite[equations~(27) and~(44)]{NiSo07} and \cite[equation~(77)]{NiSo12}. For $2<d<4$ and $s=\frac d2-1$, the kernel $t^{s-1}S(t)/\G(s)$ of $\Le^{-s}$ is, after the same spatial rescaling, a constant multiple of $t^{-2}\exp(im|x|^2/t)$ for $t>0$, in agreement with $\Delta=2$ and $M=2m$.

The square root is not confined to fermionic systems. In the
zero-range theory of identical bosons reviewed by Braaten and Hammer,
the resummed diatom propagator in units $\hbar=m=1$ is proportional to
\[
\left[a_{\rm s}^{-1}-\sqrt{-\omega+|k|^2/4-i0}\right]^{-1},
\]
where $a_{\rm s}$ denotes the scattering length;
see \cite[equation~(329)]{BH06}. At unitarity, where $a_{\rm s}^{-1}=0$,
it reduces to the inverse square-root symbol.

More broadly, nonlocal operators with memory are a classical tool for
modelling hereditary phenomena, such as dissipation in viscoelastic media
and anomalous diffusion; see \cite{CM,Mainardi,MK,DGLZ}. 
These works treat operators that differ from \eqref{fullyfrac} and also from Laskin's equation \cite{Laskin},
\[
\partial_tu+i(-\Delta_x)^su=0,
\]
in which only the spatial operator is fractional and one prescribes
the value of $u$ at $t=0$, and from the time-fractional Schr\"odinger equation in \cite{Naber}.

\subsection{The fractional Schr\"odinger operator}

Returning to \eqref{cpmem}, we represent the fractional powers
\eqref{fullyfrac}, for $u\in\mathscr S(\RNu)$, by the formula

\begin{equation}\label{E:intro-balak}
\Le^su(x,t)
=-\frac{s}{\G(1-s)}\int_0^\infty
\frac{S(\sigma)u(\cdot,t-\sigma)(x)-u(x,t)}{\sigma^{1+s}}\,d\sigma,
\end{equation}
which makes the dependence on past times
explicit. The unitary group $S(\sigma)$ leads to an oscillatory extension
kernel, in contrast with the positive kernel of the fractional heat
operator. The Fourier symbol of $\Le^s$ changes phase across the
characteristic paraboloid.
Set
\[
\mathfrak p(\xi,\tau)=2\pi\tau+4\pi^2|\xi|^2,
\]
so that in Fourier transform we have $\widehat{\Le u}=i\,\mathfrak p\,\widehat u$. Since
$\mathfrak p$ changes sign across the  paraboloid
\begin{equation}\label{para}
P=\{(\xi,\tau)\in\mathbb R^{d+1}\mid \tau = - 2\pi|\xi|^2\},
\end{equation}
the principal branch of the complex power $(i\mathfrak p)^s$ takes two different phase values,
\[
(i\mathfrak p(\xi,\tau))^s=
\begin{cases}
e^{\,i\frac{\pi s}{2}}\,|\mathfrak p(\xi,\tau)|^s, & \mathfrak p(\xi,\tau)>0,
\\[2mm]
e^{-\,i\frac{\pi s}{2}}\,|\mathfrak p(\xi,\tau)|^s, & \mathfrak p(\xi,\tau)<0,
\end{cases}
\]
with the two phases corresponding to the two sides of $P$.

We mention that the fractional powers \eqref{fullyfrac} were first investigated by
Samko and his collaborators \cite{Sa,CN,NS}, following ideas originating in the
classical work of M.~Riesz on fractional powers of the wave operator and their
subsequent extension by Four\`es-Bruhat to ultra-hyperbolic operators
\cite{R,R2,FB}. 
Here we develop an extension theory for $\mathcal L^s$ and use it
to establish well-posedness for the nonlinear Cauchy problem
with memory \eqref{cpmem}.

\subsection{Main results}\label{S:mainres}

Let $\Sm$ denote the restrictions of Schwartz functions to
$\Rd\times(-\infty,0]$. The history space $\mathcal H^s$ is the
completion of $\Sm$ under the norm $Q_s^{1/2}$, where $Q_s$ is
defined in Definition~\ref{D:Qs}. Theorem D below identifies
this norm with the weighted $L^2$ norm of the Poisson lifting.
Given a prescribed history $u_0\in\mathcal H^s$, we consider
the Cauchy problem with memory \eqref{cpmem}.
In the following theorem, memory solutions are understood
in the sense of Definition~\ref{D:memory-boundary-full}.

\medskip
\noindent\textbf{Theorem A}
(Theorems~\ref{T:memory-boundary-full}
and~\ref{T:nonlinear-classical-limit}).
{\itshape
Let $\mu\in\mathbb C\setminus\{0\}$.
\begin{enumerate}
\item[$(i)$]
Let $0<s<1$ and set
\begin{equation}\label{pstars}
p_*(s)=1+\frac{4s}{d+2-2s}.
\end{equation}
For $1<p\le p_*(s)$, every prescribed history
$u_0\in\mathcal H^s$ admits a unique local mild solution
\[
u\in L^q((0,T);L^{p+1}(\Rd)),
\qquad
\frac2q+\frac d{p+1}=\frac{d+2-2s}{2}.
\]
The solution map $u_0\mapsto u$ is locally Lipschitz from
$\mathcal H^s$ into $L^q((0,T);L^{p+1}(\Rd))$,
where $T>0$ can be chosen uniformly for histories in
a sufficiently small neighborhood of any fixed history.
If $p<p_*(s)$, then for every $M>0$ there exists
$T=T(d,s,p,|\mu|,M)>0$ such that every history satisfying
$\|u_0\|_{\mathcal H^s}\le M$ admits a solution on $(0,T)$. At $p=p_*(s)$, sufficiently small
$\|u_0\|_{\mathcal H^s}$ gives a global solution satisfying
\[
u\in L^q((0,\infty);L^{p+1}(\Rd)).
\]
\item[$(ii)$]
Fix $u_0\in\Sm$ and $1<p\le1+4/d$.
Let $u_1$ be the maximal classical mild solution of
\[
\begin{cases}
(\partial_t-i\Delta_x)u_1
=\mu|u_1|^{p-1}u_1,&t>0,\\
u_1(\cdot,0)=u_0(\cdot,0).
\end{cases}
\]
Let $[0,T_{\max})$ be the maximal interval of existence
of the classical mild solution 
\[
u_1\in
C([0,T_{\max});L^2(\Rd)\cap L^{p+1}(\Rd)).
\]
For every finite $0<T<T_{\max}$ and all $s<1$
sufficiently close to $1$, there is a memory solution
with prescribed history $u_0$ satisfying
\[
u_s\in C([0,T];L^2(\Rd)\cap L^{p+1}(\Rd)).
\]
These solutions are unique in
$C([0,T];L^{p+1}(\Rd))$, and
\[
u_s\longrightarrow u_1
\quad\text{in }
C([0,T];L^2(\Rd)\cap L^{p+1}(\Rd))
\qquad\text{as }s\uparrow1.
\]
\end{enumerate}
\par}

\medskip

The convergence in part~$(ii)$ includes the classical
mass critical power $p=1+4/d$. For every $s<1$, this
power lies strictly above $p_*(s)$, so the existence
of the corresponding memory solutions does not follow
from part~$(i)$.
We establish their existence and convergence in
Theorem~\ref{T:nonlinear-classical-limit}.
For a fixed Schwartz history, its proof uses estimates
uniform in $s$ near $1$ to construct the memory solutions
and compare them with the classical nonlinear
Schr\"odinger solution. The convergence holds on every
interval $[0,T]$ with $T<T_{\max}$, including at
$p=1+4/d$. Thus the classical NLS evolution is recovered
from the problem with prescribed past, with
$u_0(\cdot,0)$ as its initial datum.

\smallskip

\begin{remark}\label{R:afterthA}
For $0<s\le\frac12$ and $1<p<p_*(s)$,
Corollary~\ref{C:global-memory} also gives global
continuation of the solution in part~$(i)$ whenever
\[
\mu\in e^{i\pi s/2}\R.
\]
See Remark~\ref{R:admissible-mu} for the relation
between this condition and mass conservation
for the half-space evolution.
\end{remark}

Henceforth in this paper we denote by
\[
\Ba=\p_{zz}+\frac az\p_z ,\qquad a=1-2s\in(-1,1),
\]
the Bessel operator on $\R^+_z$. We work in
$\RN=\Rd_x\times\R^+_z$ with measure $d\omega_a(X)=z^adxdz$
and use the notation of Section~\ref{S:not}.

The next theorem solves the extension equation in Fourier variables
and identifies its weighted Dirichlet-to-Neumann map. To state it, consider the phase function $\theta(\xi,\tau)$ introduced in \eqref{phase}: we note that
$\theta(\xi,\tau)=1$ when $\mathfrak p(\xi,\tau)>0$ and
$\theta(\xi,\tau)=-i$ when $\mathfrak p(\xi,\tau)<0$.
The notation $K_s$ indicates the Bessel function in \eqref{Knu}.

\medskip
\noindent\textbf{Theorem B} (Theorems~\ref{T:itstheFT} and~\ref{T:classical}).
\emph{Let $0<s<1$ and $u\in\mathscr S(\RNu)$. The extension problem
\[
\p_tU-i\bigl(\Delta_xU+\Ba U\bigr)=0
\quad\text{in }\RN\times\R,
\qquad
U(x,0,t)=u(x,t),
\]
admits a solution $U$ with the Fourier representation
\begin{equation}\label{phrep}
\widehat U(\xi,z,\tau)
=\frac{2^{1-s}}{\G(s)}
\Bigl(z\,\theta(\xi,\tau)|\mathfrak p(\xi,\tau)|^{\frac12}\Bigr)^{s}
K_{s}\Bigl(z\,\theta(\xi,\tau)|\mathfrak p(\xi,\tau)|^{\frac12}\Bigr)\,\widehat u(\xi,\tau),
\end{equation}
where the multiplier is assigned its continuous value $1$
on the paraboloid \eqref{para}. The weighted Neumann derivative of $U$ satisfies
\[
-\frac{2^{2s-1}\G(s)e^{i\frac{\pi s}2}}{\G(1-s)}
\,z^{1-2s}\partial_zU(\cdot,z,\cdot)
\longrightarrow \Le^su
\quad\text{strongly in }L^2(\RNu)
\quad\text{as }z\to0^+.
\]
}

\medskip

We now express the extension function $\widehat{U}$ of Theorem B  in physical variables.
By \eqref{phrep}, we have
\[
\widehat U(\xi,z,\tau)
=m_s(z;\xi,\tau)\widehat u(\xi,\tau),
\]
where
\[
m_s(z;\xi,\tau)
=
\frac{2^{1-s}}{\G(s)}
\Bigl(z\,\theta(\xi,\tau)|\mathfrak p(\xi,\tau)|^{1/2}\Bigr)^s
K_s\Bigl(z\,\theta(\xi,\tau)|\mathfrak p(\xi,\tau)|^{1/2}\Bigr).
\]
Thus, for each fixed $z>0$, the inverse Fourier transform
of $m_s(z;\cdot,\cdot)$ is the kernel that maps the boundary
datum $u$ to $U(\cdot,z,\cdot)$. The following theorem
computes this kernel explicitly.

\medskip
\noindent\textbf{Theorem C} (Theorem \ref{T:invFT}).
{\itshape
Let $0<s<1$ and $a=1-2s$. For each $z>0$, the inverse
Fourier transform of $m_s(z;\cdot,\cdot)$ in
$\mathscr S'(\RNu)$ is
$\mathbb P_a((x,z),(0,0),t)$, where, writing
$X=(x,z)$ and $Y_0=(y,0)$,
\[
\mathbb P_a(X,Y_0,t)=P_a(z,t)\,S(x,y,t),
\]
with $S(x,y,t)$ the free Schr\"odinger kernel in
\eqref{Scla} and
\[
P_a(z,t)
=
\frac{e^{\,i\frac{(a-1)\pi}{4}}}
     {2^{1-a}\G\!\left(\frac{1-a}{2}\right)}
\frac{z^{1-a}}{t^{\frac{3-a}{2}}}
e^{i\frac{z^2}{4t}}\,
\mathbf 1_{(0,\infty)}(t).
\]
For every $u\in\mathscr S(\RNu)$, the extension in
Theorem B is therefore given by
\begin{equation}\label{U}
U(X,t)
=
\int_0^\infty\int_{\Rd}
\mathbb P_a(X,Y_0,\sigma)\,
u(y,t-\sigma)\,dy\,d\sigma.
\end{equation}
For fixed $z>0$, this formula is understood as the
convolution in $(x,t)$ of the tempered distribution
$\mathbb P_a((x,z),(0,0),t)$ with the Schwartz function $u$.
\par}

\medskip

For a prescribed history $u_0\in\Sm$, the Poisson kernel
$\mathbb P_a$ of Theorem C gives the initial datum for the
half-space equation through the lifting
\begin{equation}\label{lift}
U_0(X)=\mathcal P_su_0(X)
:=\int_0^\infty\int_{\Rd}\mathbb P_a(X,Y_0,\sigma)\,u_0(y,-\sigma)\,dy\,d\sigma .
\end{equation}
The weighted energy of this lifting determines the history space.

\medskip

\noindent\textbf{Theorem D} (Theorem \ref{T:history} and Corollary \ref{C:extends}).
\emph{There is an explicit nonlocal quadratic form $Q_s$ on histories
such that for every $u_0\in\Sm$, one has
\[
Q_s(u_0) = \|\mathcal P_su_0\|^2_{L^2_a(\RN)}.
\]
Denoting by $\mathcal H^s$ the completion of $\Sm$ under $Q_s^{1/2}$, the
Poisson lifting extends uniquely to an isometry
$\mathcal P_s:\mathcal H^s\to L^2_a(\RN)$.}

\medskip

Theorem D allows us to apply the boundary estimates of
\cite{GS2}: the lifted history is an admissible initial datum
 for the singular Schr\"odinger equation
in $\RN$ with nonlinear Neumann interaction,
\begin{equation}\label{cp020}
\begin{cases}
\partial_tU-i(\Delta_xU+\Ba U)=0,
&X\in\RN,\ t>0,
\\
\displaystyle
-\frac{2^{2s-1}\G(s)e^{i\frac{\pi s}2}}{\G(1-s)}\lim_{z\to0^+}z^a\partial_zU
=\mu |U(x,0,t)|^{p-1}U(x,0,t),
&x\in\Rd,\ t>0,
\\
U(X,0)=U_0(X).
\end{cases}
\end{equation}

\begin{remark}\label{afterthD}
It is worth reiterating that the history $u_0$ supplies the initial datum $U_0=\mathcal P_su_0$. With such $U_0$, in
Section~\ref{S:wellpos} we use the boundary estimates of
\cite{GS2} to prove Theorem A, with the boundary trace
understood as in Definition~\ref{D:memory-boundary-full}.
\end{remark}

For later use, we set
\begin{equation}\label{cs0}
c_s=-\frac{2^{1-2s}\Gamma(1-s)}{\Gamma(s)}
e^{-i\frac{\pi s}2}.
\end{equation}
The Neumann condition in \eqref{cp020} then reads
\[
\lim_{z\to0^+}z^a\partial_zU
=c_s\mu|U(x,0,t)|^{p-1}U(x,0,t).
\]


\subsection{Outline of the proofs}\label{S:des}

We construct the extension by taking the Fourier transform
in $(x,t)$, which reduces the extension equation to a Bessel
equation. For $\mathfrak p>0$, decay at infinity and the
prescribed boundary value determine the extension multiplier.
We define the multiplier for $\mathfrak p<0$ by continuation
from the lower half-plane, with
$\sqrt{\mathfrak p-i0}=-i\sqrt{|\mathfrak p|}$.
Fourier inversion then gives the oscillatory Poisson kernel.

The regularity required of the boundary datum depends on
the growth of this multiplier in the region $\mathfrak p<0$.
There the Bessel argument lies on the negative imaginary
axis, where $M_s(w)=w^sK_s(w)$ is bounded for
$0<s\le\frac12$ and grows like $|w|^{s-\frac12}$
for $\frac12<s<1$
(Lemma~\ref{L:MultiplierBound}).
Consequently, the extension is bounded on $L^2(\RNu)$
in the first range. In the second range, the graph norm
of $\Dom(\Le^s)$ controls the growth of the multiplier
and gives the extension stated in
Theorem~\ref{T:L2Extension}.

To compute the weighted Neumann derivative
$z^a\partial_zU$, with $a=1-2s$, we use the Bessel
intertwining identity
\[
w^{1-2s}\frac{d}{dw}\bigl(w^sK_s(w)\bigr)
=-w^{1-s}K_{1-s}(w).
\]
Applied to the Fourier representation of $U$, this
identity expresses the weighted derivative in terms
of the Bessel factor $M_{1-s}(w)=w^{1-s}K_{1-s}(w)$.
The bounds of Lemma~\ref{L:MultiplierBound}, now applied
with order $1-s$, give the strong $L^2$ Neumann convergence
of Theorem~\ref{T:DtNL2}: the assumption
$u\in\Dom(\Le^s)$ suffices when $\frac12\le s<1$,
whereas for $0<s<\frac12$ we assume
$u\in\Dom(\Le^{\frac14+\frac s2})$.

We can now return to the Cauchy problem with memory,
for which the prescribed past $u_0$ supplies the initial
datum of the half-space equation through its Poisson
lifting. To measure this datum directly in terms of the
history, we compute its squared weighted $L^2$ norm by
a Mellin transform calculation. The resulting identity,
Theorem~\ref{T:history}, allows us to express the estimates
for the half-space evolution in our work \cite{GS2}
in the history norm. For the nonlinear argument, we
also need to estimate the boundary trace of the free
evolution. We obtain this estimate by duality from the
boundary forcing estimate of \cite{GS2}, throughout
the range $-1<a<1$
(Lemma~\ref{L:free-boundary-full}). Together with the
boundary Duhamel estimate, it gives a contraction for
the Volterra equation \eqref{E:memory-volterra-full}
when $1<p\le p_*(s)$, proving
Theorem~\ref{T:memory-boundary-full}.

As $s\uparrow1$, the fractional operators approach the
classical Schr\"odinger operator:
$\Le^su\to\Le u$ in $L^2(\RNu)$ for
$u\in\mathscr S(\RNu)$.
At the same time, the exponent $p_*(s)$ in \eqref{pstars} tends to the
classical mass critical exponent $1+4/d$.
It is therefore natural to ask whether the nonlinear
evolution with memory also approaches the classical
Schr\"odinger evolution, including at the mass critical
exponent.

Theorem~\ref{T:nonlinear-classical-limit} establishes
precisely this convergence for a fixed Schwartz history
$u_0$ and every $1<p\le1+4/d$. The limiting solution
$u_1$ is the classical nonlinear Schr\"odinger evolution
with initial datum $u_0(\cdot,0)$. For every finite
$T<T_{\max}$, the memory solutions exist on $[0,T]$
when $s$ is sufficiently close to $1$, and
\[
u_s\longrightarrow u_1
\quad\text{in }
C([0,T];L^2(\Rd)\cap L^{p+1}(\Rd))
\qquad\text{as }s\uparrow1.
\]
Thus the prescribed history gives way, in this limit,
to its value at the initial time.

To prove Theorem~\ref{T:nonlinear-classical-limit},
we first consider the linear problem $\Le^su=0$
for $t>0$, with prescribed past $u_0$.
The explicit
formula in Lemma~\ref{L:concrete-free-history} shows
that its solution converges to $S(t)u_0(\cdot,0)$
in $C([0,T];H^1(\Rd))$ for every finite $T>0$.
We use this convergence together with the dispersive
estimate to construct the nonlinear solutions on
a common time interval for $s$ close to $1$ and
to prove their convergence in $C_tL_x^{p+1}$.
The classical dual Strichartz estimate then yields
convergence in $C_tL_x^2$.

To reach any prescribed $T<T_{\max}$, we apply the
local stability argument on successive intervals,
whose lengths are chosen using the bound for the
classical solution on $[0,T]$. In each new interval,
the nonlinear integral over earlier times appears
as a known forcing term. The convergence already
proved on those earlier intervals controls this
term, allowing us to extend the convergence to the
new interval. A finite number of these extensions
covers $[0,T]$.



\subsection{Extension problems for nonlocal operators}\label{S:epnonlocal}

Caffarelli and Silvestre \cite{CS} realized the fractional Laplacian
as a weighted Dirichlet-to-Neumann map. This construction allows local
methods to be applied to nonlocal equations.

An extension for the square root of the heat operator
$(\partial_t-\Delta_x)^{1/2}$ appeared earlier in the work of
B.~F.~Jones, Jr.\ \cite{Jo} on parabolic Lipschitz spaces. The semigroup approach to extension
problems for fractional powers was developed by Stinga and Torrea
\cite{ST2010}, and extended by Gal\'e, Miana and Stinga \cite{GMS} to
generators of integrated semigroups. In the particular setting of the
fractional heat operator $(\partial_t-\Delta_x)^s$, the extension theory
and corresponding regularity theory were independently established by
Nystr\"om-Sande and Stinga-Torrea \cite{NyS,ST}. The recent work \cite{FD}
on the fully fractional heat equation studies nonlinear equations with
memory, including blow-up and Fujita-type phenomena.

Enciso, Gonz\'alez and Vergara \cite{EGV} obtained fractional powers
of the wave operator as a Dirichlet-to-Neumann map for the Klein-Gordon
equation in anti-de Sitter spacetime. Their construction also requires
a choice of branch across a characteristic set. Here the extension
equation is of Schr\"odinger type and the characteristic set is the
paraboloid \eqref{para}.

There is also a related bulk-boundary construction in the physics
literature. Son \cite{Son08} proposed a geometric realization of
Schr\"odinger symmetry and related the two boundary asymptotics of a
bulk scalar field to the scaling dimensions of free and unitary
fermion pairs. The Fourier-transformed bulk equation is a Bessel
equation; see \cite[Section~V, equations~(28), (29) and~(33)]{Son08}.  
The geometry and boundary conditions there differ from those of
our weighted half-space problem.


\subsection{Notation}\label{S:not}
\begin{itemize}
\item Generic points in the bulk space $\RN$ will be indicated with $X = (x,z), Y = (y,\zeta)$, etc.;
\item points of the boundary 
$\p \RN$ will be denoted by $X_0 = (x,0), Y_0 = (y,0)$, etc.; 
\item space-time points in the bulk $\RN\times \R$ will be denoted by
$(X,t), (Y,\tau)$, etc.; 
\item as customary, $\mathscr S(\R^{d+1})$ is the Schwartz class in the variables $(x,t)\in \R^{d+1}$;
\item we  denote by $\Sm$ the restrictions to the half-space $\RNm$ of the Schwartz functions in $\RNu$;
\item we indicate by  $d\omega_a(X) = z^a dx dz$ the invariant measure for the Weinstein operator $\Delta_x +\Ba$ in $\RN$. We mean by this that given $U, V\in C^\infty_0(\RN)$, we have 
 \[
\int_{\RN} \overline V(\Delta_x U + \Ba U)d\omega_a(X) = \int_{\RN} U\ \overline{(\Delta_x V + \Ba V)}d\omega_a(X);
\]
\item finally, we denote $L^2_a(\RN) =  L^2(\RN,d\omega_a(X))$. 
\end{itemize}

\subsection{Organization of the paper}

Section \ref{S:fullyfra} introduces $\Le^s$, first through the shifted
Schr\"odinger semigroup and then as a Fourier multiplier, and establishes the
equivalence of the two definitions. 
Section~\ref{S:osc} gives the Fourier representation in
Theorem B. In Section~\ref{S:appears} we prove Theorem C
and derive the Poisson representation \eqref{U}.
Section~\ref{S:schwartz} establishes classical solvability
for Schwartz boundary data and proves the strong $L^2$
boundary limits stated in Theorem B. Section~\ref{S:L2poi} establishes boundedness of the
Poisson extension on $L^2(\RNu)$ for $0<s\le\frac12$,
and from $\Dom(\Le^s)$ into $L^2(\RNu)$ for
$\frac12<s<1$. It also proves the Dirichlet-to-Neumann
formula as a strong $L^2$ limit for
$u\in\Dom(\Le^s)$ when $\frac12\le s<1$,
and for $u\in\Dom(\Le^{\frac14+\frac s2})$
when $0<s<\frac12$.
 Section \ref{S:inverse}
develops the potential theory associated with $\Le^s$: we construct the inverse
fractional operator, identify its fundamental solution, compute its oscillatory
Poisson lifting, and relate the latter to the boundary Duhamel operator of
\cite{GS2}. Section \ref{S:history} is devoted to the prescribed past and proves
Theorem D. Section~\ref{S:wellpos} proves the boundary
well-posedness and classical-limit results announced in
Theorem A. The appendix, Section~\ref{S:app}, collects the
Bessel-function and oscillatory-integral identities used
throughout.

\section{The fully fractional Schr\"odinger operator}\label{S:fullyfra}

Throughout this work the notation $S(x,y,t)$ indicates the classical integral kernel of the unitary group $e^{i t\Delta_x}$,
\begin{equation}\label{Scla}
S(x,y,t)
=
(4\pi i t)^{-\frac d2}
e^{i \frac{|x-y|^2}{4t}},
\qquad t\neq 0.
\end{equation}
The solution to the Cauchy problem $\Le u = 0$, $u(x,0) =\vf(x)$, is given by
\[
u(x,t) = \int_{\Rd} S(x,y,t) \vf(y) dy.
\]
For a function $u:\R^{d+1}\to \overline \C$, its Fourier transform is
\begin{equation}\label{fou}
\hat u(\xi,\tau) = \int_{\R^{d+1}} e^{-2\pi i (\tau t + \sa\xi,x\da)} u(x,t) dx dt.
\end{equation}
Consider the semigroup $P^{\Le}_\sigma = e^{-\sigma \Le}$ associated with the Cauchy problem
\begin{equation}\label{PsS2}
\p_\sigma f = - \Le f,\ \ \ \ \ f(x,t,0) = u(x,t).
\end{equation}
For $u\in \mathscr S(\R^{d+1})$, the solution to \eqref{PsS2} is given by
\begin{align}\label{PsS}
f(x,t,\sigma) = P^{\Le}_\sigma u(x,t) & = \int_{\Rd} S(x,y,\sigma) u(y,t-\sigma) dy
 = S(\sigma)u(\cdot,t-\sigma)(x).
\end{align}
Observe that, if  
\begin{equation}\label{p}
\mathfrak p(\xi,\tau)= 2\pi \tau + 4\pi^2|\xi|^2,
\end{equation}
then \eqref{PsS} gives 
\begin{equation}\label{fSF}
\widehat{P^\Le_\sigma u}(\xi,\tau)  = e^{-i \sigma \mathfrak p(\xi,\tau)} \hat u(\xi,\tau),
\end{equation}
and therefore
\begin{equation}\label{map}
P^\Le_\sigma : \mathscr S(\R^{d+1})\ \longrightarrow\ \mathscr S(\R^{d+1}).
\end{equation}
An elementary, yet basic consequence of \eqref{fSF}
 is the following.

\begin{proposition}\label{P:L2}
The semigroup $\{P^\Le_\sigma\}_{\sigma>0}$ extends to a unitary one on $L^2(\RNu)$, i.e., one has 
\begin{equation}\label{Suni}
||P^{\mathcal L}_\sigma u||_{L^2(\R^{d+1})} =  ||u||_{L^2(\R^{d+1})}. 
\end{equation}
\end{proposition}

Henceforth we denote by  
\begin{equation}\label{Lmasomeno}
L^+ = \{(\xi,\tau)\in \R^{d+1}\mid \mathfrak p(\xi,\tau) >0\},\ \ \ \ L^- = \{(\xi,\tau)\in \R^{d+1}\mid \mathfrak p(\xi,\tau) <0\}.
\end{equation}
 For a set $E$ we will denote by $\mathbf 1_E$ its indicator function.

\begin{definition}\label{D:ff} 
On functions $u\in \mathscr S(\RNu)$, we define the \emph{fully fractional} Schr\"odinger operator as  
\begin{equation}\label{fS}
\mathcal L^s u(x,t) = - \frac{s}{\G(1-s)} \int_0^\infty \frac{1}{\sigma^{1+s}} \left[P^{\Le}_\sigma u(x,t) - u(x,t)\right] d\sigma,\ \ \ \ \ 0<s<1.
\end{equation}
\end{definition}

Next, we observe that if we presently define the parabolic dilations 
\begin{equation}\label{pardila}
\delta_\la u(x,t) = u(\la x,\la^2 t),
\end{equation}
then elementary arguments show that
\begin{equation}\label{Ptdilheat}
P^\Le_\sigma(\delta_\la u)(x,t) = P^\Le_{\la^2 \sigma} u(\la x,\la^2 t).
\end{equation}
One easily obtains from \eqref{Ptdilheat}
\begin{equation}\label{dilfracheat}
\Le^s(\delta_\la u)(x,t) = \la^{2s} \Le^s u(\la x,\la^2 t),
\end{equation}  
which shows that the fully fractional Schr\"odinger operator $\Le^s$ is of order $2s$ with respect to the anisotropic parabolic dilations \eqref{pardila}.

The following proposition identifies the Fourier multiplier of
$\Le^s$ as $(i\mathfrak p(\xi,\tau))^s$.

\begin{proposition}\label{P:FTfp}
Let $0<s<1$. Then for $u\in \mathscr S(\R^{d+1})$ one has in $\mathscr S'(\R^{d+1})$
\begin{equation}\label{fp1}
\widehat{\mathcal L^{s} u}(\xi,\tau) = \left\{e^{-i\frac{\pi s}2} \mathbf 1_{L^-}(\xi,\tau) + e^{i\frac{\pi s}2} \mathbf 1_{L^+}(\xi,\tau)\right\}|\mathfrak p(\xi,\tau)|^{s} \hat u(\xi,\tau).
\end{equation}
Equivalently, we have
\begin{equation}\label{fp2}
\widehat{\mathcal L^{s} u} = (i \mathfrak p)^s \hat u.
\end{equation}
\end{proposition}

\begin{proof}
 We need to show that, for every $\vf\in \mathscr S(\R^{d+1})$, we have
\begin{align}\label{ex}
& \int_{\R^{d+1}} \hat \vf(\xi,\tau) \mathcal L^{s} u(\xi,\tau) d\xi d\tau
\\
& = \int_{\R^{d+1}} \vf(\xi,\tau) \left\{e^{-i\frac{\pi s}2} \mathbf 1_{L^-}(\xi,\tau) + e^{i\frac{\pi s}2} \mathbf 1_{L^+}(\xi,\tau)\right\}|\mathfrak p(\xi,\tau)|^{s} \hat u(\xi,\tau) d\xi d\tau.
\notag
\end{align}
From definition \eqref{fS}  and Fubini, we have
\begin{align*}
& \int_{\R^{d+1}} \hat \vf(\xi,\tau) \mathcal L^{s} u(\xi,\tau) d\xi d\tau  
 = - \frac{s}{\G(1 - s)} \int_{\R^{d+1}} \hat \vf(\xi,\tau) \int_0^\infty \frac{1}{\sigma^{1+s}} \left[P^\mathcal L_\sigma u(\xi,\tau) - u(\xi,\tau)\right] d\sigma d\xi d\tau
\\
& = - \frac{s}{\G(1 - s)} \int_{\R^{d+1}}\vf(\xi,\tau)\int_0^\infty \frac{1}{\sigma^{1+s}} \left[e^{-i \sigma \mathfrak p(\xi,\tau)} - 1\right]d\sigma\  \hat u(\xi,\tau)  d\xi d\tau 
\\
& = - \frac{s}{\G(1 - s)} \int_0^\infty \frac{1}{\sigma^{1+s}}  \int_{L^-} \vf(\xi,\tau) \left[e^{i \sigma |\mathfrak p(\xi,\tau)|} - 1\right]\hat u(\xi,\tau)  d\xi d\tau d\sigma
\\
&   - \frac{s}{\G(1 - s)} \int_0^\infty \frac{1}{\sigma^{1+s}}  \int_{L^+} \vf(\xi,\tau)\left[e^{-i \sigma \mathfrak p(\xi,\tau)} - 1\right]\hat u(\xi,\tau)  d\xi d\tau d\sigma
\\
& = - \frac{s}{\G(1 - s)} \int_{L^-} \vf(\xi,\tau) \hat u(\xi,\tau) \int_0^\infty \frac{1}{\sigma^{1+s}} \left[e^{i \sigma |\mathfrak p(\xi,\tau)|} - 1\right] d\sigma d\xi d\tau
\notag
\\
&  - \frac{s}{\G(1 - s)}  \int_{L^+} \vf(\xi,\tau) \hat u(\xi,\tau) \int_0^\infty \frac{1}{\sigma^{1+s}}\left[e^{- i \sigma \mathfrak p(\xi,\tau)} - 1\right]d\sigma d\xi d\tau.
\end{align*}
Applying Lemma \ref{L:aux} we obtain
 \begin{align*}
& \int_{\R^{d+1}} \hat \vf(\xi,\tau) \mathcal L^{s} u(\xi,\tau) d\xi d\tau
 = e^{-i\frac{\pi s}2} \int_{L^-} \vf(\xi,\tau) \hat u(\xi,\tau) |\mathfrak p(\xi,\tau)|^{s}  d\xi d\tau
\\
& + e^{i\frac{\pi s}2} \int_{L^+} \vf(\xi,\tau) \hat u(\xi,\tau) \mathfrak p(\xi,\tau)^{s}  d\xi d\tau.
\end{align*}
This completes the proof.

\end{proof}

\begin{remark}\label{R:not}
We note explicitly that formula \eqref{fp2} in Proposition \ref{P:FTfp} shows that 
\[
\Le^s:\mathscr S(\R^{d+1})\ \not\longrightarrow\ \mathscr S(\R^{d+1}).
\]
\end{remark}

\begin{definition}\label{D:domain}
Motivated by Proposition~\ref{P:FTfp}, we realize $\Le^s$ as a closed
operator on $L^2(\RNu)$ by setting
\[
\Dom(\Le^s)
=
\left\{u\in L^2(\RNu): |\mathfrak p|^s\widehat u\in L^2(\RNu)\right\},
\qquad
\widehat{\Le^su}:=(i\mathfrak p)^s\widehat u,
\]
where the principal branch is used. On $\mathscr S(\RNu)$ this agrees with the
semigroup definition in Definition~\ref{D:ff}, by Proposition~\ref{P:FTfp}.
We equip $\Dom(\Le^s)$ with the graph norm
\[
\|u\|_{\Dom(\Le^s)}
=
\|u\|_{L^2(\RNu)}+\|\Le^su\|_{L^2(\RNu)}.
\]
\end{definition}
For $0<s<1$, comparison of $1+|\mathfrak p|^s$ with $(1+|\mathfrak p|^2)^{s/2}$ gives
\begin{equation}\label{equigraph}
\|u\|_{\Dom(\Le^s)}
\cong \|(1+|\mathfrak p|^2)^{\frac s2}\widehat u\|_{L^2(\RNu)}.
\end{equation}

 
\subsection{Comparison with Samko's definition}\label{S:samko}

Formula \eqref{fp1} shows that our definition of the nonlocal operators \eqref{fS} coincides with that adopted in \cite[Sec.4, p.296]{Sa}
\begin{equation}\label{sam1}
\Le^s u = \mathscr F^{-1}\left\{(e^{i\frac{\pi s}2} \mathbf 1_{L^+} + e^{-i\frac{\pi s}2} \mathbf 1_{L^-})|\mathfrak p(\xi,\tau)|^{s}\ \hat u\right\},
\end{equation}
see also \cite{Samp} for a previous study concerning the fractional powers $(\p_t -\Delta_x)^s$.
The semigroup formula \eqref{fS} defines $\Le^s$ without first
inverting its Fourier symbol. Samko also gives the representation
\begin{align}\label{sam}
& (\Delta_x + i \p_t)^s u(x,t) = \frac{e^{i\frac{\pi s }{2}}}{d_{d,1}(s)}\int_0^\infty \frac{1}{\sigma^{1+s}} \int_{\Rd} \tilde \Delta^{(1)}_{y,\sigma} u(x,t) S(y,\sigma) dy d\sigma, 
\end{align} 
where for arbitrary $b>1$, he sets 
\[
d_{d,1}(s) = \G(-s)\left[C^{(1)}_0 - b^{s} C^{(1)}_1\right] = \G(-s)\left[1 - b^{s}\right] = - \frac{\G(1-s)}s \left[1 - b^{s}\right],
\]
and 
\[
\tilde \Delta^{(1)}_{y,\sigma} u(x,t) = u(x-y,t-\sigma) - u(x-\sqrt b y,t - b \sigma).
\]
Using these quantities, after a direct  computation one  obtains
\begin{align*}
& e^{-i\frac{\pi s }{2}}(\Delta_x + i \p_t)^s u(x,t) = \Le^s u(x,t).
\end{align*} 
Hence, Samko's alternative definition \eqref{sam} coincides with the semigroup definition \eqref{fS}.

\vskip 0.2in


\section{The oscillatory extension problem in phase space}\label{S:osc}

In this section we solve the extension equation in Fourier variables and identify
its Dirichlet-to-Neumann map. Then in Section~\ref{S:appears} we compute
the inverse transform and this leads to the Poisson kernel.

Let $0<s<1$, $a=1-2s$, and $X=(x,z)\in\RN$. For
$u\in\mathscr S(\RNu)$, we seek $U(X,t)$, smooth for $z>0$
and continuous up to $z=0$, satisfying
\begin{equation}\label{expbS}
\begin{cases}
z^a \p_t U- i \operatorname{div}_X(z^a \nabla_X U) = 0 \ \ \ \ \ \text{in}\ \ \R^{d+1}\times\R^+,
\\
U(x,0,t) = u(x,t),\ \ \ \ \ \ \ \ \ \ \ \ \ \ \ \ \ \ \ (x,t)\in \RNu.
\end{cases}
\end{equation} 
If we denote by 
\begin{equation}\label{Ba}
\mathscr B_z^{(a)} = \p_{zz}  + \frac az \p_z,\ \ \ \ \ \ -1<a<1,
\end{equation}
the Bessel operator on the half-line $\R^+$, then 
\[
z^a \p_t U - i \operatorname{div}_X(z^a \nabla_X U) = z^a\left\{\p_t U - i \left(\Delta_x U+ \mathscr B_z^{(a)} U\right)\right\}.
\]
Therefore, we can rewrite the problem \eqref{expbS} in the following equivalent form.

\begin{definition}[The extension problem]\label{D:ep}
Given $u\in \mathscr S(\RNu)$, find a function $U\in C^\infty(\RNu\times \R^+)\cap C(\RNu\times [0,\infty))$,  such that
\begin{equation}\label{ep}
\p_t U - i \left(\Delta_x U + \mathscr B_z^{(a)} U\right) = 0,\ \ \ \ \ U(x,0,t) = u(x,t),\ \ \ \  (x,t)\in \R^{d+1}.
\end{equation}
\end{definition}

\subsection{The extension problem in phase space}\label{S:poiFT}

Given $u\in \mathscr S(\R^{d+1})$, we denote by  
\[
\widehat U(\xi,z,\tau) = \int_{\R^{d+1}} e^{-2\pi i (\tau t + \sa\xi,x\da)} U(x,z,t) dx dt
\]
the partial Fourier transform with respect to $(x,t)$ of the unknown $U$ in \eqref{ep}. 
Taking the Fourier transform in $(x,t)$ gives the ordinary
differential equation
\begin{equation}\label{Bah}
\Ba\widehat U(\xi,z,\tau)
=
\mathfrak p(\xi,\tau)\widehat U(\xi,z,\tau),
\qquad
\widehat U(\xi,0,\tau)=\widehat u(\xi,\tau),
\end{equation}
where $\mathfrak p(\xi,\tau)$ is the real-valued symbol in \eqref{p}.
For each $(\xi,\tau)$ with $\mathfrak p(\xi,\tau)\ne0$, the equation
has two linearly independent solutions, so the boundary
condition alone does not determine $\widehat U$.
We construct the extension in the form
\[
\widehat U(\xi,z,\tau)
=
m_s(z;\mathfrak p(\xi,\tau))\widehat u(\xi,\tau),
\]
where the multiplier $m_s$ satisfies
\[
\Ba m_s(z;\mathfrak p)=\mathfrak p\,m_s(z;\mathfrak p),
\qquad
\lim_{z\to0^+}m_s(z;\mathfrak p)=1.
\]
For $\mathfrak p>0$, we choose the multiplier that decays as
$z\to\infty$, namely
\[
m_s(z;\mathfrak p)
=
\frac{2^{1-s}}{\G(s)}
(z\sqrt{\mathfrak p})^sK_s(z\sqrt{\mathfrak p}).
\]
For $\mathfrak p<0$, we define the multiplier by continuing this
formula in an auxiliary complex parameter and approaching
$\mathfrak p$ from the lower half-plane:
\begin{equation}\label{E:limiting-absorption}
m_s(z;\mathfrak p)
=
\lim_{\varepsilon\downarrow0}
\frac{2^{1-s}}{\G(s)}
\bigl(z\sqrt{\mathfrak p-i\varepsilon}\bigr)^s
K_s\bigl(z\sqrt{\mathfrak p-i\varepsilon}\bigr).
\end{equation}
Here the square root has positive real part for
$\varepsilon>0$, so that
$\sqrt{\mathfrak p-i\varepsilon}\to-i\sqrt{|\mathfrak p|}$.
In both cases, multiplication by $\widehat u(\xi,\tau)$
gives the selected solution of \eqref{Bah}.
Since $\sqrt{\mathfrak p-i0}=-i|\mathfrak p|^{1/2}$ for $\mathfrak p<0$, this prescription gives precisely the
phase $\theta=-i$ in \eqref{phase} below. 
Note that $\mathfrak p(\xi,\tau)$ is the symbol of the modified Schr\"odinger operator
\begin{equation}\label{M}
\M = - i \p_t - \Delta = e^{-i\frac{\pi}2} (\p_t - i \Delta) = e^{-i\frac{\pi}2} \Le,
\end{equation}
and therefore
\begin{equation}\label{spec}
\widehat{\M u}(\xi,\tau) = \mathfrak p(\xi,\tau) \hat u(\xi,\tau).
\end{equation}

With $L^+, L^-$ as in \eqref{Lmasomeno}, we now introduce the following phase function defined on $L^+\cup L^-$:
\begin{equation}\label{phase}
\theta(\xi,\tau) := e^{-i \frac{\pi}4(1-\operatorname{sgn}(\mathfrak p(\xi,\tau)))}.
\end{equation}
Notice that 
\[
\theta(\xi,\tau) = \begin{cases}
1,\ \ \ \ \ \ \ \ \ \text{for}\ (\xi,\tau)\in L^+,
\\
-i,\ \ \ \ \ \ \  \text{when}\ (\xi,\tau)\in L^-.
\end{cases}
\]
By
$K_s$ we indicate the Bessel function of the third kind \eqref{Knu}.

\begin{theorem}\label{T:itstheFT}
The solution $U$ of the extension problem \eqref{ep} selected by the preceding decay and limiting-absorption prescription admits the following representation in phase space
\begin{equation}\label{supernicehatV0}
\widehat U(\xi,z,\tau)  = \frac{2^{1-s}}{ \G(s)} \left(z \theta(\xi,\tau)  |\mathfrak p(\xi,\tau)|^{1/2}\right)^s K_{s}\left(z \theta(\xi,\tau)  |\mathfrak p(\xi,\tau)|^{1/2}\right) \hat u(\xi,\tau)\ \ \ (\xi,\tau)\in L^+\cup L^-, 
\end{equation}
in the sense that one has
\begin{equation}\label{pointwise}
\underset{z\to 0^+}{\lim}\ \widehat U(\xi,z,\tau) = \widehat u(\xi,\tau).
\end{equation}
Moreover, $U$ satisfies the Dirichlet-to-Neumann identity:
\begin{equation}\label{onL+0}
- \frac{2^{2s-1} \G(s)e^{i\frac{\pi s}2} }{\G(1-s)}\ \underset{z\to 0^+}{\lim} z^{1-2s} \widehat U_z(\xi,z,\tau) = (i\,\mathfrak p(\xi,\tau))^s\,\widehat u(\xi,\tau) = \widehat{\Le^s u}(\xi,\tau). 
\end{equation}
\end{theorem}

\begin{proof}
To solve \eqref{Bah}, we distinguish the cases $(\xi,\tau)\in L^+$ and $(\xi,\tau)\in L^-$. 

\vskip 0.2in

\noindent \underline{Case $(\xi,\tau)\in L^+$}: Keeping in mind that $a=1-2s$ we rewrite \eqref{Bah} as 
\begin{equation}\label{ppav}
z^2 \widehat U_{zz} + (1-2s) z \widehat U_{z} - z^2 |\mathfrak p(\xi,\tau)| \widehat U = 0,\ \ \ \ \ \widehat U(\xi,0,\tau) = \hat u(\xi,\tau).
\end{equation}
This is a modified Bessel equation such as \eqref{modgenbessel}, in which we have taken:
\begin{equation}\label{alfabeta}
\alpha = s,\ \ \gamma = 1,\ \ \nu =  \pm s,\ \ \ \beta^2 = |\mathfrak p(\xi,\tau)|>0.
\end{equation}  
The general solution to \eqref{ppav} is given by a linear combination of the  functions $I_s$ and $K_s$, modified according to \eqref{besselcv}. We thus find  
\begin{equation}\label{notsonicehat}
\widehat U(\xi,z,\tau) = z^{s} \left[A(\xi,\tau) I_{s}(z |\mathfrak p(\xi,\tau)|^{1/2})+ B(\xi,\tau) K_{s}(z |\mathfrak p(\xi,\tau)|^{1/2})\right].
\end{equation}
By the selection principle described above, we have
$A(\xi,\tau) = 0,$
 and therefore  
\begin{align}\label{nicerhatV}
\widehat U(\xi,z,\tau) & = B(\xi,\tau) z^{s} K_{s}\left(z |\mathfrak p(\xi,\tau)|^{1/2}\right). 
\end{align}
If $A(\xi,\tau)\ne0$, the asymptotic formulas
\eqref{ab} and \eqref{abK} give
$|\widehat U(\xi,z,\tau)|\to\infty$ exponentially as
$z\to\infty$, contrary to the decay condition.
From \eqref{nicerhatV} and from the asymptotic relation \eqref{zeroKnu}, we know that
\[
\widehat U(\xi,z,\tau)\ \underset{z\to 0^+}{\longrightarrow}\  B(\xi,\tau) |\mathfrak p(\xi,\tau)|^{-\frac{s}2} \frac{\G(s)}{2^{1-s}}.
\]
On the other hand, the Dirichlet condition $\widehat U(\xi,z,\tau) \underset{z\to 0^+}{\longrightarrow} \hat u(\xi,\tau)$ imposes at every $(\xi,\tau) \in L^+$ that
\[
B(\xi,\tau) = \frac{2^{1-s}}{ \G(s)}|\mathfrak p(\xi,\tau)|^{\frac{s}2} \hat u(\xi,\tau).
\]
We thus obtain from \eqref{nicerhatV}:
\begin{equation}\label{supernicehatV}
\widehat U(\xi,z,\tau)  = \frac{2^{1-s}}{ \G(s)} M_{s}\left(z |\mathfrak p(\xi,\tau)|^{1/2}\right) \hat u(\xi,\tau),
\end{equation}
where we have let $M_s(w) = w^{s} K_{s}(w)$. This proves \eqref{supernicehatV0} when $(\xi,\tau)\in L^+$.

We next want to understand the limit as $z\to 0^+$ of $z^{1-2s} \widehat U_z(\xi,z,\tau)$. 
Since \eqref{recuu} gives
\begin{equation*}
M_s'(w) = - w^{2s-1} M_{1-s}(w),
\end{equation*}
we obtain from \eqref{supernicehatV}:
\begin{align}\label{DtoN}
z^{1-2s} \widehat U_z(\xi,z,\tau) & = \frac{2^{1-s}}{ \G(s)}z^{1-2s} |\mathfrak p(\xi,\tau)|^{1/2}  M'_{s}\left(z |\mathfrak p(\xi,\tau)|^{1/2}\right) \hat u(\xi,\tau)
\\
& = -  \frac{2^{1-s} }{ \G(s)} z^{1-2s}  |\mathfrak p(\xi,\tau)|^{1/2} (|\mathfrak p(\xi,\tau)|^{1/2}z)^{2s-1} M_{1-s}(z |\mathfrak p(\xi,\tau)|^{1/2})\hat u(\xi,\tau)
\notag\\
& = -  \frac{2^{1-s}}{ \G(s)} |\mathfrak p(\xi,\tau)|^{s}  M_{1-s}(z |\mathfrak p(\xi,\tau)|^{1/2})\hat u(\xi,\tau).
\notag
\end{align}
Keeping in mind that \eqref{zeroKnu} gives
\[
M_{1-s}(z |\mathfrak p(\xi,\tau)|^{1/2})\ \underset{z\to 0^+}{\longrightarrow}\ 2^{-s}\G(1-s),
\]
we obtain from \eqref{DtoN}
\begin{equation}\label{DtoN2}
z^{1-2s} \widehat U_z(\xi,z,\tau)\ \underset{z\to 0^+}{\longrightarrow}\ -  \frac{2^{1-2s}\G(1-s)}{ \G(s)} |\mathfrak p(\xi,\tau)|^{s}\ \hat u(\xi,\tau).
\end{equation}
Comparing the right-hand side of \eqref{DtoN2} with \eqref{fp1} in Proposition \ref{P:FTfp}, we conclude that for any $(\xi,\tau)\in L^+$ we have
\begin{equation}\label{DtoN3}
z^{1-2s} \widehat U_z(\xi,z,\tau)\ \underset{z\to 0^+}{\longrightarrow}\ -  \frac{2^{1-2s}\G(1-s)}{ \G(s)} e^{-i\frac{\pi s}2} \widehat{\Le^s u}(\xi,\tau)
\end{equation}
Equivalently, we can write \eqref{DtoN3}
\begin{equation}\label{onL+}
- \frac{ \G(s)e^{i\frac{\pi s}2} }{2^{1-2s}\G(1-s)}\ \underset{z\to 0^+}{\lim} z^{1-2s} \widehat U_z(\xi,z,\tau) = \widehat{\Le^s u}(\xi,\tau), 
\end{equation}
which proves \eqref{onL+0}.

\vskip 0.2in

\noindent \underline{Case $(\xi,\tau)\in L^-$}: 
Here $\mathfrak p(\xi,\tau)<0$, and we use
\eqref{E:limiting-absorption}.
Since
\[
\sqrt{\mathfrak p(\xi,\tau)-i0}=-i|\mathfrak p(\xi,\tau)|^{1/2},
\]
it gives directly
\begin{equation}\label{oplada}
\widehat U(\xi,z,\tau)
=
\frac{2^{1-s}}{\G(s)}
\bigl(-iz|\mathfrak p(\xi,\tau)|^{1/2}\bigr)^s
K_s\bigl(-iz|\mathfrak p(\xi,\tau)|^{1/2}\bigr)
\widehat u(\xi,\tau).
\end{equation}
This is exactly \eqref{supernicehatV0} on $L^-$, because $\theta=-i$ there.
The small-argument asymptotic
\[
w^sK_s(w)\longrightarrow 2^{s-1}\G(s),
\qquad w\to0,
\]
along the ray $\arg w=-\pi/2$ shows at once that
$\widehat U(\xi,z,\tau)\to\widehat u(\xi,\tau)$ as $z\to0^+$.

It remains to compute the Neumann trace. Using
\[
\frac{d}{dw}\bigl(w^sK_s(w)\bigr)=-w^sK_{s-1}(w)
=-w^sK_{1-s}(w)
\]
and
\[
w^{1-s}K_{1-s}(w)\longrightarrow 2^{-s}\G(1-s),
\qquad w\to0,
\]
we obtain, with $w=-iz|\mathfrak p|^{1/2}$,
\[
\lim_{z\to0^+}z^{1-2s}\partial_z\widehat U(\xi,z,\tau)
=-\frac{2^{1-2s}\G(1-s)}{\G(s)}
(-i)^{2s}|\mathfrak p(\xi,\tau)|^s\widehat u(\xi,\tau).
\]
Since $(-i)^{2s}=e^{-i\pi s}$ and, on $L^-$,
\[
\widehat{\Le^su}(\xi,\tau)=e^{-i\frac{\pi s}2}|\mathfrak p(\xi,\tau)|^s\widehat u(\xi,\tau),
\]
it follows that
\[
-\frac{\G(s)e^{i\frac{\pi s}2}}{2^{1-2s}\G(1-s)}
\lim_{z\to0^+}z^{1-2s}\partial_z\widehat U(\xi,z,\tau)
=\widehat{\Le^su}(\xi,\tau).
\]
This proves \eqref{supernicehatV0} and \eqref{onL+0} on $L^-$.

\end{proof}

\section{The Poisson kernel appears} 
\label{S:appears}

We compute the inverse Fourier transform of the multiplier in
\eqref{supernicehatV0} to obtain the Poisson kernel.
We will need the following auxiliary oscillatory lemma.

\begin{lemma}\label{L:critical}
Let $0<s<1$, $z>0$, and $\mathfrak p\in\mathbb{R}\setminus\{0\}$. Then
\[
I_s(z,\mathfrak p)
:=
\int_0^\infty \frac{e^{i\frac{z^2}{4\sigma}}}{\sigma^{1+s}}
e^{-i\mathfrak p\sigma}\,d\sigma = \begin{cases}
2^{s+1}z^{-s} e^{i \frac{\pi s}2} \mathfrak p^{\frac s2} K_s(z \mathfrak p^{1/2}),\ \ \ \ \ \ \ \ \ \ \mathfrak p>0,
\\
2^{s+1}z^{-s}  |\mathfrak p|^{\frac s2} K_s(z e^{-i \frac{\pi }2}|\mathfrak p|^{1/2}),\ \ \ \ \ \ \mathfrak p<0.
\end{cases}
\]
\end{lemma}

\begin{proof} 

The change of variable $x = \frac{z^2}{4\sigma}$ gives
\[
I_s(z,\mathfrak p) = 2^{2s}z^{-2s} \int_0^\infty x^{s-1} e^{ix} e^{-i\frac{\mathfrak p z^2}{4x}} dx.
\]
Before evaluating this integral, we justify the regularization
that we shall use. The integral over $(0,1)$ is absolutely
convergent because $s>0$. For the tail, integration by parts
against $e^{(i-\ve)x}$ gives, for $R\ge1$,
\[
\left|\int_R^\infty x^{s-1}
e^{(i-\ve)x}e^{-\frac{\delta+i\mathfrak p z^2/4}x}\,dx\right|
\le C_{s,z,\mathfrak p}R^{s-1},
\qquad 0\le\delta,\ve\le1.
\]
Indeed, the amplitude $x^{s-1}e^{-(\delta+i\mathfrak p z^2/4)/x}$
tends to zero at infinity, and the absolute value of its
derivative is bounded by
$(1-s)x^{s-2}+(1+|\mathfrak p|z^2/4)x^{s-3}$.
The tail therefore tends to zero uniformly in the two
regularization parameters. Dominated convergence on bounded
intervals now yields
\[
I_s(z,\mathfrak p)
=2^{2s}z^{-2s}\lim_{\delta,\ve\downarrow0}
\int_0^\infty x^{s-1} e^{-\frac{\beta}{x} - \gamma x} dx,
\qquad
\beta=\delta+i\frac{\mathfrak p z^2}{4},\quad
\gamma=\ve-i.
\]
Since $\Re\beta=\delta>0$ and $\Re\gamma=\varepsilon>0$, we may apply Schl\"afli's identity, see \cite[9., 3.471]{GR}
\begin{equation}\label{reg}
\int_0^\infty x^{s-1} e^{-\frac{\beta}{x} - \gamma x} dx = 2 \left(\frac{\beta}{\gamma}\right)^{\frac s2} K_s(2\sqrt{\beta \gamma}),
\end{equation}
to find
\begin{equation}\label{Isp}
I_s(z,\mathfrak p) =  2^{2s}z^{-2s} \underset{\delta, \ve\to 0^+}{\lim}\ 2 \left(\frac{\beta}{\gamma}\right)^{\frac s2} K_s(2\sqrt{\beta \gamma})
\end{equation}
Observe that as $\delta,\ve \to 0^+$
\[
\beta=\delta+i\frac{\mathfrak p z^2}{4}\longrightarrow i\frac{\mathfrak p z^2}{4},
\qquad
\gamma=\ve-i\longrightarrow -i.
\]
Hence
\[
\frac{\beta}{\gamma}
\longrightarrow
-\frac{\mathfrak p z^2}{4},
\qquad
\beta\gamma
\longrightarrow
\frac{\mathfrak p z^2}{4}.
\]
Since the complex powers and square roots are taken with respect to the principal branch, we find 
\begin{align*}
& \frac \beta{\gamma} = \left(-\frac{\mathfrak p z^2}{4(1+\ve^2)} + \frac{\delta \ve}{1+\ve^2}\right) + i \left(\frac{\delta}{1+\ve^2} + \frac{\ve \mathfrak p z^2}{4(1+\ve^2)}\right)\ \longrightarrow\ \begin{cases} e^{i\pi} \frac{|\mathfrak p| z^2}4,\ \ \ \ \text{if}\ \mathfrak p>0,
\\
\frac{|\mathfrak p| z^2}4,\ \ \ \ \ \ \ \ \text{if}\ \mathfrak p<0,
\end{cases}
\\
& \beta \gamma = \left(\frac{\mathfrak p z^2}{4} +\delta \ve\right) + i \left(\ve \frac{\mathfrak p z^2}{4}- \delta\right)\ \longrightarrow\ \begin{cases}  \frac{|\mathfrak p| z^2}4,\ \ \ \ \ \ \ \ \ \ \text{if}\ \mathfrak p>0,
\\
e^{-i\pi} \frac{|\mathfrak p| z^2}4,\ \ \ \ \ \text{if}\ \mathfrak p<0.
\end{cases}
\end{align*}
We thus obtain from \eqref{Isp}
\[
I_s(z,\mathfrak p) = \begin{cases}
2^{s+1}z^{-s} e^{i \frac{\pi s}2} \mathfrak p^{\frac s2} K_s(z \sqrt{\mathfrak p}),\ \ \ \ \ \ \ \ \ \ \mathfrak p>0,
\\
2^{s+1}z^{-s}  |\mathfrak p|^{\frac s2} K_s(z e^{-i \frac{\pi }2}\sqrt{|\mathfrak p|}),\ \ \ \ \ \mathfrak p<0.
\end{cases}
\]
This gives the desired conclusion.

\end{proof}

\begin{remark}\label{R:critical}
Lemma~\ref{L:critical} may also be read in the sense of tempered
distributions in the variable $\mathfrak p\in \R$.
Indeed, for $\varepsilon,\delta>0$, define
\[
I_{s,\varepsilon,\delta}(z,\mathfrak p)
:= 2^{2s}z^{-2s}\int_0^\infty x^{s-1} e^{-\frac{\beta}x - \gamma x} dx,
\]
where $\beta = \beta(\delta), \gamma = \gamma(\ve)$ are as before.
Then $I_{s,\varepsilon,\delta}(z,\cdot)\in\mathscr S'(\mathbb R)$, and
the family converges in $\mathscr S'(\mathbb R)$, as
$\varepsilon,\delta\to0^+$, to the oscillatory kernel $I_s(z,\mathfrak p)$.
Consequently, the identity in Lemma~\ref{L:critical} holds as an identity
in $\mathscr S'(\mathbb R)$, with the pointwise formula for $\mathfrak p\neq0$
as its representative.
\end{remark}

\begin{theorem}\label{T:invFT}
For $z>0$ and $(\xi,\tau)\in L^+\cup L^-$, with $\theta(\xi,\tau)$ as in \eqref{phase}, let
\begin{equation}\label{ms}
m_s(z;\xi,\tau) := \frac{2^{1-s}}{ \G(s)} \left(z \theta(\xi,\tau)  |\mathfrak p(\xi,\tau)|^{1/2}\right)^s K_{s}\left(z \theta(\xi,\tau)  |\mathfrak p(\xi,\tau)|^{1/2}\right).
\end{equation}
Then, setting $a = 1-2s$, we have in $\mathscr S'(\R^{d+1})$
\begin{equation}\label{invFTFs}
\mathscr F^{-1}_{(\xi,\tau)\to (x,t)}(m_s(z;\cdot))
=
P_a(z,t)\,S(x,t),
\end{equation}
where  
\begin{equation}\label{Poi}
P_a(z,t)
=
\frac{e^{\,i\frac{(a-1)\pi}{4}}}{2^{1-a}\Gamma\!\left(\frac{1-a}{2}\right)}
\frac{z^{1-a}}{t^{\frac{3-a}{2}}}
e^{i\frac{z^2}{4t}}\mathbf 1_{(0,\infty)}(t).
\end{equation}
\end{theorem}

\begin{proof}
In order to compute $\mathscr F^{-1}(m_s(z;\cdot))$ in $\mathscr S'(\RNu)$, given $\Phi\in\mathscr S(\mathbb R^{d+1})$, we use the duality 
\[
\langle \mathscr F^{-1}(m_s(z;\cdot)),\Phi\rangle
=
\langle m_s(z;\cdot),\mathscr  F^{-1}\Phi\rangle.
\]
Thanks to Lemma~\ref{L:critical} and Remark \ref{R:critical}, the multiplier admits the oscillatory representation
\[
m_s(z;\xi,\tau)
= \frac{e^{-i\frac{\pi s}2}}{4^s\Gamma(s)} z^{2s}
\int_0^\infty
\frac{1}{t^{1+s}}
e^{i \frac{z^2}{4t}}
e^{-i t \mathfrak p(\xi,\tau)}\,dt,
\]
in the sense of $\mathscr S'(\RNu)$.
Let now $\eta\in C^\infty(0,\infty)$ be such that $0\le\eta\le1$, $\eta\equiv0$ on $(0,1]$ and $\eta\equiv1$ on $[2,\infty)$, and for any $R>1$ define $\chi_R(t) = \eta(Rt)\bigl(1-\eta(\frac tR)\bigr)$. Then $0\le\chi_R\le1$, $\chi_R\equiv 1$ on $[\frac2R,R]$ and $\operatorname{supp}\chi_R\subseteq[\frac1R,2R]$; in particular, for every fixed $t>0$ we have $\chi_R(t)\to 1$ as $R\to \infty$, and $\chi_R(t) = 0$ for $t \le \frac{1}{R}$ or $t\ge 2R$. Define
\[
m_{s,R}(z;\xi,\tau)
=
\frac{e^{-i\frac{\pi s}2}}{4^s\Gamma(s)} z^{2s}
\int_0^\infty
\frac{\chi_R(t)}{
t^{1+s}}
e^{i \frac{z^2}{4t}}
e^{-i t \mathfrak p(\xi,\tau)}\,dt.
\]
To estimate $m_{s,R}(z;\xi,\tau)$ uniformly in $R>1$,
we split its defining integral over $(0,\infty)$ into
the integrals over $(0,1)$ and $(1,\infty)$.
We use integration by parts for the first and an absolute
estimate for the second. Fix $z>0$ and put $b=z^2/4$.
On $(0,1)$, integration by parts using
$e^{ib/t}=-(t^2/(ib))\partial_t e^{ib/t}$ gives
\[
\left|\int_0^1\chi_R(t)t^{-1-s}e^{ib/t-it\mathfrak p}\,dt\right|
\le C_{s,z}\left(1+
\int_0^1|\chi_R'(t)|t^{1-s}\,dt+
\int_0^1t^{-s}\,dt+|\mathfrak p|\int_0^1t^{1-s}\,dt\right).
\]
The cutoff derivative term is bounded uniformly for $R>1$: its support
in $(0,1)$ lies in $[1/R,2/R]$, and its contribution is at most
$C R^{s-1}\le C$. The integral over $[1,\infty)$ is absolutely bounded
by $\int_1^\infty t^{-1-s}dt$. Consequently
\[
|m_{s,R}(z;\xi,\tau)|\le C_{s,z}(1+|\mathfrak p(\xi,\tau)|),\qquad R>1.
\]
The same integration by parts at zero, and absolute convergence at infinity,
show pointwise convergence to the oscillatory integral in Lemma~\ref{L:critical}
for $\mathfrak p\ne0$. Its value is $m_s(z;\xi,\tau)$ after multiplication by the
prefactor above. The polynomial majorant and dominated convergence against
Schwartz test functions therefore give
\[
m_{s,R}(z;\cdot)\ \underset{R\to\infty}{\longrightarrow}\
m_s(z;\cdot)\qquad\hbox{in }\mathscr S'(\RNu).
\]
Since for each fixed $R$, the integral defining $m_{s,R}(z;\cdot)$ is absolutely convergent, we can exchange order of integration. Hence,
\begin{equation}\label{reg-cutoff}
\langle \mathscr F^{-1}(m_{s,R}(z;\cdot)),\Phi\rangle
=
\frac{e^{-i\frac{\pi s}2}}{4^s\Gamma(s)} z^{2s}
\int_0^\infty
\frac{\chi_R(t)}{
t^{1+s}}
e^{i \frac{z^2}{4t}}
\langle e^{-i t \mathfrak p(\xi,\tau)},\mathscr F^{-1}(\Phi)\rangle
\,dt.
\end{equation}
We now claim that for any fixed $t>0$ we have
\begin{equation}\label{beauty}
\langle e^{-i t \mathfrak p(\xi,\tau)},\mathscr F^{-1}(\Phi)\rangle = (4\pi i t)^{-\frac d2} \int_{\Rd} e^{i\frac{|x|^2}{4t}} \Phi(x,t) dx.
\end{equation}
To check the claim, write
\[
e^{-i t \mathfrak p(\xi,\tau)}
=
e^{-2\pi i t\tau}\,e^{-4\pi^2 i t |\xi|^2}.
\]
First using Fourier inversion in the time variable, we find
\[
\begin{aligned}
\langle e^{-i t \mathfrak p(\xi,\tau)},\mathscr F^{-1}(\Phi)\rangle
&=\int_{\Rd}e^{-4\pi^2 i t|\xi|^2}
\left[\int_{\mathbb R}e^{-2\pi i t\tau}
(\mathscr F^{-1}\Phi)(\xi,\tau)\,d\tau\right]d\xi\\
&=\int_{\Rd}e^{-4\pi^2 i t|\xi|^2}
\bigl(\mathscr F^{-1}_{x\to\xi}\Phi(\cdot,t)\bigr)(\xi)\,d\xi.
\end{aligned}
\]
Hence, denoting by $\widecheck{\Phi}(x,t) = \Phi(-x,t)$, and using the standard formulas
\[
\mathscr F(\Phi) = \mathscr F^{-1}(\widecheck \Phi),\ \ \ \ \ \ \mathscr F^{-1}(\Phi) = \mathscr F(\widecheck \Phi), 
\]
we infer
\begin{align*}
\langle e^{-i t \mathfrak p(\xi,\tau)},\mathscr F^{-1}(\Phi)\rangle & = \int_{\Rd} e^{-4\pi^2 i t |\xi|^2}
\mathscr F_{x\to\xi}(\widecheck \Phi(\cdot,t))(\xi)\,d\xi
 = \int_{\Rd} 
\mathscr F_{x\to\xi}(S(t)\widecheck \Phi(\cdot,t))(\xi)\,d\xi
\\
& = \sa 1,\mathscr F_{x\to\xi}(S(t)\widecheck \Phi(\cdot,t))\da = \sa \delta, (S(t)\widecheck \Phi(\cdot,t))^{\widecheck{}}\ \da
\\
& = (4\pi i t)^{-\frac d2} \int_{\Rd} e^{i\frac{|x|^2}{4t}} \Phi(x,t) dx,
\end{align*}
which proves the claim \eqref{beauty}. Substituting \eqref{beauty} in \eqref{reg-cutoff}, we find
\begin{equation}\label{reg2}
\begin{aligned}
\langle \mathscr F^{-1}(m_{s,R}),\Phi\rangle
& =
\frac{e^{-i\frac{\pi s}2}}{4^s\Gamma(s)} z^{2s}
\int_\R \int_{\Rd}
\mathbf 1_{(0,\infty)}(t)\frac{\chi_R(t)}{
t^{1+s}}
e^{i \frac{z^2}{4t}}
(4\pi i t)^{-\frac d2}  e^{i\frac{|x|^2}{4t}} \Phi(x,t) dx
\,dt
\\
& = \frac{e^{-i\frac{\pi s}2}}{4^s\Gamma(s)} z^{2s}
\sa \mathbf 1_{(0,\infty)}\frac{\chi_R}{
t^{1+s}}
e^{i \frac{z^2}{4t}} S(x,t),\Phi \da.
\end{aligned}
\end{equation}
By continuity of the inverse Fourier transform on
$\mathscr S'(\RNu)$, equation \eqref{reg2} gives
\[
\lim_{R\to\infty}
\frac{\mathbf 1_{(0,\infty)}(t)\chi_R(t)}{t^{1+s}}
e^{iz^2/(4t)}S(x,t)
=
\frac{4^s\Gamma(s)e^{i\pi s/2}}{z^{2s}}\,
\mathscr F^{-1}(m_s(z;\cdot))
\quad\text{in }\mathscr S'(\RNu).
\]
We use this limit to define
\[
\frac{\mathbf 1_{(0,\infty)}(t)}{t^{1+s}}
e^{iz^2/(4t)}S(x,t)
\]
as a tempered distribution. For $t>0$ it agrees with the
displayed function, and for $t<0$ it vanishes.
Multiplying by $e^{-i\pi s/2}z^{2s}/(4^s\Gamma(s))$
and setting $a=1-2s$, we obtain \eqref{invFTFs}
with $P_a$ given by \eqref{Poi}.

\end{proof}

With $X=(x,z)$ and $Y_0=(y,0)$, as in Section~\ref{S:not},
we define the Poisson kernel as follows. 

\begin{definition}\label{D:Pa}
We define the \emph{Poisson kernel} associated with the extension problem  \eqref{expbS} as the function
\begin{equation}\label{Pa}
\mathbb P_a(X,Y_0,t) := P_a(z,t) S(x,y,t)
= \frac{e^{i\frac{(a-1)\pi}{4}}}
     {2^{1-a}\Gamma\!\left(\frac{1-a}{2}\right)}
\frac{z^{1-a}}{t^{\frac{3-a}{2}}}
e^{i\frac{z^2}{4t}}\ \mathbf 1_{(0,\infty)}(t)\ S(x,y,t).
\end{equation}
\end{definition}

Combining Theorems \ref{T:itstheFT} and \ref{T:invFT}, we obtain the representation formula
\begin{equation}\label{U2}
U(x,z,t) = \int_0^\infty \int_{\Rd} \mathbb P_a(X,Y_0,\sigma) u(y,t-\sigma) dy d\sigma.
\end{equation}

\begin{lemma}\label{L:Paone}
Assume that $-1<a<1$. Then for every $z>0$ the function in \eqref{Poi} satisfies the following identity as an improper integral:
\[
\int_0^\infty P_a(z,t) dt
=
1.
\]
\end{lemma}

\begin{proof}
Set $s=\frac{1-a}{2}$,
$b=\frac{z^2}{4}$.
Then $\frac{3-a}{2}=s+1$,
and in view of \eqref{Poi} the integral can be written as
\begin{equation}\label{Paint}
\int_0^\infty P_a(z,t) dt =
\frac{e^{-i\frac{\pi s}{2}}}
     {2^{2s}\Gamma(s)}
z^{2s}
\int_0^\infty t^{-s-1}e^{i \frac{b}t}\,dt.
\end{equation}
By the change of variable
$u=\frac{b}{t}$ we find
\begin{align*}
\int_0^\infty t^{-s-1}e^{i\frac{b}t}\,dt
&=
\frac{2^{2s}}{z^{2s}}
\int_0^\infty u^{s-1}e^{iu}\,du.
\end{align*}
Since $0<\Re(s)<1$, formula \eqref{ei} in Lemma \ref{L:cis} yields
\[
\int_0^\infty u^{s-1}e^{iu}\,du
=
e^{i\frac{\pi s}2}\Gamma(s).
\]
Substituting these identities in \eqref{Paint} we reach the desired conclusion.

\end{proof}

From Lemma \ref{L:Paone} and tensorization we obtain the following.
 
\begin{corollary}\label{C:Paone}
We have the following oscillatory integral
\begin{equation}\label{intPa}
\int_{\RNu} \mathbb P_a(X,Y_0,t) dy dt = 1,\ \ \ \ \ \ \ \ \ \forall X\in \RN.
\end{equation} 
\end{corollary}

The following consequence of Lemma~\ref{L:Abel} will be used in Section~\ref{P:liftingEs}.

\begin{lemma}\label{L:abepoi}
For $0<s<1, a = 1-2s, t>0$, and $z>0$ we have:
\begin{equation}\label{abepoi}
\frac{1}{\G(s)} \int_0^\infty \frac{\mathbf 1_{(0,\infty)}(t-\sigma)}{(t-\sigma)^{1-s}} P_a(z,\sigma) d\sigma = \frac{1}{\G(s)} \frac{e^{i \frac{z^2}{4t}}}{t^{1-s}}.
\end{equation}
\end{lemma}

\begin{proof} Keeping \eqref{Poi}, $a = 1-2s$ and Lemma~\ref{L:Abel} in mind, we have
\begin{align*}
& \frac{1}{\G(s)} \int_0^\infty \frac{\mathbf 1_{(0,\infty)}(t-\sigma)}{(t-\sigma)^{1-s}} P_a(z,\sigma) d\sigma = \frac{e^{-i\frac{\pi s}{2}}}{4^s\G(s)^2}
z^{2s} \int_0^t 
\frac{e^{i\frac{z^2}{4\sigma}}}{(t-\sigma)^{1-s} \sigma^{1+s}} d\sigma
\\
& = \frac{e^{-i\frac{\pi s}{2}}}{4^s\G(s)^2}
z^{2s}  4^s\,\G(s)\,e^{i\frac{\pi s}{2}}\,t^{s-1}\,z^{-2s}
e^{i \frac{z^2}{4t}} = \frac{1}{\G(s)} \frac{e^{i \frac{z^2}{4t}}}{t^{1-s}},
\end{align*}
which proves the lemma.

\end{proof}


\vskip 0.2in

\subsection{Poisson kernel and Bessel intertwining}\label{S:ep}

The intertwining identities for Bessel operators provide another
construction of the Poisson kernel; see \cite{We48, We53, We2,MS} and
\cite[Sections~2.2-2.3]{CS}. This construction uses the propagator
of \cite{GS1,GS2} for the homogeneous Neumann problem
\begin{equation}\label{eppn}
\begin{cases}
\p_t U - i \left(\Delta_x U+\mathscr B_z^{(a)} U\right) = 0,\ \ \ (X,t)\in \R_+^{d+1}\times \R
\\
\underset{z\to 0^+}{\lim} z^a \p_z U(X,t) = 0,\ \  \ \ U(X,0) = U_0(X).
\end{cases}
\end{equation}
The spatial
operator $\Delta_x +\mathscr B_z^{(a)}$ is referred to as the \emph{Weinstein operator} of generalized axially symmetric potential theory; see \cite{AO, GV} and the references therein.  Now, the fundamental solution of  \eqref{eppn} was constructed in \cite{GS2} and it is given by the kernel:
\begin{align}\label{grandeSaS0}
\mathbb S_a(X,Y,t)
= \frac{e^{-i\,\operatorname{sgn}(t)\,\frac{(d+a+1)\pi}{4}}}{2^{\frac{2d+a+1}{2}}\pi^{\frac d2}}  
\frac{e^{i\,\frac{z^2+\zeta^2+|x-y|^2}{4t}}}{|t|^{\frac{d+a+1}{2}}}\left(\frac{z\zeta}{2|t|}\right)^{\frac{1-a}{2}}
J_{\frac{a-1}{2}}\!\left(\frac{z\zeta}{2|t|}\right),
\qquad t\neq 0.
\end{align}
We note explicitly the tensorization property
\begin{equation}\label{Sa}
\mathbb S_a(X,Y,t) =  S_a(z,\zeta,t) S(x,y,t),
\end{equation}
where $S(x,y,t)$ is the classical propagator \eqref{Scla}, whereas we have denoted by 
\begin{equation}\label{SaS}
S_a(z,\zeta,t)
=
\frac{e^{-i\,\operatorname{sgn}(t)\,\frac{(a+1)\pi}{4}}}{(2|t|)^{\frac{a+1}{2}}}
\left(\frac{z\zeta}{2|t|}\right)^{\frac{1-a}{2}}
J_{\frac{a-1}{2}}\!\left(\frac{z\zeta}{2|t|}\right)
\,e^{i \frac{z^2+\zeta^2}{4t}}, \quad t \neq 0,
\end{equation}
the fundamental solution, constructed in \cite{GS1},  of the unitary group $e^{i t \Ba}$ on the measure space $(\R^+_z, z^a dz)$, with zero Neumann condition.
Note that \eqref{Sa} implies the Neumann condition
\begin{equation}\label{Sazeron}
\underset{z\to 0^+}{\lim} z^a \p_z \mathbb S_a(X,Y,t) = 0.
\end{equation}

We first compute the boundary value of $S_{-a}$.

\begin{lemma}\label{L:S-a}
One has for $t\not= 0$
\[
S_{-a}(z,0,t) 
=
\frac{2^a\,|t|^{\frac{a-1}{2}}}
     {\Gamma\!\left(\frac{1-a}{2}\right)}
\,e^{-i\,\operatorname{sgn}(t)\,\frac{(1-a)\pi}{4}}
\,e^{i\frac{z^2}{4t}}.
\]
\end{lemma}

\begin{proof}
From \eqref{SaS} we have
\[
S_{-a}(z,\zeta,t)
=
\frac{e^{-i\,\operatorname{sgn}(t)\,\frac{(1-a)\pi}{4}}}
     {(2|t|)^{\frac{1-a}{2}}}
\left(\frac{z\zeta}{2|t|}\right)^{\frac{a+1}{2}}
J_{-\frac{a+1}{2}}\!\left(\frac{z\zeta}{2|t|}\right)
e^{i\frac{z^{2}+\zeta^{2}}{4t}},
\qquad t\neq 0.
\]
The asymptotic of the Bessel function,
\[
J_\nu(x)
=
\frac{(x/2)^\nu}{\Gamma(\nu+1)}
+o\!\left(x^\nu\right),
\qquad x\to0,
\]
with \(\nu=-\frac{a+1}{2}\), immediately gives
\[
\lim_{\zeta\to0^+}
\left(\frac{z\zeta}{2|t|}\right)^{\frac{a+1}{2}}
J_{-\frac{a+1}{2}}
\!\left(\frac{z\zeta}{2|t|}\right)
=
\frac{2^{\frac{a+1}{2}}}
     {\Gamma\!\left(\frac{1-a}{2}\right)},
\]
and therefore the desired conclusion follows.

\end{proof}

\begin{lemma}\label{L:Smenoa}
Let \(-1<a<1\) and \(t>0\). Then the boundary kernel in \eqref{Poi} satisfies 
\[
P_{a}(z,t) = -i\,z^{-a}\,\partial_z S_{-a}(z,0,t)
=
\frac{2^{a-1}\operatorname{sgn}(t)}{\Gamma\!\left(\frac{1-a}{2}\right)}
\,z^{1-a}\,|t|^{\frac{a-3}{2}}
\,e^{-i\,\operatorname{sgn}(t)\,\frac{(1-a)\pi}{4}}
\,e^{i\frac{z^2}{4t}}.
\]
\end{lemma}

\begin{proof}

Differentiate the expression for $S_{-a}(z,0,t)$ in
Lemma~\ref{L:S-a}.

\end{proof}

Tensoring the kernel in Lemma~\ref{L:Smenoa} with $S(x,y,t)$
gives $\mathbb P_a$ in \eqref{Pa}.
The following result reflects the intertwining properties of the Bessel operators.

\begin{lemma}\label{L:Pa-solves-equation}
For any fixed $Y_0\in \p \RN$, the smooth function $(X,t)\to \mathbb P_a(X,Y_0,t)$ solves the equation 
\[
\partial_t \mathbb P_a - i\bigl(\Delta_x \mathbb P_a+\mathscr B_z^{(a)}\mathbb P_a\bigr)=0
\]
in the bulk space $\RN\times (0,\infty)$.
\end{lemma}

\begin{proof}
From \eqref{Pa} we have
\[
\partial_t \mathbb P_a - i\bigl(\Delta_x \mathbb P_a+\mathscr B_z^{(a)}\mathbb P_a\bigr)
=
\bigl(\partial_t P_a-i\mathscr B_z^{(a)}P_a\bigr)S
+
P_a \bigl(\partial_t S-i\Delta_x S\bigr).
\]
Since
$\partial_t S-i\Delta_x S=0$,
 it remains to show that
\[
\partial_t P_a-i\mathscr B_z^{(a)}P_a=0.
\]
Set $C_a=\frac{e^{i\frac{(a-1)\pi}{4}}}{2^{1-a}\Gamma\!\left(\frac{1-a}{2}\right)}$, so that
$P_a(z,t)=C_a\,t^{-\frac{3-a}{2}}\,z^{1-a}e^{i\frac{z^2}{4t}}$.
A direct computation gives
\[
\partial_t P_a
=
P_a\left(-\frac{3-a}{2t}-i\frac{z^2}{4t^2}\right).
\]
Next,
\[
\mathscr B_z^{(a)}\!\left(z^{1-a}e^{i\frac{z^2}{4t}}\right)
=
\left(i\frac{3-a}{2t}-\frac{z^2}{4t^2}\right)
z^{1-a}e^{i\frac{z^2}{4t}},
\]
and it follows that
\[
\mathscr B_z^{(a)}P_a
=
P_a\left(i\frac{3-a}{2t}-\frac{z^2}{4t^2}\right).
\]
Hence
\[
i\,\mathscr B_z^{(a)}P_a
=
P_a\left(-\frac{3-a}{2t}-i\frac{z^2}{4t^2}\right)
=
\partial_t P_a.
\]
Therefore,
\[
\partial_t P_a-i\mathscr B_z^{(a)}P_a=0.
\]

\end{proof}



\section{Classical solvability for Schwartz data}\label{S:schwartz}

For Schwartz boundary data, we prove smoothness of the Poisson
extension, attainment of the boundary datum, and the
Dirichlet-to-Neumann identity.
Throughout this section we assume that $u\in\mathscr S(\RNu)$,
and with $m_s$ as in \eqref{ms}, we denote by
\[
\widehat U(\xi,z,\tau)
=
m_s(z;\xi,\tau)\widehat u(\xi,\tau)
\]
the oscillatory extension introduced in \eqref{supernicehatV0} of Theorem \ref{T:itstheFT}.
The rapid decay of $\widehat u$ allows one to differentiate under the
inverse Fourier transform and thereby recover a smooth solution of the
singular Schr\"odinger equation.

\begin{theorem}[Classical solvability]
\label{T:classical}

Let $u\in\mathscr S(\RNu)$.
Then the oscillatory extension satisfies
\[
U\in C^\infty(\RNu\times(0,\infty)).
\]
Moreover:
\begin{itemize}

\item[(1)] $U$ satisfies the extension equation
\[
\partial_tU - i(\Delta_x U+\Ba U) = 0
\]
pointwise in $\RNu\times(0,\infty)$;
\item[(2)] the boundary conditions
\[
U(\cdot,z,\cdot)\ 
\underset{z\to 0^+}{\longrightarrow}\ 
u,
\]
and
\[
-
\frac{\Gamma(s)}{2^{1-2s}\Gamma(1-s)}
e^{i\pi s/2}
 z^a\partial_zU\ 
\underset{z\to 0^+}{\longrightarrow}\ 
\Le^s u
\]
hold strongly in
$L^2(\mathbb R^{d+1})$.
\end{itemize}
\end{theorem}

\begin{proof}

To prove that $U$ is smooth for $z>0$, we differentiate its
Fourier representation. For a multi-index $\gamma$ and an
integer $\beta\ge0$, the integrand for
$\partial_{x,t}^{\gamma}\partial_z^\beta U$ is
\[
(2\pi i(\xi,\tau))^\gamma
\partial_z^\beta m_s(z;\xi,\tau)\,\widehat u(\xi,\tau).
\]
Since $\widehat u$ is rapidly decreasing, differentiation
under the integral is justified if every $\partial_z^\beta m_s$
has at most polynomial growth in $(\xi,\tau)$, uniformly
for $z$ in compact intervals of $(0,\infty)$.
We now establish these bounds.
For $\mathfrak p=\mathfrak p(\xi,\tau)\ne0$, set
\[
w=
\begin{cases}
z\sqrt{\mathfrak p}, & \mathfrak p>0,\\
-iz\sqrt{|\mathfrak p|}, & \mathfrak p<0.
\end{cases}
\]
Then
\[
m_s(z;\mathfrak p)=\frac{2^{1-s}}{\Gamma(s)}w^sK_s(w).
\]
Here the powers of $w$ are taken with argument $0$ when
$\mathfrak p>0$ and $-\pi/2$ when $\mathfrak p<0$. The differentiation formula
for $K_s$, together with $K_{s-1}=K_{1-s}$, gives
\[
\partial_zm_s(z;\mathfrak p)
=-\frac{2^{1-s}}{\Gamma(s)\,z}\,
w^{s+1}K_{1-s}(w).
\]
Fix a compact interval $I\subset(0,\infty)$.
For $\nu>0$, the expansion
\[
K_\nu(w)\sim 2^{\nu-1}\Gamma(\nu)w^{-\nu}
\qquad\text{as }w\to0
\]
holds along both rays under consideration. Applying it
with $\nu=s$ and $\nu=1-s$, we obtain
\[
m_s(z;\mathfrak p)\longrightarrow1,
\qquad
\partial_zm_s(z;\mathfrak p)\longrightarrow0
\qquad\text{as }\mathfrak p\to0,
\]
uniformly for $z\in I$.
For large $|w|$, the asymptotic formula
\[
K_\nu(w)=
\left(\frac{\pi}{2w}\right)^{1/2}e^{-w}
\bigl(1+O(|w|^{-1})\bigr)
\]
shows that both $m_s$ and $\partial_zm_s$ have at most
polynomial growth in $\mathfrak p$: the factor $e^{-w}$ decays
on the positive real axis and has modulus one on the
negative imaginary axis. In particular, since $0<s<1$,
\[
\sup_{z\in I}
\bigl(|m_s(z;\mathfrak p)|+|\partial_zm_s(z;\mathfrak p)|\bigr)
\le C_I(1+|\mathfrak p|).
\]
Higher derivatives are controlled by the Bessel equation
\[
\partial_{zz}m_s
=\mathfrak p\,m_s-\frac{a}{z}\partial_zm_s.
\]
Differentiating this equation repeatedly gives polynomial
bounds for every derivative in $z$, since $z^{-1}$ and
all its derivatives are bounded on $I$. Recalling that
$\mathfrak p(\xi,\tau)$ is a polynomial in $(\xi,\tau)$, we conclude
that for every integer $\beta\ge0$ there are constants
$C_{\beta,I}$ and $N_\beta$ such that
\[
\sup_{z\in I}
|\partial_z^\beta m_s(z;\xi,\tau)|
\le C_{\beta,I}(1+|\xi|+|\tau|)^{N_\beta}.
\]
Since
$\widehat u$
is rapidly decreasing,
all mixed derivatives
\[
\partial_z^\beta
\partial_{x,t}^\gamma U
\]
may therefore be obtained by differentiating under the inverse Fourier
transform.
This proves that
\[
U
\in
C^\infty(\mathbb R^{d+1}\times(0,\infty)).
\]
Finally, the multiplier
$m_s$
satisfies the ordinary differential equation
\[
\partial_{zz}m_s
+
\frac az
\partial_zm_s
-
\mathfrak p(\xi,\tau)m_s
=
0,
\]
which is precisely the Fourier transform of
\[
\partial_tU
-
i(\Delta_x U+\Ba U)
=
0.
\]
Taking inverse Fourier transforms gives the extension equation
pointwise.

It remains to verify the boundary conditions.
Theorem~\ref{T:itstheFT} gives the limits of the Fourier
multipliers as $z\to0^+$. To obtain uniform bounds at the
boundary, write $M_\nu(w)=w^\nu K_\nu(w)$ and use
\[
z^a\partial_zm_s
=-\frac{2^{1-s}}{\Gamma(s)}
\theta^{2s}|\mathfrak p|^s M_{1-s}(w),
\qquad w=z\theta|\mathfrak p|^{1/2}.
\]
The small- and large-argument estimates for $K_s$ and
$K_{1-s}$ give
\[
|m_s(z;\xi,\tau)|
+|z^a\partial_zm_s(z;\xi,\tau)|
\le C_s(1+|\mathfrak p(\xi,\tau)|),
\qquad 0<z\le1.
\]
Since $\widehat u$ is Schwartz, dominated convergence
applies in both $L^1(\RNu)$ and $L^2(\RNu)$.
Plancherel's theorem gives the strong $L^2$ limits in
item~(2), while Fourier inversion gives uniform convergence
in $(x,t)$. In particular, $U$ extends continuously to
$z=0$ with value $u$, as required in Definition~\ref{D:ep}.

\end{proof}

\section{Extension of the oscillatory Poisson operator to $L^2$}\label{S:L2poi}

We next extend the Poisson operator to $L^2$ data and establish strong
boundary convergence. Fix $0<s<1$ and let $m_s$ be the multiplier
in \eqref{ms}. The required estimates follow from the behavior of
$w^\nu K_\nu(w)$ on the positive real and negative imaginary rays.

\begin{lemma}\label{L:MultiplierBound}
Let $0<\nu<1$ and set $M_\nu(w)=w^\nu K_\nu(w)$ on the rays
$\arg w\in\{0,-\pi/2\}$. There exists $C_\nu>0$ such that:
\[
\begin{cases}
|M_\nu(w)|\le C_\nu,\qquad\qquad\qquad\ \ \ \  0<\nu\le\frac12,
\\ 
|M_\nu(w)|\le C_\nu\bigl(1+|w|^{\nu-\frac12}\bigr),\qquad \frac12<\nu<1.
\end{cases}
\]
Consequently, for $(\xi,\tau)\in L^+\cup L^-$,
\[
|m_s(z;\xi,\tau)|\le C_s
\Bigl(1+(z|\mathfrak p(\xi,\tau)|^{1/2})^{(s-\frac12)_+}\Bigr).
\]
\end{lemma}

\begin{proof}
The small-argument estimate follows from \eqref{zeroKnu}. For large argument,
\eqref{abK} gives exponential decay on the positive real ray, while on the
ray $\arg w=-\pi/2$ it gives
\[
|M_\nu(w)|\le C_\nu |w|^{\nu-\frac12}(1+O(|w|^{-1})).
\]
Combining the small- and large-argument estimates proves the assertions.

\end{proof}

Lemma~\ref{L:MultiplierBound} gives the following extension to
$L^2$ for $0<s\le\frac12$ and to $\Dom(\Le^s)$ for
$\frac12<s<1$.

\begin{theorem}\label{T:L2Extension}
Let $0<s<1$ and $z>0$. For $0<s\le\frac12$ let $u\in L^2(\RNu)$, while for $\frac12<s<1$ let $u\in\Dom(\Le^s)$. We define 
\[
\widehat{U}(\xi,z,\tau)
:=
m_s(z;\xi,\tau)\ \widehat u(\xi,\tau).
\]
If we define $T_z u(x,t) := U(x,z,t)$, then:
\begin{itemize}
\item[(i)] if $0<s\le\frac12$, $T_z:L^2(\RNu)
\longrightarrow
L^2(\RNu)$
is a bounded linear operator for every $z>0$. Moreover,  $U(\cdot,z,\cdot)\ 
\underset{z\to 0^+}{\longrightarrow}\ 
u$
\ strongly in\ $L^2(\mathbb R^{d+1})$ for every
$u\in L^2(\RNu)$;
\item[(ii)] if $\frac12<s<1$,  $T_z:\Dom(\Le^s)
\longrightarrow
L^2(\RNu)$
is a bounded linear operator for every $z>0$. Moreover, $U(\cdot,z,\cdot)\ 
\underset{z\to 0^+}{\longrightarrow}\ 
u$\ strongly in $L^2(\RNu)$ for every
$u\in\Dom(\Le^s)$.
\end{itemize}
\end{theorem}

\begin{proof}
By Plancherel's theorem, we have
\[
\|T_zu\|_{L^2(\RNu)}^2 = \|\widehat{U}(\cdot,z,\cdot)\|_{L^2(\RNu)}^2
=
\int_{\R^{d+1}}
|m_s(z;\xi,\tau)|^2\,|\widehat u(\xi,\tau)|^2\,d\xi\,d\tau.
\]

We consider separately the two regimes.

\smallskip

\noindent \underline{Case (i):} $0<s\le \frac12$.
The boundedness assertion of Lemma~\ref{L:MultiplierBound}
for $0<s\le\frac12$ gives
\[
|m_s(z;\xi,\tau)|\le C_s,
\qquad z>0,\ \ (\xi,\tau)\in L^+\cup L^-.
\]
Hence
\[
\|T_zu\|_{L^2(\R^{d+1})}^2
\le
C_s^2 \int_{\R^{d+1}} |\widehat u(\xi,\tau)|^2\,d\xi\,d\tau
=
C_s^2 \|u\|_{L^2(\R^{d+1})}^2,
\]
so $T_z$ is bounded on $L^2(\RNu)$.
Moreover, \eqref{pointwise} shows that for every fixed $(\xi,\tau)\in L^+\cup L^-$ one has 
\[
m_s(z;\xi,\tau)\longrightarrow 1
\qquad\text{as }z\to0^+,
\]
and the same bound gives
$|m_s(z;\xi,\tau)-1|^2|\widehat u(\xi,\tau)|^2
\le (C_s+1)^2|\widehat u(\xi,\tau)|^2$.
The right-hand side is integrable.
Therefore dominated convergence gives
\[
\|U(\cdot,z,\cdot)-u\|_{L^2(\R^{d+1})}\to0
\qquad\text{as }z\to0^+.
\]

\smallskip

\noindent\underline{Case (ii):} $\frac12<s<1$.
The growth estimate of Lemma~\ref{L:MultiplierBound}
for $\frac12<s<1$ gives
\[
|m_s(z;\xi,\tau)|
\le
C_s\bigl(1+(z|\mathfrak p(\xi,\tau)|^{1/2})^{\,s-\frac12}\bigr)
\le C_s\bigl(1+|\mathfrak p(\xi,\tau)|^{\,\frac{2s-1}{4}}\bigr),
\qquad 0<z\le1 .
\]
For $0<z\le1$, squaring this estimate gives
\[
|m_s(z;\xi,\tau)|^2
\le C_s\bigl(1+|\mathfrak p(\xi,\tau)|^{\,s-\frac12}\bigr)
\le C_s\bigl(1+|\mathfrak p(\xi,\tau)|^{\,2s}\bigr).
\]
For an arbitrary fixed $z>1$ the same argument gives boundedness with a
constant depending on $z$, which is sufficient for the first assertion of
part~(ii).
Hence, if $u\in\Dom(\Le^s)$, then
\[
\|T_zu\|_{L^2(\RNu)}^2
\le
C_s
\int_{\RNu}
\bigl(1+|\mathfrak p(\xi,\tau)|^{\,2s}\bigr)\,|\widehat u(\xi,\tau)|^2\,d\xi\,d\tau \le C_s\ ||u||^2_{\Dom(\Le^s)},
\]
where in the last inequality we have used the norm equivalence \eqref{equigraph}. This shows that $T_z:\Dom(\Le^s)\to L^2(\RNu)$ is bounded. 

Finally, since \(m_s(z;\xi,\tau)\to1\) pointwise as \(z\to0^+\), and
\[
|m_s(z;\xi,\tau)-1|^2
\le
C_s\bigl(1+|\mathfrak p(\xi,\tau)|^{\,2s}\bigr),
\]
dominated convergence yields
\[
\|U(\cdot,z,\cdot)-u\|_{L^2(\R^{d+1})}\to0
\qquad\text{as }z\to0^+.
\]
This completes the proof.

\end{proof}

\begin{theorem}[Dirichlet-to-Neumann convergence in $L^2$]\label{T:DtNL2}
Let $0<s<1$. If $\frac12\le s<1$ and $u\in \Dom(\Le^s)$, or if $0<s<\frac12$ and $u\in \Dom(\Le^{\frac14+\frac s2})$, then the following strong limit holds in $L^2(\RNu)$:
\begin{equation}\label{des}
-\ \underset{z\to 0^+}{\lim}\ \frac{\Gamma(s)e^{i\pi s/2}}{2^{1-2s}\Gamma(1-s)} \ z^a \p_z U(\cdot,z,\cdot)\ =\  
\Le^s u.
\end{equation}
\end{theorem}

\begin{proof}
By Plancherel's theorem and \eqref{onL+0} in Theorem \ref{T:itstheFT}, 
\begin{align*}
& \bigl\|\frac{\Gamma(s)e^{i\frac{\pi s}2}}{2^{1-2s}\Gamma(1-s)} \ z^a \p_z U(\cdot,z,\cdot)
+\ \Le^su
\bigr\|_{L^2(\RNu)}^2
\\
& =
\int_{\RNu}
\left|
\frac{\Gamma(s)e^{i\frac{\pi s}2}}{2^{1-2s}\Gamma(1-s)} z^a\ \p_z m_s(z;\xi,\tau)
+\ (i\mathfrak p(\xi,\tau))^s
\right|^2
|\widehat u(\xi,\tau)|^2\,d\xi\,d\tau,
\end{align*}
where in the last equality we have used \eqref{supernicehatV0}.
Hence it suffices to show pointwise convergence of the multiplier and to
find an \(L^2\)-majorant independent of \(z\).
From the proof of Theorem \ref{T:itstheFT} we have
\[
z^a\partial_z m_s(z;\xi,\tau)
=
-\frac{2^{\,1-s}}{\Gamma(s)}\,
\theta(\xi,\tau)^{2s}\,
|\mathfrak p(\xi,\tau)|^s\,
M_{1-s}\!\bigl(z\theta(\xi,\tau)|\mathfrak p(\xi,\tau)|^{1/2}\bigr),
\]
where \(M_\nu(w)=w^\nu K_\nu(w)\). Since
\[
M_{1-s}(w)\longrightarrow 2^{-s}\Gamma(1-s)
\qquad\text{as }w\to0,
\]
we obtain the pointwise limit
\[
\frac{\Gamma(s)e^{i\frac{\pi s}2}}{2^{1-2s}\Gamma(1-s)}\ z^a\partial_z m_s(z;\xi,\tau)
\underset{z\to 0^+}{\longrightarrow}\ 
- (i\mathfrak p(\xi,\tau))^s.
\]
Applying Lemma~\ref{L:MultiplierBound} with $\nu=1-s$, for $\frac12\le s<1$ the function $M_{1-s}$ is bounded on the admissible rays,
and therefore the majorant $C_s|\mathfrak p|^s$ follows. For $0<s<\frac12$, on $L^-$ one has
$|M_{1-s}(-iy)|\le C_s(1+y^{\frac12-s})$, and hence
\[
|z^a\partial_z m_s(z;\xi,\tau)|\le C_s|\mathfrak p(\xi,\tau)|^s\left(1+(z|\mathfrak p(\xi,\tau)|^{1/2})^{\frac12-s}\right).
\]
For $u\in\Dom(\Le^{\frac14+\frac s2})$ this yields an $L^2$ majorant for $0<z\le1$.
Thus dominated convergence applies in the two stated ranges.

\end{proof}

\section{Inverse fractional Schr\"odinger operators and their fundamental solutions}\label{S:inverse}

In this section we develop the potential theory associated with the fully
fractional Schr\"odinger operator $\Le^s$. We construct $\Le^{-s}$ from the shifted Schr\"odinger semigroup,
following the integral representation used for Riesz and parabolic
potentials \cite{R,St}, and show that its kernel is a
fundamental solution of $\Le^s$. We then compute the Poisson
extension of this kernel and relate it to the boundary operator
$\Theta_a^\star$ of \cite{GS2}.

To begin, we observe that, given $u\in \mathscr S(\R^{d+1})$, the standard dispersive estimate deriving from \eqref{PsS} implies
\begin{equation}\label{dispps}
\|\ps u(\cdot,t)\|_{L^\infty(\Rd)}\le C(d) |\sigma|^{-d/2} ||u(\cdot,t-\sigma)||_{L^1(\Rd)}.
\end{equation}
For every $1\le p\le \infty$ we obtain from \eqref{dispps}
\begin{equation}\label{pp}
\|\ps u\|_{L^p_t L^\infty_x}\le C(d) |\sigma|^{-d/2} ||u||_{L^p_t L^1_x},\ \ \ \sigma>0.
\end{equation}
On the other hand, the unitary identity \eqref{Suni} in Proposition \ref{P:L2} gives
\begin{equation}\label{twotwo}
\|\ps u\|_{L^2_t L^2_x} = ||u||_{L^2_t L^2_x},\ \ \ \ \ \sigma>0.
\end{equation}
Motivated by the representation \eqref{fS} of the fractional powers $\Le^s$, we now
introduce the inverse potential associated with the Schr\"odinger semigroup.

\begin{definition}\label{D:Spot}
Let $0<s<\frac d2$. For any $u\in \mathscr S(\R^{d+1})$, we define
\[
\mathscr J_{2s}(u)(x,t) = \frac{1}{\G(s)} \int_0^\infty \sigma^{s -1} \ps u(x,t) d\sigma.
\]
\end{definition}
 
Let us notice that $(x,t)\to \mathscr J_{2s}(u)(x,t)$ is measurable and the integral converges absolutely for a.e. $(x,t)\in \R^{d+1}$. This can be seen by writing
\begin{align*}
\mathscr J_{2s}(u)(x,t) & = \frac{1}{\G(s)} \int_0^1 \sigma^{s -1} \ps u(x,t) d\sigma + \frac{1}{\G(s)} \int_1^\infty \sigma^{s -1} \ps u(x,t) d\sigma
\\
& = \mathscr J^0_{2s}(u)(x,t) + \mathscr J^\infty_{2s}(u)(x,t),
\end{align*}
and observing that
\eqref{twotwo} implies 
\[
||\mathscr J^0_{2s}(u)||_{L^2(\R^{d+1})} \le \frac{1}{\G(1+s)} ||u||_{L^2(\R^{d+1})}.
\]
Furthermore, since $0<s<\frac d2$, from \eqref{dispps} we obtain for any $1\le p\le \infty$
\[
||\mathscr J^\infty_{2s}(u)||_{L^p_t L^{\infty}_x} \le \frac{C(d)}{\G(s)} ||u||_{L^p_t L^{1}_x}\ \int_1^\infty \sigma^{-\frac d2 + s -1} d\sigma = \frac{2\ C(d)}{\G(s)(d-2s)} ||u||_{L^p_t L^{1}_x} < \infty. 
\]

The next theorem identifies the operator introduced in Definition \ref{D:Spot} with the inverse
fractional power of the Schr\"odinger operator.

Neither $\Le^s$ nor $\mathscr J_{2s}$ preserves the Schwartz class
in general. We define their compositions by regularization. For
$\varepsilon>0$ and $v\in L^2(\R^{d+1})$, set
\begin{equation}\label{Jeps}
\mathscr J^\ve_{2s}(v)(x,t) := \frac{1}{\G(s)} \int_0^\infty \sigma^{s -1} e^{-\ve \sigma}\,P^\Le_\sigma v(x,t)\, d\sigma,
\end{equation}
an absolutely convergent $L^2$-valued integral, since
$\|P^\Le_\sigma v\|_{L^2}=\|v\|_{L^2}$ by \eqref{twotwo}, and
$\int_0^\infty \sigma^{s-1}e^{-\ve\sigma}d\sigma<\infty$. 
For $v\in\mathscr S(\RNu)$ and $0<s<d/2$,
$\mathscr J^\varepsilon_{2s}v\to\mathscr J_{2s}v$
uniformly on $\RNu$ as $\varepsilon\to0^+$.

\begin{theorem}\label{T:inverse}
Let $0<s<\min\{1,d/2\}$ and $u\in \mathscr S(\R^{d+1})$. Then
\begin{itemize}
\item[(i)] $\mathscr J^\ve_{2s}(\Le^s u)\longrightarrow u$ in $L^2(\R^{d+1})$ as
$\ve\to0^+$; we write this as $\mathscr J_{2s}\Le^su=u$;
\item[(ii)] if moreover $|\mathfrak p|^{-s}\hat u\in L^2(\R^{d+1})$, which holds
automatically when $0<s<\frac12$, then $\mathscr J_{2s}u\in\Dom(\Le^s)$ and
$\Le^s\mathscr J_{2s}u=u$.
\end{itemize}
In this sense $\mathscr J_{2s}=\Le^{-s}$.
\end{theorem}

\begin{proof}
Both assertions are immediate on the Fourier side. By \eqref{fSF} and
\eqref{Jeps}, for $v\in L^2(\R^{d+1})$,
\[
\widehat{\mathscr J^\ve_{2s}(v)}
= \frac{1}{\G(s)}\int_0^\infty \sigma^{s-1}e^{-(\ve+i\mathfrak p)\sigma}\,d\sigma\ \hat v
= (\ve + i \mathfrak p)^{-s}\,\hat v ,
\]
the interchange being legitimate because the $\sigma$-integral converges
absolutely in $L^2$, and the elementary identity
$\int_0^\infty\sigma^{s-1}e^{-\la\sigma}d\sigma=\G(s)\la^{-s}$ being valid for
$\Re\la>0$ with the principal branch.

For (i) we take $v=\Le^su$, so that $\hat v=(i\mathfrak p)^s\hat u$ by \eqref{fp2}, and
obtain
\[
\widehat{\mathscr J^\ve_{2s}(\Le^su)}
= \left(\frac{i\mathfrak p}{\ve+i\mathfrak p}\right)^{s}\hat u .
\]
Since $\mathfrak p$ is real and $\ve>0$ we have $|\ve+i\mathfrak p|^2=\ve^2+\mathfrak p^2\ge \mathfrak p^2$, whence
$\bigl|\bigl(\tfrac{i\mathfrak p}{\ve+i\mathfrak p}\bigr)^{s}\bigr|\le1$, while for every
$(\xi,\tau)$ with $\mathfrak p(\xi,\tau)\ne0$ the same quantity tends to $1$ as
$\ve\to0^+$. The set $\{\mathfrak p=0\}$ is the paraboloid
$2\pi\tau+4\pi^2|\xi|^2=0$, a smooth hypersurface and hence a null set in
$\R^{d+1}$. Dominated convergence with majorant $|\hat u|\in L^2$ therefore
gives $\widehat{\mathscr J^\ve_{2s}(\Le^su)}\to\hat u$ in $L^2$, and (i)
follows by Plancherel. Note that this argument uses only $0<s<1$.

For (ii), $(\ve+i\mathfrak p)^{-s}\hat u\to (i\mathfrak p)^{-s}\hat u$ pointwise off the paraboloid
and $|(\ve+i\mathfrak p)^{-s}\hat u|\le |\mathfrak p|^{-s}|\hat u|$, which is in $L^2$ by
hypothesis; hence $\mathscr J_{2s}u$ has Fourier transform $(i\mathfrak p)^{-s}\hat u$
and
\[
|\mathfrak p|^s\widehat{\mathscr J_{2s}u}
=e^{-i\pi s\operatorname{sgn}(\mathfrak p)/2}\hat u
\in L^2,
\]
so that
$\mathscr J_{2s}u\in\Dom(\Le^s)$ by Definition \ref{D:domain} and
$\widehat{\Le^s\mathscr J_{2s}u}=(i\mathfrak p)^s(i\mathfrak p)^{-s}\hat u=\hat u$.

 To verify the sufficient condition $0<s<\frac12$,
split the integral as follows:
\[
\int_{\RNu}|\mathfrak p|^{-2s}|\widehat u|^2
=
\int_{\{|\mathfrak p|\ge1\}}|\mathfrak p|^{-2s}|\widehat u|^2
+
\int_{\{|\mathfrak p|<1\}}|\mathfrak p|^{-2s}|\widehat u|^2.
\]
The first term is bounded by $\|\widehat u\|_{L^2}^2$.
For the second, fix $\xi$ and set
$r=2\pi\tau+4\pi^2|\xi|^2$. Since $\widehat u$ is
Schwartz, for any $N>d$ we have
\[
\begin{aligned}
\int_{\{|\mathfrak p|<1\}}|\mathfrak p|^{-2s}|\widehat u(\xi,\tau)|^2
\,d\xi\,d\tau
&=
\frac1{2\pi}\int_{\Rd}\int_{-1}^{1}
|r|^{-2s}
\left|\widehat u\left(\xi,
\frac{r-4\pi^2|\xi|^2}{2\pi}\right)\right|^2
\,dr\,d\xi\\
&\le
C_{u,N}\int_{\Rd}(1+|\xi|)^{-N}\,d\xi
\int_{-1}^{1}|r|^{-2s}\,dr
<\infty,
\end{aligned}
\]
because $2s<1$. Thus $|\mathfrak p|^{-s}\widehat u\in L^2(\RNu)$.

\end{proof}


\subsection{Fundamental solution of the fractional
Schr\"odinger operator $\Le^s$}\label{SS:fs}

Throughout this subsection, we assume
$0<s<\min\{1,d/2\}$, as in Theorem~\ref{T:inverse}.
For $u\in\mathscr S(\RNu)$, formula \eqref{PsS} gives
\begin{equation}\label{pss}
P^\Le_\sigma u(x,t)
=
\int_{\Rd}S(x,y,\sigma)u(y,t-\sigma)\,dy
=
S(\sigma)\bigl(u(\cdot,t-\sigma)\bigr)(x),
\qquad \sigma>0.
\end{equation}
Substitution in Definition~\ref{D:Spot} yields
\begin{equation}\label{J2s}
\begin{aligned}
\mathscr J_{2s}u(x,t)
&=
\frac1{\Gamma(s)}
\int_0^\infty
\sigma^{s-1}
S(\sigma)\bigl(u(\cdot,t-\sigma)\bigr)(x)\,d\sigma\\
&=
\frac1{\Gamma(s)}
\int_{-\infty}^{t}
(t-\tau)^{s-1}
\left(
\int_{\Rd}S(x,y,t-\tau)u(y,\tau)\,dy
\right)d\tau.
\end{aligned}
\end{equation}
Thus the kernel of $\mathscr J_{2s}$ is
\begin{equation}\label{Es}
E_s(x,t;y,\tau)
=
\frac{\mathbf1_{(0,\infty)}(t-\tau)}{\Gamma(s)}
(t-\tau)^{s-1}S(x,y,t-\tau).
\end{equation}

We intend to prove that $E_s$ is a fundamental solution
of $\Le^s$. For a pole at the origin, write
$E_s(x,t)=E_s(x,t;0,0)$. The identity we seek is
\begin{equation}\label{pairing}
\langle E_s,(\Le^s)^t\varphi\rangle=\varphi(0,0),
\qquad \varphi\in\mathscr S(\RNu).
\end{equation}

The pairing in \eqref{pairing} requires an extension of $E_s$
beyond $\mathscr S$, because $(\Le^s)^t\varphi$ need not be
Schwartz even for $\varphi\in C_0^\infty$. We introduce a
Banach space $\mathcal X_s$ containing $\mathscr S$ and
$(\Le^s)^t\mathscr S$, and prove that $E_s$ extends uniquely
to an element of $\mathcal X_s'$. This defines the pairing in \eqref{pairing}.

Write $\eta=(\xi,\tau)$ and set
$\mathfrak p(\eta)=2\pi\tau+4\pi^2|\xi|^2$, 
$m(\eta)=(i\mathfrak p(\eta))^s$,
with the principal branch of the complex power.
By \eqref{fp2}, we have
\[
\widehat{\Le^sf}=m\widehat f,
\qquad f\in\mathscr S(\RNu).
\]
Since $m$ has at most polynomial growth,
$m\widehat f\in L^2(\RNu)$. Hence
$\Le^sf\in L^2(\RNu)\subset\mathscr S'(\RNu)$.
The transpose $(\Le^s)^t$ will be taken with respect to the usual
bilinear pairing between $\mathscr S'$ and $\mathscr S$,
without complex conjugation. It is not difficult to verify that 
\begin{equation}\label{Les}
(\Le^s)^t\varphi
=
\mathscr F\bigl(m\,\mathscr F^{-1}\varphi\bigr)
=
\mathscr F(m)*\varphi,
\qquad \varphi\in\mathscr S(\RNu).
\end{equation}
Here $\mathscr F(m)$ is the Fourier transform of the
tempered distribution $m$, thus its convolution with
$\varphi\in\mathscr S$ is well defined.
In particular, $(\Le^s)^t\varphi\in C^\infty(\RNu)$ and has polynomial growth. Moreover, since from \eqref{Les} we have
\begin{equation}\label{FTtrans}
\widehat{(\Le^s)^t\varphi}(\eta)
=
m(-\eta)\widehat\varphi(\eta),
\end{equation}
we infer that $(\Le^s)^t\vf\in L^2(\RNu)$. 
In general, $(\Le^s)^t\varphi$ is not Schwartz; see
Remark~\ref{R:not}.

To define the space $\mathcal X_s$, we first observe that, since $0<s<1$, we have $|\mathfrak p|^{-s}\in L^1_{\mathrm{loc}}(\RNu)$, and that \[
\int_{|\eta|\le R}|\mathfrak p(\eta)|^{-s}\,d\eta
\le C_{d,s}R^{d+2(1-s)},\qquad R\ge1.
\]
Using this estimate and summing over dyadic rings gives
\[
\int_{\RNu}
(1+|\eta|)^{-N}|\mathfrak p(\eta)|^{-s}\,d\eta<\infty
\qquad\text{if }N>d+2(1-s),
\]
and the same conclusions hold for $\mathfrak p(-\eta)$.
Set
\begin{equation}\label{ws}
w_s(\eta)=
\begin{cases}
1+|\mathfrak p(-\eta)|^{-s},&\mathfrak p(-\eta)\ne0,\\
1,&\mathfrak p(-\eta)=0,
\end{cases}
\end{equation}
and consider the Banach space $L^1_s(\RNu)
=
L^1(\RNu,w_s d\eta)$, with its natural norm
\[
\|g\|_{L^1_s}
=
\int_{\RNu}|g(\eta)| w_s(\eta) d\eta.
\]
Since $w_s\ge1$, one has $L^1_s(\RNu) \hookrightarrow L^1(\RNu)$. We consider the space
\[
\mathcal X_s := \mathscr F^{-1}(L^1_s(\RNu)) \hookrightarrow C_0(\RNu),
\]
where the latter denotes the space of continuous
functions vanishing at infinity, with its natural norm 
\[
\|\psi\|_{\mathcal X_s}=\|g\|_{L^1_s},\ \ \ \ \ \ \text{if}\ \psi = \mathscr F^{-1}(g).
\]
Injectivity of the
Fourier transform on $L^1$ shows that $\psi$ determines
$g$ almost everywhere, so the norm $||\cdot||_{\mathcal X_s}$ is well defined, and $\mathcal X_s$ is a Banach space, isometrically
identified with $L^1_s(\RNu)$, and one has
\[
\|\psi\|_{L^\infty}
\le \|g\|_{L^1}
\le \|g\|_{L^1_s} = \|\psi\|_{\mathcal X_s}.
\]
The above integrability estimates also give the
continuous inclusion
\[
\mathscr S(\RNu)\hookrightarrow\mathcal X_s.
\]
Indeed, if $N>\max\{d+1,d+2(1-s)\}$, then one has
\[
\|\varphi\|_{\mathcal X_s}
\le
C_{d,s,N}
\sup_{\eta\in\RNu}
(1+|\eta|)^N|\widehat\varphi(\eta)|.
\]

Next, it is easy to see that
\begin{equation}\label{cont}
(\Le^s)^t:\mathscr S(\RNu)\longrightarrow\mathcal X_s
\end{equation}
continuously. Let in fact $\vf\in\mathscr S(\RNu)$. Then \eqref{FTtrans} and the definitions of $m$ and $w_s$, give
\begin{align*}
\|(\Le^s)^t\varphi\|_{\mathcal X_s}
&= ||m(-\eta) \widehat \vf||_{L^1_s} = 
\int_{\RNu}
\bigl(1+|\mathfrak p(-\eta)|^s\bigr)
|\widehat\vf(\eta)|d\eta.
\end{align*}
The last integral is bounded by a continuous Schwartz
seminorm. Consequently, \eqref{cont} holds.

For $T\in\mathcal X_s'$, we may now define
$\Le^sT\in\mathscr S'$ by
\begin{equation}\label{action}
\langle\Le^sT,\varphi\rangle
=
\langle T,(\Le^s)^t\varphi\rangle,
\qquad \varphi\in\mathscr S(\RNu),
\end{equation}
where the pairing on the right is between
$\mathcal X_s'$ and $\mathcal X_s$. This defines a tempered distribution because the map
in \eqref{cont} is continuous.
We next prove that
\begin{equation}\label{dense}
\overline{\mathscr S(\RNu)}^{\mathcal X_s} = \mathcal X_s.
\end{equation}
To see \eqref{dense}, let $\psi= \mathscr F^{-1}(g)\in \mathcal X_s$, with $g \in L^1_s(\RNu)$, and define $g_n=\mathbf1_{K_n}g$, where
\[
K_n=\{\eta\in \RNu\mid |\eta|\le n,\ |\mathfrak p(-\eta)|\ge n^{-1}\}.
\]
Since $K_n\nearrow \RNu\setminus\{\mathfrak p(-\eta) =0\}$, dominated convergence gives
\[
\|g-g_n\|_{L^1_s}
=
\int_{\RNu\setminus K_n}|g(\eta)| w_s(\eta) d\eta
\underset{n\to \infty}{\longrightarrow}\ 0.
\]
The bounded open set
\[
U_n=\{\eta\in \RNu\mid |\eta|<n+1,\ |\mathfrak p(-\eta)|>(2n)^{-1}\}
\]
contains $K_n$, and
\[
w_s(\eta)\le 1+(2n)^s
\qquad\text{for }\eta\in U_n.
\]
By density of $C_0^\infty(U_n)$ in $L^1(U_n)$, choose
$h_n\in C_0^\infty(U_n)$ such that
\[
\|g_n-h_n\|_{L^1}
<
\frac{1}{n\bigl(1+(2n)^s\bigr)}.
\]
Both functions vanish outside $U_n$, so
\[
\|g_n-h_n\|_{L^1_s}
\le
\bigl(1+(2n)^s\bigr)\|g_n-h_n\|_{L^1}
<\frac1n.
\]
Consequently,
\[
\|g-h_n\|_{L^1_s}
\le \|g-g_n\|_{L^1_s}+\frac1n\ \ 
\underset{n\to\infty}{\longrightarrow}\ \ 0.
\]
Since $\psi_n = \mathscr F^{-1}h_n\in\mathscr S(\RNu)$ and
\[
||\psi - \psi_n||_{\mathcal X_s} = \|\mathscr F^{-1}g-\mathscr F^{-1}h_n\|_{\mathcal X_s}
=
\|g-h_n\|_{L^1_s},
\]
we have shown that \eqref{dense} is true. Every element of $\mathcal X_s'$ restricts to a tempered
distribution, and density makes this restriction
injective.

\begin{proposition}[Fundamental solution]
For each $(y,\tau)\in\RNu$, the kernel
$E_s(\cdot,\cdot;y,\tau)$ defines a tempered distribution
with a unique continuous extension to $\mathcal X_s$.
With the operator action defined above,
\[
\Le^sE_s(\cdot,\cdot;y,\tau)=\delta_{(y,\tau)}.
\]
Equivalently,
\[
\bigl\langle E_s(\cdot,\cdot;y,\tau),
(\Le^s)^t\varphi\bigr\rangle
=
\varphi(y,\tau),
\qquad \varphi\in\mathscr S(\RNu).
\]
\end{proposition}

\begin{proof}
We first take the pole at the origin, and for $\vf\in\mathscr S(\RNu)$, we define
\begin{equation}\label{Ess}
\langle E_s,\varphi\rangle
:=
\frac1{\Gamma(s)}
\int_0^\infty t^{s-1}
\left(\int_{\Rd}S(x,t)\varphi(x,t)\,dx\right)dt,
\end{equation}
with the spatial integral taken first. The estimate below
shows that \eqref{Ess} defines a distribution in $\mathscr S'(\RNu)$. Choose an integer $k>d/2$ and set
\[
P_k(\varphi)
=
\max_{|\beta|\le2k}
\sup_{(x,t)\in\RNu}
(1+|x|)^{d+1}(1+|t|)^2
|\partial_x^\beta\varphi(x,t)|.
\]
Since
\[
(1+4\pi^2|\xi|^2)^k\mathscr F_x\varphi(\xi,t)
=
\mathscr F_x\bigl((1-\Delta_x)^k\varphi\bigr)(\xi,t),
\]
we have
\[
|\mathscr F_x\varphi(\xi,t)|
\le
C_{d,k}P_k(\varphi)
(1+|t|)^{-2}(1+4\pi^2|\xi|^2)^{-k}.
\]
Integration in $\xi$ gives
\[
\int_{\Rd}|\mathscr F_x\varphi(\xi,t)|\,d\xi
\le C_{d,k}P_k(\varphi)(1+|t|)^{-2}.
\]
For each fixed $t>0$, Fourier duality in
$\mathscr S'(\Rd)$ gives
\[
\begin{aligned}
\int_{\Rd}S(x,t)\varphi(x,t) dx
&=
\int_{\Rd}e^{-4\pi^2it|\xi|^2}
\mathscr F_x\varphi(-\xi,t) d\xi.
\end{aligned}
\]
It follows that
\[
\begin{aligned}
|\langle E_s,\varphi\rangle|
&\le
\frac{C_{d,k}}{\Gamma(s)}P_k(\vf)
\int_0^\infty t^{s-1}(1+t)^{-2} dt\\
&=
C_{d,k}\Gamma(2-s)P_k(\vf),
\end{aligned}
\]
which proves 
$E_s\in\mathscr S'(\RNu)$.
We now claim that 
\begin{equation}\label{Esft}
\widehat E_s=m^{-1}
\qquad\text{in }\mathscr S'(\RNu),
\end{equation}
where $m(\eta)=(i\mathfrak p(\eta))^s$ as above. Once this is known, we can complete the proof of \eqref{pairing} as follows.
For $\psi\in\mathscr S(\RNu)$, keeping in mind that $m^{-1}\in L^1_{loc}(\RNu)$, by Fourier duality and
\eqref{Esft} we find
\begin{equation}\label{Escont}
\langle E_s,\psi\rangle
=
\int_{\RNu}m(\eta)^{-1}\widehat\psi(-\eta)\,d\eta
=
\int_{\RNu}m(-\eta)^{-1}\widehat\psi(\eta)\,d\eta.
\end{equation}
For a general $\psi=\mathscr F^{-1}g\in\mathcal X_s$,
we therefore define
\[
\langle E_s,\psi\rangle
:=
\int_{\RNu}m(-\eta)^{-1}g(\eta)\,d\eta.
\]
Note that this integral is absolutely convergent, and we have by \eqref{ws}
\[
|\langle E_s,\psi\rangle|
\le
\int_{\RNu}|g(\eta)| |\mathfrak p(-\eta)|^{-s} d\eta
\le \|\psi\|_{\mathcal X_s}.
\]
It therefore defines a continuous extension of $E_s$ to
$\mathcal X_s$, unique by the density established
in \eqref{dense}.
Now let $\vf\in\mathscr S(\RNu)$. By \eqref{cont}, we have
$(\Le^s)^t\vf\in\mathcal X_s$, and by \eqref{FTtrans} we have
\begin{align*}
\langle E_s,(\Le^s)^t\vf\rangle
&=
\int_{\RNu}
m(-\eta)^{-1}m(-\eta)\widehat\vf(\eta) d\eta =
\int_{\RNu}\widehat\vf(\eta) d\eta
=\varphi(0,0).
\end{align*}
This proves \eqref{pairing}, and therefore $\Le^sE_s=\delta$
with the operator action defined by \eqref{action}.
For a pole $(y,\tau)$, translation gives
\[
\bigl\langle E_s(\cdot-y,\cdot-\tau),
(\Le^s)^t\varphi\bigr\rangle
=
\bigl\langle E_s,
(\Le^s)^t[\varphi(\cdot+y,\cdot+\tau)]\bigr\rangle
=
\varphi(y,\tau).
\]
These pairings are well defined because translations
preserve $\mathcal X_s$ and its norm, and commute
with $(\Le^s)^t$.

To complete the proof of the proposition, we are thus left with checking \eqref{Esft}. For this, we define for $\ve>0$,
\[
\sa E_{s,\ve},\vf\da
=
\frac1{\Gamma(s)}
\int_0^\infty t^{s-1}e^{-\varepsilon t}
\left(\int_{\Rd}S(x,t)\varphi(x,t)\,dx\right)dt.
\]
The preceding bound and dominated convergence give
$E_{s,\varepsilon}\to E_s$ in $\mathscr S'(\RNu)$.
Taking the Fourier transform in $(x,t)$ gives
\begin{align*}
\langle\widehat E_{s,\varepsilon},\chi\rangle
&=
\frac1{\Gamma(s)}
\int_0^\infty t^{s-1}e^{-\varepsilon t}
\int_{\RNu}e^{-it\mathfrak p(\eta)}\chi(\eta)\,d\eta\,dt\\
&=
\int_{\RNu}
\left[
\frac1{\Gamma(s)}
\int_0^\infty
t^{s-1}e^{-(\varepsilon+i\mathfrak p(\eta))t}\,dt
\right]\chi(\eta)\,d\eta\\
&=
\int_{\RNu}
(\varepsilon+i\mathfrak p(\eta))^{-s}\chi(\eta) d\eta,
\end{align*}
where the last equality follows since
$\Re(\ve+i\mathfrak p(\eta))=\ve>0$. We also note that
Fubini's theorem applies because
\[
\int_0^\infty\int_{\RNu}
t^{s-1}e^{-\ve t}|\chi(\eta)| d\eta  dt
=
\Gamma(s)\ve^{-s}\|\chi\|_{L^1}<\infty.
\]
We conclude that
\[
\widehat E_{s,\ve}
=
(\ve+i\mathfrak p)^{-s}
\qquad\text{in }\mathscr S'(\RNu).
\]
Since we have for a.e. $\eta\in \RNu$: $|(\ve+i\mathfrak p(\eta))^{-s}|
\le |\mathfrak p(\eta)|^{-s}\in L^1_{loc}(\RNu)$,
passing to the limit as $\ve\to 0^+$ we conclude that \eqref{Esft} holds.

\end{proof}

\subsection{The linear Cauchy problem with prescribed past}
\label{SS:linear-cauchy}

We now use the fundamental solution $E_s$ to construct a solution
of the linear inhomogeneous problem with prescribed past. Precisely, given $\Phi, v_0\in\mathscr S(\RNu)$ we seek a solution $u$ to the problem
\begin{equation}\label{linear}
\begin{cases}
\Le^su(x,t)=\Phi(x,t),&t>0,\\
u(x,t)=v_0(x,t),&t\le0.
\end{cases}
\end{equation}
We first establish a potential formula for a general
source $F\in L^1(\RNu)\cap L^2(\RNu)$, and then choose
\begin{equation}\label{F}
F= \begin{cases} \Phi\ \ \  \ \ \ \text{for}\ \  t>0,
\\
 \Le^sv_0\ \  \text{for}\ \  t\le0,
 \end{cases}
 \end{equation}
 see Proposition \ref{P:fsLs}.
For $u\in\mathscr S(\RNu)$,
the inversion formula of Theorem~\ref{T:inverse} reads
\begin{equation}\label{E:fundamental-inversion}
u
=
\lim_{\varepsilon\downarrow0}
\frac1{\Gamma(s)}
\int_0^\infty
\sigma^{s-1}e^{-\varepsilon\sigma}
S(\sigma)
\bigl((\Le^su)(\cdot,t-\sigma)\bigr)\,d\sigma
\qquad\text{in }L^2(\RNu).
\end{equation}
For each $\varepsilon>0$, the integral is a Bochner
integral in $L^2(\RNu)$.
Let now $F\in L^1(\RNu)\cap L^2(\RNu)$. To solve
\begin{equation}\label{Lsu=F}
\Le^su=F,
\end{equation}
we define $u\in\mathcal X_s'$ by
\[
\langle u,\psi\rangle
=
\int_{\RNu}
m(-\eta)^{-1}\widehat F(-\eta)g(\eta)\,d\eta,
\qquad
\psi=\mathscr F^{-1}g\in\mathcal X_s.
\]
Since $\widehat F$ is bounded, this integral is absolutely
convergent and, keeping $m(\eta)=(i\mathfrak p(\eta))^s$ and \eqref{ws} in mind, we have
\[
|\langle u,\psi\rangle|
\le
\|\widehat F\|_{L^\infty}
\int_{\RNu}|\mathfrak p(-\eta)|^{-s}|g(\eta)|\,d\eta
\le
\|\widehat F\|_{L^\infty}\|\psi\|_{\mathcal X_s}.
\]
Its restriction to $\mathscr S$ is the tempered distribution
whose Fourier transform is
\[
\widehat u=m^{-1}\widehat F.
\]
For $\varphi\in\mathscr S(\RNu)$, formula \eqref{FTtrans}
therefore gives
\begin{align*}
\langle\Le^su,\varphi\rangle
&=\langle u,(\Le^s)^t\varphi\rangle
=
\int_{\RNu}
m(-\eta)^{-1}\widehat F(-\eta)
m(-\eta)\widehat\vf(\eta)d\eta
=
\int_{\RNu}
\widehat F(-\eta)\widehat\varphi(\eta)d\eta
 =\langle F,\varphi\rangle.
\end{align*}
Thus \eqref{Lsu=F} holds
in $\mathscr S'(\RNu)$. The solution $u$ of \eqref{Lsu=F} defined by
$\widehat u=(i\mathfrak p)^{-s}\widehat F$ also admits the
integral representation
\begin{equation}\label{potential}
u
=
\lim_{\ve\downarrow0}
\frac1{\Gamma(s)}
\int_0^\infty
\sigma^{s-1}e^{-\ve\sigma}
S(\sigma)\bigl(F(\cdot,t-\sigma)\bigr)\,d\sigma
\qquad\text{in }\mathscr S'(\RNu).
\end{equation}
Indeed, denote the integral in the right-hand side of \eqref{potential}  by $u_\ve$.
Unitarity of $S(\sigma)$ and invariance of the $L^2$ norm
under time translation give
\[
\|u_\varepsilon\|_{L^2(\RNu)}
\le
\frac{\|F\|_{L^2(\RNu)}}{\Gamma(s)}
\int_0^\infty \sigma^{s-1}e^{-\varepsilon\sigma}\,d\sigma
=
\varepsilon^{-s}\|F\|_{L^2(\RNu)}.
\]
In particular, the integral converges in $L^2(\RNu)$.
Taking its Fourier transform in $(x,t)$ yields
\[
\widehat u_\varepsilon(\eta)
=
\frac1{\Gamma(s)}
\int_0^\infty
\sigma^{s-1}e^{-(\varepsilon+i\mathfrak p(\eta))\sigma}
\widehat F(\eta)\,d\sigma
=
(\varepsilon+i\mathfrak p(\eta))^{-s}\widehat F(\eta).
\]
Since
\[
|(\varepsilon+i\mathfrak p(\eta))^{-s}\widehat F(\eta)|
\le
\|\widehat F\|_{L^\infty}|\mathfrak p(\eta)|^{-s},
\]
the weighted integrability estimate for $|\mathfrak p|^{-s}$
allows dominated convergence against Schwartz functions.
Consequently, $\widehat u_\varepsilon\to m^{-1}\widehat F$
in $\mathscr S'$, proving \eqref{potential}.
If $m^{-1}\widehat F\in L^2(\RNu)$, then
$u\in\Dom(\Le^s)$ and \eqref{Lsu=F} holds in $L^2(\RNu)$.

To recover the prescribed past, we take $F$ as in \eqref{F}.
Although $\Le^sv_0\in L^2(\RNu)$ for
$v_0\in\mathscr S(\RNu)$, it need not belong to
$L^1(\RNu)$. To see this, 
take $f(x)=e^{-|x|^2}$, $0\le h\in C_0^\infty((-1,0)), h\not\equiv0$,
and set $v_0(x,t)=h(t)S(t)f(x)$.
For $t>0$, we have $v_0(\cdot,t)=0$, and one can check that 
\begin{align*}
\Le^sv_0(\cdot,t)
&=
-\frac{s}{\Gamma(1-s)}
\int_0^\infty
\sigma^{-1-s}h(t-\sigma)
S(\sigma)S(t-\sigma)f d\sigma =
-\frac{s}{\Gamma(1-s)}S(t)f
\int_{-1}^{0}\frac{h(\tau)}{(t-\tau)^{1+s}} d\tau.
\end{align*}
Since
\[
S(t)f(x)
=
(1+4it)^{-d/2}
\exp\left(-\frac{|x|^2}{1+4it}\right),
\]
one easily finds
\[
\|S(t)f\|_{L^1_x}
=
(1+16t^2)^{-d/4}
\int_{\Rd}
\exp\left(-\frac{|x|^2}{1+16t^2}\right) dx
=
\pi^{d/2}(1+16t^2)^{d/4}.
\]
Since $h\ge0$ and
\[
\lim_{t\to+\infty}
t^{1+s}\int_{-1}^{0}
\frac{h(\tau)}{(t-\tau)^{1+s}} d\tau
=
\int_{-1}^{0}h(\tau)d\tau > 0,
\]
we obtain
\[
\|(\Le^sv_0)(\cdot,t)\|_{L^1_x}
\sim
\frac{s\,2^d\pi^{d/2}}{\Gamma(1-s)}
\left(\int_{-1}^{0}h(\tau)d\tau\right)
t^{d/2-1-s}
\qquad\text{as }t\to\infty.
\]
Consequently,
\[
\int_1^\infty
\|(\Le^sv_0)(\cdot,t)\|_{L^1_x}\,dt
=\infty
\qquad\text{for }0<s<1\text{ with }s\le d/2.
\]
This shows that $\Le^sv_0\notin L^1(\RNu)$ in the range
$0<s<\min\{1,d/2\}$.
Nevertheless, $\Le^sv_0\in L^2(\RNu)$, since
$(i\mathfrak p)^s\widehat v_0\in L^2(\RNu)$.

In order to guarantee that $F$ in \eqref{F} belongs to $L^1(\RNu)\cap L^2(\RNu)$, 
we must therefore verify that the
restriction of $\Le^sv_0$ to $t\le0$ is integrable.
This is sufficient,
since its future part is the Schwartz function $\Phi$.

\begin{proposition}\label{P:fsLs}
Let $\Phi,v_0\in\mathscr S(\RNu)$, and define $F$ by
\eqref{F}. Then $F\in L^1(\RNu)\cap L^2(\RNu)$, and
formula \eqref{potential} defines a solution $u$ of
$\Le^su=F$ in $\mathscr S'(\RNu)$.
This solution has an $L^2(\Rd)$-valued representative
given by
\begin{equation}\label{pot2}
u(\cdot,t)=
\begin{cases}
v_0(\cdot,t), & t\le0,\\[6pt]
\begin{aligned}[b]
&\frac{1}{\Gamma(s)}
\int_0^t \sigma^{s-1}
S(\sigma)\bigl(\Phi(\cdot,t-\sigma)\bigr)d\sigma\\
&\quad+\frac{1}{\Gamma(s)}
\int_t^\infty \sigma^{s-1}
S(\sigma)\bigl((\Le^sv_0)(\cdot,t-\sigma)\bigr)d\sigma,
\end{aligned}
& t>0.
\end{cases}
\end{equation}
Both integrals converge in $L^2(\Rd)$ for every $t>0$.
\end{proposition}

 \begin{proof}
We first establish the decay of $\Le^sv_0$ for $t\le0$.
This will prove that the source $F$ in \eqref{F} belongs to
$L^1(\RNu)\cap L^2(\RNu)$, and it will justify the integral
formula \eqref{pot2}.

By weighted Cauchy-Schwarz, we easily obtain for any integer $k>d/2$: 
\[
\|S(\sigma)v_0(\cdot,t-\sigma)\|_{L^1_x}
\le
C_{d,k}\sum_{|\alpha|\le k}
\|x^\alpha S(\sigma)v_0(\cdot,t-\sigma)\|_{L^2_x}.
\]
To estimate the terms on the right, differentiation of
the Fourier multiplier $e^{-4\pi^2i\sigma|\xi|^2}$ gives
\[
x_jS(\sigma)v_0(\cdot,t-\sigma)
=
S(\sigma)(x_j-2i\sigma\partial_{x_j})
v_0(\cdot,t-\sigma).
\]
Iterating this identity and using unitarity, we obtain
\[
\|S(\sigma)v_0(\cdot,t-\sigma)\|_{L^1_x}
+\|S(\sigma)v_0(\cdot,t-\sigma)\|_{L^2_x}
\le
C_{d,k}(1+\sigma)^k
\sum_{|\alpha|+|\beta|\le k}
\|x^\alpha\partial_x^\beta
v_0(\cdot,t-\sigma)\|_{L^2_x}.
\]
Indeed, expanding the iterated operators produces
spatial derivatives and powers of $x$ of total order
at most $k$, with coefficients bounded by
$C_{d,k}(1+\sigma)^k$.

Since $v_0$ is Schwartz and
$|t-\sigma|=|t|+\sigma$ for $t\le0$ and $\sigma\ge0$,
it follows that, for every $N>0$,
\begin{equation}\label{E:fsLs-propagator}
\|S(\sigma)v_0(\cdot,t-\sigma)\|_{L^1_x}
+\|S(\sigma)v_0(\cdot,t-\sigma)\|_{L^2_x}\le
C_N(1+\sigma)^k(1+|t|+\sigma)^{-N},
\qquad t\le0,\quad \sigma\ge0.
\end{equation}
The same argument applies to $\Le v_0$, which is also
Schwartz. In particular, for every $M>0$ there exists $C_M>0$ such that for $t\le0$ and $0\le\tau\le1$
\begin{equation}\label{prop2}
\|S(\tau)(\Le v_0)(\cdot,t-\tau)\|_{L^1_x}
+
\|S(\tau)(\Le v_0)(\cdot,t-\tau)\|_{L^2_x}
\le C_M(1+|t|)^{-M}.
\end{equation}

We now use \eqref{E:fsLs-propagator}, \eqref{prop2} in the formula \eqref{fS}, which gives
\[
\Le^sv_0(\cdot,t)
=
-\frac{s}{\Gamma(1-s)}
\int_0^\infty
\frac{S(\sigma)v_0(\cdot,t-\sigma)-v_0(\cdot,t)}
{\sigma^{1+s}}\,d\sigma.
\]
For $0<\sigma\le1$, differentiation gives
\[
S(\sigma)v_0(\cdot,t-\sigma)-v_0(\cdot,t)
=
-\int_0^\sigma
S(\tau)(\Le v_0)(\cdot,t-\tau)\,d\tau.
\]
Consequently, \eqref{prop2} gives for any $M>0$ 
\[
\|S(\sigma)v_0(\cdot,t-\sigma)-v_0(\cdot,t)\|_{L^1_x}+
\|S(\sigma)v_0(\cdot,t-\sigma)-v_0(\cdot,t)\|_{L^2_x}
\le C_M\ \sigma \ (1+|t|)^{-M}.
\]
After division by $\sigma^{1+s}$, this bound is
integrable on $(0,1)$ because $s<1$.

For $\sigma\ge1$, choose $N>M+k+1$ in
\eqref{E:fsLs-propagator}. Then
\[
\begin{aligned}
&\int_1^\infty \sigma^{-1-s}
\left(
\|S(\sigma)v_0(\cdot,t-\sigma)\|_{L^1_x}
+
\|S(\sigma)v_0(\cdot,t-\sigma)\|_{L^2_x}
\right)d\sigma\\
&\qquad\le
C_N\int_1^\infty
\sigma^{-1-s}(1+\sigma)^k
(1+|t|+\sigma)^{-N}\,d\sigma\\
&\qquad\le C_M(1+|t|)^{-M}.
\end{aligned}
\]
The contribution of $v_0(\cdot,t)$ has the same bound,
since $v_0$ is Schwartz and
$\int_1^\infty\sigma^{-1-s}\,d\sigma<\infty$.
Combining the two ranges proves that, for every $M>0$,
\begin{equation}\label{E:fsLs-past-decay}
\|(\Le^sv_0)(\cdot,t)\|_{L^1_x}
+
\|(\Le^sv_0)(\cdot,t)\|_{L^2_x}
\le C_M(1+|t|)^{-M},
\qquad t\le0.
\end{equation}
The same bounds permit dominated convergence in the
fractional-power formula as $t$ varies in $(-\infty,0]$.
Thus $t\mapsto(\Le^sv_0)(\cdot,t)$ is continuous
with values in $L^2(\Rd)$ on that half-line.

It follows from \eqref{E:fsLs-past-decay} and the
Schwartz decay of $\Phi$ that
$F\in L^1(\RNu)\cap L^2(\RNu)$.
The construction preceding the proposition therefore
defines $u\in\mathcal X_s'$ satisfying
$\Le^su=F$ in $\mathscr S'(\RNu)$.

It remains to identify $u$ as a function of time.
Write
\[
u_\varepsilon(\cdot,t)
=
\frac1{\Gamma(s)}
\int_0^\infty
\sigma^{s-1}e^{-\varepsilon\sigma}
S(\sigma)\bigl(F(\cdot,t-\sigma)\bigr)\,d\sigma.
\]
By \eqref{potential}, $u_\varepsilon\to u$ in
$\mathscr S'(\RNu)$.

For $t\le0$, the definition of $F$ gives
\[
u_\varepsilon(\cdot,t)
=
\mathscr J_{2s}^{\varepsilon}(\Le^sv_0)(\cdot,t).
\]
Unitarity and \eqref{E:fsLs-past-decay} imply
\[
\sigma^{s-1}
\|S(\sigma)(\Le^sv_0)(\cdot,t-\sigma)\|_{L^2_x}
\le
C_M\sigma^{s-1}(1+\sigma)^{-M},
\qquad t\le0.
\]
For $M>s$, the right-hand side is integrable on
$(0,\infty)$. Hence the integrals converge in $L^2_x$
as $\varepsilon\downarrow0$, uniformly for $t$ in
compact intervals of $(-\infty,0]$, and their limit
is continuous there.
On the other hand, Theorem~\ref{T:inverse} gives
\[
\mathscr J_{2s}^{\varepsilon}(\Le^sv_0)
\longrightarrow v_0
\qquad\text{in }L^2(\RNu).
\]
The two limits agree for almost every $t<0$.
Their continuity gives equality for every $t\le0$,
including $t=0$.

For $t>0$, splitting the integral at $\sigma=t$ gives
\[
\begin{aligned}
u_\varepsilon(\cdot,t)
&=
\frac1{\Gamma(s)}
\int_0^t
\sigma^{s-1}e^{-\varepsilon\sigma}
S(\sigma)\bigl(\Phi(\cdot,t-\sigma)\bigr)\,d\sigma\\
&\quad+
\frac1{\Gamma(s)}
\int_t^\infty
\sigma^{s-1}e^{-\varepsilon\sigma}
S(\sigma)\bigl((\Le^sv_0)(\cdot,t-\sigma)\bigr)\,d\sigma.
\end{aligned}
\]
The first integral converges in $L^2_x$ after removing
the exponential factor, since $s>0$ and $\Phi$ is
continuous with values in $L^2_x$.
For the second, \eqref{E:fsLs-past-decay} gives
\[
\int_t^\infty \sigma^{s-1}
\|S(\sigma)(\Le^sv_0)(\cdot,t-\sigma)\|_{L^2_x}\,d\sigma
\le
C_M\int_t^\infty
\sigma^{s-1}(1+\sigma-t)^{-M}\,d\sigma<\infty
\]
when $M>s$.
Dominated convergence therefore gives the expression
in \eqref{pot2} for $t>0$.

Finally, on every bounded time interval, the integrands
defining $u_\varepsilon$ have a common integrable
$L^2_x$ bound: for bounded $\sigma$ this follows from
the boundedness of the spatial $L^2$ norms of $F$,
and for large $\sigma$ from
\eqref{E:fsLs-past-decay}.
Thus $u_\varepsilon$ converges locally in
$L^1_tL^2_x$ to the function just described.
Its distributional limit is therefore $u$, which
proves the asserted representation \eqref{pot2}.

\end{proof}


\subsection{Poisson lifting of the fundamental solution}\label{P:liftingEs}

We next compute the Poisson extension of the fundamental solution $E_s$ in \eqref{Es}. According to \eqref{U} such extension is provided by
\begin{equation}\label{EEs}
\mathscr  E_s(x,z,t) := \int_0^\infty \int_{\Rd} \mathbb P_a(X,Y_0,\sigma) E_s(y,t-\sigma) dy d\sigma,\ \ \ \ a=1-2s,
\end{equation}
where for brevity we have written 
\[
E_s(y,t-\sigma) = E_s(y,t;0,\sigma) = \frac{1}{\G(s)} \frac{\mathbf 1_{(0,\infty)}(t-\sigma)}{(t-\sigma)^{1-s}}  S(y,0,t-\sigma). 
\]
For each $z>0$, we define the Poisson lifting of $E_s$ by
\[
\widehat{\mathscr E_s}(\xi,z,\tau)
=
m_s(z;\xi,\tau)(i\mathfrak p(\xi,\tau))^{-s},
\]
where $m_s$ is the Poisson multiplier defined in
\eqref{ms}, and the Fourier transform is taken in $(x,t)$.
This is an ordinary product of locally integrable functions: $m_s$ is bounded near $\mathfrak p=0$, and $|\mathfrak p|^{-s}$ is locally integrable because $0<s<1$. The product defines a tempered distribution by the polynomial bounds at infinity. Formula~\eqref{EEs} and the identity below are understood with this Fourier prescription.

\begin{proposition}\label{P:Es}
For every $0<s<\min\{1,d/2\}$, $t>0$ and $z>0$, we have
\[
\mathscr  E_s(x,z,t) = \frac{1}{\G(s)} \frac{e^{i \frac{z^2}{4t}}}{t^{1-s}} S(x,0,t) = e^{i \frac{z^2}{4t}} \ E_s(x,t).
\]
\end{proposition}

\begin{proof} 
For $z>0$ and $\varepsilon\ge0$, define $K_{z,\varepsilon}$ by
\[
\langle K_{z,\varepsilon},\varphi\rangle
=\frac1{\Gamma(s)}\int_0^\infty
t^{s-1}e^{-\varepsilon t}e^{iz^2/(4t)}
\left(\int_{\Rd}S(x,t)\varphi(x,t)\,dx\right)dt,
\qquad \varphi\in\mathscr S(\RNu),
\]
with the spatial integral taken first, and put $K_z=K_{z,0}$.
As in the proof of \eqref{Esft},
\[
\left|\int_{\Rd}S(x,t)\varphi(x,t)\,dx\right|
\le\|\mathscr F_x\varphi(\cdot,t)\|_{L^1}
\le C P_k(\varphi)(1+t)^{-2},
\]
where $P_k$ is the Schwartz seminorm used there.
The outer integral therefore converges absolutely and defines
a tempered distribution. Dominated convergence also gives
$K_{z,\varepsilon}\to K_z$ in $\mathscr S'$ as
$\varepsilon\downarrow0$.
For $\varepsilon>0$, Fourier transformation in $(x,t)$ gives
\[
\widehat K_{z,\varepsilon}(\xi,\tau)
=\frac1{\Gamma(s)}\int_0^\infty
t^{s-1}e^{-(\varepsilon+i\mathfrak p)t+iz^2/(4t)}\,dt.
\]
Here the interchange is justified after pairing with
$\chi\in\mathscr S(\RNu)$ by
\[
\int_{\RNu}\int_0^\infty
t^{s-1}e^{-\varepsilon t}|\chi(\xi,\tau)|
\,dt\,d\xi\,d\tau
=\Gamma(s)\varepsilon^{-s}\|\chi\|_{L^1}.
\]
For $\operatorname{Im}q<0$, write
\[
m_s(z;q)=\frac{2^{1-s}}{\Gamma(s)}
(z\sqrt q)^s K_s(z\sqrt q),
\]
with the square root of positive real part.
Applying \eqref{reg} with $\beta=\delta-iz^2/4$ and
$\gamma=\varepsilon+i\mathfrak p$, and then letting
$\delta\downarrow0$, gives
\[
\widehat K_{z,\varepsilon}
=(\varepsilon+i\mathfrak p)^{-s}
m_s(z;\mathfrak p-i\varepsilon).
\]
The limit in this application of \eqref{reg} follows by
dominated convergence with majorant $t^{s-1}e^{-\varepsilon t}$.
Moreover, $\beta\gamma\to(z^2/4)(\mathfrak p-i\varepsilon)$,
so the branch is precisely the one in
\eqref{E:limiting-absorption}.

The small- and large-argument bounds for $K_s$ are uniform
on $-\pi/2\le\arg w\le0$. They imply, for
$0<\varepsilon\le1$ and almost every $(\xi,\tau)$,
\[
|\widehat K_{z,\varepsilon}(\xi,\tau)|
\le C_{s,z}(1+|\mathfrak p(\xi,\tau)|^{-s}).
\]
This bound is integrable against every Schwartz function
by the weighted integrability estimate preceding \eqref{ws}.
Dominated convergence therefore gives
\[
\widehat K_z
=m_s(z;\xi,\tau)(i\mathfrak p(\xi,\tau))^{-s}
=\widehat{\mathscr E_s}.
\]
Thus $K_z=\mathscr E_s$ as tempered distributions. On $t>0$,
the defining integral for $K_z$ has the representative
\[
K_z(x,t)=\frac{t^{s-1}}{\Gamma(s)}
e^{iz^2/(4t)}S(x,0,t),
\]
which proves the proposition.

\end{proof}

Proposition \ref{P:Es} shows that the Poisson lifting of the fundamental solution is obtained simply by multiplication by the oscillatory phase $e^{i \frac{z^2}{4t}}$. 
We use this formula to identify the kernel of the boundary operator
$\Theta_a^\star$. Recall the Duhamel representation of \cite{GS2}
for the problem
\begin{equation}\label{nozeroin}
\begin{cases}
\partial_tU-i(\Delta_x U+\Ba U)=F(X,t),\\
\displaystyle \lim_{z\to0^+} z^a\partial_zU(X,t)=\Phi(x,t),\\
U(X,0)=U_0(X),
\end{cases}
\end{equation}
where $\Ba$ is given by \eqref{Ba} and $-1<a<1$. Specifically, we have proved that solutions can be represented as
\[
U= \mathbb T_a^\star(U_0)+\mathbb D_a(F)+\Theta_a^\star(\Phi),
\]
where the operators $\mathbb T_a^\star$ and $\mathbb D_a$ describe the bulk evolution, while $\Theta_a^\star$ encodes the boundary interaction. Since the bulk propagator plays an essential role in Section \ref{S:memory} below, we recall its definition explicitly. With $\mathbb S_a$ the kernel \eqref{grandeSaS0} and $d\omega_a(Y) = \zeta^a dy\, d\zeta$, one has
\begin{equation}\label{E:Tstar}
\mathbb T_a^\star(U_0)(X,t) := \int_{\RN} \mathbb S_a(X,Y,t)\, U_0(Y)\, d\omega_a(Y),
\end{equation}
\begin{equation}\label{E:Da}
\mathbb D_a(F)(X,t) := \int_0^t\!\!\int_{\RN} \mathbb S_a(X,Y,t-\tau)\, F(Y,\tau)\, d\omega_a(Y)\, d\tau .
\end{equation}

By the tensorization \eqref{Sa} of the kernel, $\mathbb T^\star_a$ is the unitary group
\begin{equation}\label{E:Tstar-group}
\mathbb T^\star_a(\cdot)(\cdot,\cdot,t) = e^{i t\left(\Delta_x + \BaN\right)}
\qquad\text{on } L^2_a(\RN),
\end{equation}
where $\BaN$ denotes the following realization of the Bessel operator. For
$-1<a<1$, the two solutions of $\Ba u=0$ are $u_1\equiv1$ and
$u_2(z)=z^{1-a}$, and both belong to $L^2((0,\ve),z^adz)$.
Thus $z=0$ is in the limit-circle case in
the whole range $-1<a<1$ relevant here, so that $\Ba$ is \emph{not} essentially
self-adjoint and a boundary condition must be imposed; the point at infinity is
always limit-point. Imposing the homogeneous Neumann condition
\begin{equation}\label{E:neumann}
\underset{z\to0^+}{\lim}\ z^a\p_z u = 0,
\end{equation}
which eliminates $u_2$, singles out a self-adjoint realization, and it is this
realization that we denote by $\BaN$; see \cite[(1.9)-(1.11)]{GS1}, where the
Weyl analysis is carried out, and where it is also shown that $e^{it\BaN}$ is
the unitary group whose kernel is $S_a$ of \eqref{SaS}. Identity
\eqref{Sazeron} is precisely the statement that this group respects
\eqref{E:neumann}. We shall use the diagonalization of $\BaN$ furnished by the
modified Hankel transform of order $\nu=\frac{a-1}2$ on $(\R^+,z^adz)$,
\begin{equation}\label{E:hankel}
\mathcal H_{\nu}(\psi)(\lambda)
= \int_0^\infty (\lambda z)^{-\nu}J_{\nu}(\lambda z)\,\psi(z)\,z^a\,dz ,
\end{equation}
see \cite[(2.3)]{GS1}. By \cite[Theorem~2.1]{GS1} the transform $\mathcal
H_\nu$ is an involution, $\mathcal H_\nu\circ\mathcal H_\nu=\operatorname{Id}$,
and by \cite[(2.7)]{GS1} it satisfies the Plancherel identity on
$L^2(\R^+,z^adz)$; it is therefore a surjective isometry of that space onto
itself. Moreover, by \cite[Theorem~2.2]{GS1},
\begin{equation}\label{E:hankeldiag}
\mathcal H_{\nu}\bigl(\BaN\psi\bigr)(\lambda)=-\lambda^2\,\mathcal H_{\nu}(\psi)(\lambda)
\end{equation}
for every $\psi$ in the space $\mathscr S^+$ introduced in \cite[Section~1]{GS1},
namely the space of the $\vf\in C^\infty(\R^+)$ for which all the seminorms
\[
\gamma_{m,k}(\vf)=\sup_{z>0}\bigl|z^m(\tfrac1z\p_z)^k\vf(z)\bigr|
\]
are finite,
equivalently the restrictions to $(0,\infty)$ of the even Schwartz functions on
$\R$. As observed in \cite{GS1}, the finiteness of $\gamma_{0,1}$ alone forces
$|\p_z\vf(z)|\le\gamma_{0,1}(\vf)\,z$, so that every element of $\mathscr S^+$
automatically satisfies \eqref{E:neumann} when $a>-1$; in particular
$\mathscr S^+\subset\Dom(\BaN)$, and $\mathcal H_\nu$ diagonalizes precisely
the realization selected by that condition. Writing
\[
\widehat{v}(\xi) = \int_{\Rd} e^{-2\pi i \sa \xi,x\da} v(x)\, dx
\]
for the Fourier transform in the tangential variable alone, it follows from \eqref{E:hankeldiag} and from the classical propagator \eqref{Scla} that, on the Fourier--Hankel side, $\mathbb T^\star_a$ acts as multiplication by the unimodular symbol
\begin{equation}\label{E:multiplier}
e^{-i t\left(4\pi^2|\xi|^2+\lambda^2\right)} .
\end{equation}

With the above Neumann convention, see \cite[(2.5)]{GS2} and the verification in \cite[Remark~2.1]{GS2}, the boundary Duhamel operator carries the factor $-i$. We therefore use
\begin{equation}\label{Theta}
\Theta_a^\star(\Phi)(x,z,t) = -i \int_0^{t} S_a(z,0,\sigma) S(\sigma)(\Phi(\cdot,t-\sigma))(x)  d\sigma.
\end{equation}

\begin{theorem}\label{T:id}
Let $\Theta^\star_a(X,Y_0,\sigma)$ denote the integral kernel of the boundary operator $\Theta_a^\star$ in \eqref{Theta}. Then 
\[
\Theta^\star_a(X,Y_0,t) = -\frac{2^{2s-1}e^{i\frac{\pi s}2}\G(s)}{\G(1-s)}\ \mathscr E_s(X,Y_0,t),
\]
where $\mathscr E_s(X,Y_0,t)$ is the Poisson lifting of the fundamental solution $ E_s(x,y,t)$ of $\Le^s$.
\end{theorem}

\begin{proof}

For $a=1-2s$, formula \cite[(3.9)]{GS2} gives
\begin{equation}\label{SaS0}
S_a(z,0,\sigma)
= \frac{2^{2s-1}}{\G(1-s)}
\frac{e^{-i \frac{(1-s)\pi}{2}}}{\sigma^{1-s}}
\,e^{i \frac{z^2}{4\sigma}}.
\end{equation}

Substituting \eqref{SaS0} in \eqref{Theta}, we obtain
\begin{equation}\label{Theta2}
\Theta_a^\star(\Phi)(x,z,t) = -\frac{2^{2s-1}e^{i\frac{\pi s}2}}{\G(1-s)} \int_0^{t} \frac{e^{i\frac{z^2}{4\sigma}}}{\sigma^{1-s}} S(\sigma)(\Phi(\cdot,t-\sigma))(x)  d\sigma.
\end{equation}

The sign is also consistent with the normalization used later in the nonlinear problem. Indeed, with
\[
c_s=-\frac{2^{1-2s}\G(1-s)}{\G(s)e^{i\pi s/2}},
\]
one has
\[
c_s\left(-\frac{2^{2s-1}e^{i\pi s/2}\G(s)}{\G(1-s)}\right)=1.
\]

Comparing \eqref{Theta2} with Proposition \ref{P:Es} and with the first term in the representation  \eqref{pot2}, we conclude that the kernel of the operator $\Theta_a^\star$ is a multiple of the Poisson extension of the fundamental solution of $\Le^s$. 

\end{proof}


\section{The natural history space}\label{S:history}

To apply the half-space theory of \cite{GS2}, the prescribed
history $u_0$ must have a Poisson lifting
$\mathcal P_su_0\in L_a^2(\RN)$, where $\mathcal P_s$ is defined by \eqref{lift}. In this section we express
the squared norm of this lifting as a quadratic form in
$u_0$ and define the history space $\mathcal H^s$ by completion.
We begin by showing that every $u_0\in\Sm$ has a lifting
in $L_a^2(\RN)$.

\begin{lemma}\label{L:U0L2}
Let $0<s<1$ and set $a=1-2s\in(-1,1)$. Assume that
$u_0\in \Sm$ and define
\begin{equation}\label{lift2}
U_0(X) = \mathcal P_s u_0(X) := \int_0^\infty\int_{\R^d}
\mathbb P_a(X,Y_0,\sigma)\,u_0(y,-\sigma)\,dy\,d\sigma.
\end{equation}
Then $U_0\in L_a^2(\R^{d+1}_+)$.
\end{lemma}

\begin{proof}
We want to show that
\begin{equation}\label{goal}
\int_0^\infty\int_{\Rd}|U_0(x,z)|^2 z^a\,dx\,dz<\infty.
\end{equation}
 Let $v_0\in \mathscr S(\RNu)$ be such that $v_0|_{\{t\le 0\}} = u_0$. Write
\[
G(\sigma)(x):=S(\sigma)v_0(\cdot,-\sigma)(x),
\qquad \sigma>0,
\]
so that, by the tensorization property of the Poisson kernel \eqref{Pa},
\[
U_0(X)
=
c_a\,z^{2s}\int_0^\infty \sigma^{-s}
e^{\,i\frac{z^2}{4\sigma}}\,G(\sigma)(x)\frac{d\sigma}{\sigma},
\]
for a constant $c_a\neq 0$ depending only on $a$.
Since $v_0\in\mathscr S(\RNu)$ and $S(\sigma)$ is unitary on $L^2(\Rd)$, we have $\|G(\sigma)\|_{L^2_x}=\|v_0(\cdot,-\sigma)\|_{L^2_x}$ and $\|G'(\sigma)\|_{L^2_x}=\|(\Le v_0)(\cdot,-\sigma)\|_{L^2_x}$, both rapidly decreasing in $\sigma$, and therefore
$\sigma\mapsto G(\sigma)\in L^2(\Rd)$ is smooth and rapidly decaying together
with all its derivatives. In particular, for every $N\ge 0$ there exists
$C_N>0$ such that
\[
\|G(\sigma)\|_{L^2_x}+\|G'(\sigma)\|_{L^2_x}
\le C_N(1+\sigma)^{-N}
\qquad\forall\,\sigma>0.
\]
Indeed,
\begin{equation}\label{Gprime}
G'(\sigma)
=
\partial_\sigma\bigl(S(\sigma)v_0(\cdot,-\sigma)\bigr)
=
i\Delta_xS(\sigma)v_0(\cdot,-\sigma)
-
S(\sigma)\partial_t v_0(\cdot,-\sigma)
=
-\,S(\sigma)\bigl(\Le v_0\bigr)(\cdot,-\sigma),
\end{equation}
and the claim follows from the unitarity of
$S(\sigma)$ on $L^2(\R^d)$ and the Schwartz decay of $v_0$.
To reach the desired conclusion, we estimate $\|U_0(\cdot,z)\|_{L^2_x}$ separately for $0<z\le 1$ and $z\ge 1$.

\medskip
\noindent \underline{Step 1}: $0<z\le 1$.
Making the change of variables
$u=\frac{z^2}{4\sigma}$,
we obtain
\begin{equation}\label{split}
U_0(x,z)
=
c_a\,z^{2s} \int_0^\infty \left(\frac{4u}{z^2}\right)^s e^{iu}\  
G\left(\frac{z^2}{4u}\right)(x) \frac{du}{u} = \tilde c_a \int_0^\infty u^{s-1} e^{iu}\ G\left(\frac{z^2}{4u}\right)(x) du,
\end{equation}
for another harmless constant $\widetilde c_a\neq 0$ depending only on $a$.
Set
\[
H_z(u)(x):=G\!\left(\frac{z^2}{4u}\right)(x).
\]
Then
\[
\|H_z(u)\|_{L^2_x}\le \sup_{\sigma>0}\|G(\sigma)\|_{L^2_x},
\]
and, for $0<z\le 1$, we have from \eqref{Gprime}
\begin{equation}\label{Hprime}
\|H_z'(u)\|_{L^2_x}
=
\frac{z^2}{4u^2}\,
\left\|G'\!\left(\frac{z^2}{4u}\right)\right\|_{L^2_x}
\le \frac{C}{u^2},
\end{equation}
with $C$ independent of $z\in(0,1]$.
We now split the integral in \eqref{split} into $(0,1)$ and $(1,\infty)$.
On $(0,1)$ we simply use the bound on $H_z$ and the fact that $s\in(0,1)$:
\[
\left\|\int_0^1 u^{s-1}e^{iu}H_z(u)\,du\right\|_{L^2_x}
\le
\sup_{\sigma>0}\|G(\sigma)\|_{L^2_x}\int_0^1 u^{s-1}\,du
<\infty.
\]
On the interval $(1,R)$, where $R>1$, we integrate by parts using
$\partial_u e^{iu}=ie^{iu}$:
\[
\begin{aligned}
\int_1^R u^{s-1}e^{iu}H_z(u)\,du
&=
\left[\frac{1}{i}u^{s-1}e^{iu}H_z(u)\right]_1^R  \\
&\qquad
-\frac1i
\int_1^R
e^{iu}
\Bigl((s-1)u^{s-2}H_z(u)
+
u^{s-1}H_z'(u)
\Bigr)\,du .
\end{aligned}
\]
Since $0<s<1$ and
\[
\sup_{u>0}\|H_z(u)\|_{L^2_x}
\le
\sup_{\sigma>0}\|G(\sigma)\|_{L^2_x},
\]
we have
\[
R^{\,s-1}\|H_z(R)\|_{L^2_x}
\longrightarrow0
\qquad\text{as }R\to\infty.
\]
Hence the boundary term at infinity vanishes in $L^2_x$, and letting
$R\to\infty$ yields
\[
\begin{aligned}
\int_1^\infty u^{s-1}e^{iu}H_z(u)\,du
&=
-\frac{e^{i}}{i}H_z(1)
\\
&\qquad
-\frac1i
\int_1^\infty
e^{iu}
\Bigl((s-1)u^{s-2}H_z(u)
+
u^{s-1}H_z'(u)
\Bigr)\,du .
\end{aligned}
\]
The first term $H_z(1)(x) = G\!\left(\frac{z^2}{4}\right)(x)$ is bounded in $L^2_x$, whereas the second one is bounded
because
\[
u^{s-2}\in L^1(1,\infty),
\]
and from \eqref{Hprime} we have
\[
u^{s-1}\|H_z'(u)\|_{L^2_x}
\le
Cu^{s-3},
\]
with
\[
u^{s-3}\in L^1(1,\infty),
\]
since $0<s<1$.
Therefore,
\[
\sup_{0<z\le 1}\|U_0(\cdot,z)\|_{L^2_x}\le C.
\]
It follows that
\[
\int_0^1 z^a\|U_0(\cdot,z)\|_{L^2_x}^2\,dz
\le
C\int_0^1 z^a\,dz<\infty,
\]
since $a>-1$.

\medskip
\noindent \underline{Step 2}: $z\ge 1$.
Starting from the representation
\[
U_0(x,z)
=
c_a\,z^{2s}\int_0^\infty \sigma^{-1-s}
e^{\,i\frac{z^2}{4\sigma}}\,G(\sigma)(x)\,d\sigma,
\]
we integrate by parts in $\sigma$, using
\[
\partial_\sigma e^{\,i\frac{z^2}{4\sigma}}
=
-\frac{i z^2}{4\sigma^2}e^{\,i\frac{z^2}{4\sigma}}.
\]
Since $\sigma^{1-s}G(\sigma)\to 0$ as $\sigma\to 0^+$ and as $\sigma\to\infty$
(the first limit because $s<1$ and $G$ is bounded near $0$, the second by the
rapid decay of $G$), we obtain
\[
U_0(x,z)
=
\frac{4c_a}{iz^2}\,z^{2s}
\int_0^\infty
\partial_\sigma\!\bigl(\sigma^{1-s}G(\sigma)\bigr)(x)\,
e^{\,i\frac{z^2}{4\sigma}}\,d\sigma.
\]
Hence
\[
\|U_0(\cdot,z)\|_{L^2_x}
\le
C z^{2s-2}
\int_0^\infty
\left\|\partial_\sigma\!\bigl(\sigma^{1-s}G(\sigma)\bigr)\right\|_{L^2_x}\,d\sigma.
\]
Now
\[
\partial_\sigma\!\bigl(\sigma^{1-s}G(\sigma)\bigr)
=
(1-s)\sigma^{-s}G(\sigma)+\sigma^{1-s}G'(\sigma),
\]
and the right-hand side is integrable in $\sigma$ with values in $L^2_x$
because $s\in(0,1)$ and $G,G'$ are rapidly decaying. Therefore
\[
\|U_0(\cdot,z)\|_{L^2_x}\le C z^{2s-2},
\qquad z\ge 1.
\]
Consequently,
\[
\int_1^\infty z^a\|U_0(\cdot,z)\|_{L^2_x}^2\,dz
\le
C\int_1^\infty z^{a+4s-4}\,dz
=
C\int_1^\infty z^{2s-3}\,dz<\infty,
\]
since $0<s<1$.
Combining the estimates on $(0,1)$ and $(1,\infty)$, we conclude that \eqref{goal} does hold. This proves the lemma.

\end{proof}

To compute the history energy, we take the Fourier transform in $x$
and set, for $u_0\in\Sm$ and $\xi\in\Rd$,
\begin{equation}\label{E:Gxi}
g_\xi(\sigma):=e^{-4\pi^2i\sigma|\xi|^2}\,
\widehat{u_0}(\xi,-\sigma)\,\mathbf 1_{(0,\infty)}(\sigma),
\qquad
\mathscr G_\xi(\omega):=\int_{\R}e^{-i\omega\sigma}
g_\xi(\sigma)\,d\sigma,
\end{equation}
where $\widehat{u_0}(\xi,t)$ denotes the Fourier transform in $x$.
For $\sigma>0$, $g_\xi(\sigma)$ is the Fourier transform of
$S(\sigma)u_0(\cdot,-\sigma)$ evaluated at $\xi$.
For each fixed $\xi$, it vanishes on $(-\infty,0)$, has a finite
right limit at the origin, and decays rapidly as $\sigma\to\infty$.

\begin{lemma}[Mellin reduction of the Poisson energy]\label{L:MellinReduction}
Let $a=1-2s$, $0<s<1$, and let $u_0\in\Sm$. Then, for a.e. $\xi$,
\begin{equation}\label{E:MellinReduction}
\int_0^\infty|\widehat U_0(\xi,z)|^2z^a\,dz
=
\frac{2^{1-2s}}{\G(s)^2}
\int_0^\infty \rho^s|\mathfrak G_\xi(\rho)|^2\,d\rho,
\end{equation}
where
\begin{equation}\label{E:mathfrakG}
\mathfrak G_\xi(\rho):=
\int_0^\infty e^{i\rho/\sigma}\sigma^{-1-s}g_\xi(\sigma)\,d\sigma.
\end{equation}
\end{lemma}

\begin{proof}
Taking the Fourier transform in $x$ in the Poisson lifting gives
\[
\widehat U_0(\xi,z)
=c_a z^{2s}\int_0^\infty \sigma^{-1-s}
 e^{iz^2/(4\sigma)}g_\xi(\sigma)\,d\sigma
=c_a z^{2s}\mathfrak G_\xi\!\left(\frac{z^2}{4}\right),
\]
where $|c_a|=(4^s\G(s))^{-1}$. Hence, with $\rho=z^2/4$,
\[
\begin{aligned}
\int_0^\infty|\widehat U_0(\xi,z)|^2z^a\,dz
&=|c_a|^2\int_0^\infty z^{1+2s}
\left|\mathfrak G_\xi\!\left(\frac{z^2}{4}\right)\right|^2dz\\
&=\frac{2^{1-2s}}{\G(s)^2}
\int_0^\infty\rho^s|\mathfrak G_\xi(\rho)|^2\,d\rho.
\end{aligned}
\]
\end{proof}

For a function $h$ on $(0,\infty)$ define
\[
(Wh)(\alpha):=\alpha^{s-1}h(1/\alpha),\qquad \alpha>0,
\]
so that $W^2=\mathrm{Id}$. We also write
\[
E_\pm[h]:=\int_0^\infty\omega^s|\widehat h(\pm\omega)|^2\,d\omega,
\qquad
\widehat h(\omega):=\int_0^\infty e^{-i\omega\sigma}h(\sigma)\,d\sigma,
\]
whenever these quantities are finite, with the Fourier transforms understood as
Abel boundary values when necessary.

\begin{lemma}[Mellin transform under inversion]\label{L:MellinW}
For every $h$ for which the Mellin transforms below exist,
\[
\mathscr M(Wh)(\lambda)=\mathscr Mh(1-s-\lambda).
\]
Moreover,
\[
\mathfrak G_\xi(\rho)=\widehat{Wg_\xi}(-\rho).
\]
\end{lemma}

\begin{proof}
The first identity follows from $\sigma=1/\alpha$:
\[
\mathscr M(Wh)(\lambda)
=\int_0^\infty \alpha^{s-1}h(1/\alpha)\alpha^{\lambda-1}\,d\alpha
=\int_0^\infty h(\sigma)\sigma^{1-s-\lambda}\frac{d\sigma}{\sigma}.
\]
The second follows from the same substitution:
\[
\mathfrak G_\xi(\rho)
=\int_0^\infty e^{i\rho\alpha}(Wg_\xi)(\alpha)\,d\alpha
=\widehat{Wg_\xi}(-\rho).
\]
\end{proof}

\begin{lemma}[The inversion identity for the one-sided energy]\label{L:MellinSwap}
Let $0<s<1$, and let $h\in C^\infty([0,\infty))$ be rapidly decreasing at infinity together with all its derivatives. Then
\begin{equation}\label{E:MellinSwap}
E_-[Wh]=E_+[h].
\end{equation}
\end{lemma}

\begin{proof}
We first justify the Mellin identity by an explicit Abel regularization. For
$0<\Re\mu<1$ and $\delta>0$, set
\[
F_{\pm,\delta}(\omega):=e^{-\delta\omega}\widehat h(\pm\omega),\qquad \omega>0.
\]
For fixed $\delta>0$, the double integral is absolutely
convergent, since
\[
\int_0^\infty\!\int_0^\infty
\omega^{\Re\mu-1}e^{-\delta\omega}|h(\sigma)|
\,d\sigma\,d\omega
=\Gamma(\Re\mu)\delta^{-\Re\mu}\|h\|_{L^1(0,\infty)}<\infty.
\]
We may therefore apply Fubini's theorem. Evaluating the
inner integral by the Gamma integral,
\[
\int_0^\infty\omega^{\mu-1}e^{-\delta\omega}
e^{\mp i\omega\sigma}\,d\omega
=\Gamma(\mu)(\delta\pm i\sigma)^{-\mu},
\]
we obtain
\[
\mathscr M[F_{\pm,\delta}](\mu)
=\G(\mu)\int_0^\infty h(\sigma)(\delta\pm i\sigma)^{-\mu}
\,d\sigma .
\]
For fixed $\mu$ on any compact substrip of $0<\Re\mu<1$, the integrand is
bounded by $C\,|h(\sigma)|\sigma^{-\Re\mu}$, which is integrable.
Dominated convergence as $\delta\downarrow0$ thus yields
\begin{equation}\label{E:MellinFourier}
\mathscr M[\widehat h(\pm\cdot)](\mu)
=\G(\mu)e^{\mp i\pi\mu/2}\mathscr Mh(1-\mu),
\qquad 0<\Re\mu<1,
\end{equation}
where the left-hand side is the Mellin transform of the Abel boundary value.

We also justify the use of \eqref{E:MellinFourier} for $k=Wh$, which need not be integrable at infinity. Integration by parts gives
\[
|\widehat h(\omega)|\le C\min\{1,|\omega|^{-1}\}.
\]
At zero, $k$ and all its derivatives vanish to every order; at infinity,
\[
|k(\alpha)|\le C\alpha^{s-1},\qquad
|k'(\alpha)|\le C\alpha^{s-2}.
\]
Thus $k'\in L^1(0,\infty)$, and its oscillatory Fourier integral exists for every nonzero frequency. Splitting that integral at $\alpha=|\omega|^{-1}$ and integrating its tail by parts when $0<|\omega|\le1$, and integrating the whole integral by parts when $|\omega|\ge1$, gives
\[
|\widehat k(\omega)|\le
\begin{cases}
C|\omega|^{-s},&0<|\omega|\le1,\\
C|\omega|^{-1},&|\omega|\ge1.
\end{cases}
\]
These estimates imply $E_\pm[h],E_\pm[k]<\infty$. To establish the Mellin formula for $k$, choose a smooth cutoff $\chi$ equal to one on $[0,1]$ and zero on $[2,\infty)$, and set $k_R(\alpha)=\chi(\alpha/R)k(\alpha)$, $R\ge1$. The same Fourier estimates hold uniformly in $R$, and $\widehat k_R(\omega)\to\widehat k(\omega)$ for $\omega\ne0$. Apply \eqref{E:MellinFourier} to $k_R$ and let $R\to\infty$. Dominated convergence is valid on the Fourier side for $s<\Re\mu<1$ by the displayed bounds, and on the physical side because
$\int_0^\infty |k(\alpha)|\alpha^{-\Re\mu}\,d\alpha<\infty$
in this strip. Therefore \eqref{E:MellinFourier} holds for $k$ on that strip. In particular it holds on $\Re\mu=(1+s)/2$, and Mellin--Plancherel gives
\[
E_\pm[h]
=\frac1{2\pi}\int_\R
\left|\G\!\left(\frac{1+s}{2}+i\eta\right)\right|^2
 e^{\pm\pi\eta}
\left|\mathscr Mh\!\left(\frac{1-s}{2}-i\eta\right)\right|^2d\eta.
\]
For $Wh$, Lemma~\ref{L:MellinW} gives
\[
\mathscr M(Wh)(1-\mu)=\mathscr Mh(\mu-s),
\]
so that, on the same Mellin line,
\[
E_-[Wh]
=\frac1{2\pi}\int_\R
\left|\G\!\left(\frac{1+s}{2}+i\eta\right)\right|^2
 e^{-\pi\eta}
\left|\mathscr Mh\!\left(\frac{1-s}{2}+i\eta\right)\right|^2d\eta.
\]
Replacing $\eta$ by $-\eta$ and using
$|\G(c-i\eta)|=|\G(c+i\eta)|$ transforms the last expression into $E_+[h]$.
This proves \eqref{E:MellinSwap}.
 
\end{proof}

Before stating the next result, we recall that, given a history  $u_0\in\Sm$, writing $\widehat{u_0}(\xi,t)$ for its
Fourier transform in $x$, formula \eqref{E:Gxi} gives
\[
\mathscr G_\xi(\omega)
=
\int_0^\infty
e^{-i\omega\sigma-4\pi^2i\sigma|\xi|^2}
\widehat{u_0}(\xi,-\sigma)\,d\sigma.
\]

\begin{theorem}[History energy identity]\label{T:history}
Let $0<s<1$, set $a=1-2s$, and let $u_0\in\Sm$.
Let $U_0=\mathcal P_su_0$ be the Poisson lifting
\eqref{lift2}. Then
\begin{equation}\label{E:freq}
\|U_0\|_{L_a^2(\RN)}^2
=
\frac{2^{1-2s}}{\G(s)^2}
\int_{\Rd}\int_0^\infty
|\mathscr G_\xi(\omega)|^2 \omega^s d\omega\,d\xi.
\end{equation}
\end{theorem}

\begin{proof}
By Plancherel in $x$ and Tonelli,
\[
\|U_0\|_{L_a^2}^2
=\int_{\Rd}\int_0^\infty|\widehat U_0(\xi,z)|^2z^a\,dz\,d\xi.
\]
Lemma~\ref{L:MellinReduction} and Lemma~\ref{L:MellinW} give the right-hand
side as
\[
\frac{2^{1-2s}}{\G(s)^2}\int_{\Rd}E_-[Wg_\xi]d\xi.
\]
Lemma~\ref{L:MellinSwap} identifies $E_-[Wg_\xi]=E_+[g_\xi]$, which is precisely
\eqref{E:freq}.

\end{proof}

We next define the history energy as the squared weighted
$L^2$ norm of the Poisson lifting.

\begin{definition}\label{D:Qs} 
For $u_0\in\Sm$, let $U_0$ be the Poisson lifting defined by
\eqref{lift2}. We set
\[
Q_s(u_0):=\|U_0\|_{L_a^2(\RN)}^2.
\]
\end{definition}

\begin{lemma}[One-sided Fourier uniqueness]\label{L:onesided-Fourier}
Let $g\in L^2(\R)$ be supported in $[0,\infty)$.
If $\widehat g$ vanishes on a measurable subset of $\R$
of positive measure, then $g=0$ almost everywhere.
In particular, if $\widehat g(\omega)=0$ for almost
every $\omega>0$, then $g=0$.
\end{lemma}

\begin{proof}
For $\zeta\in\mathbb C_-=\{\zeta\in\C:
\operatorname{Im}\zeta<0\}$, define
\[
F(\zeta)=\int_0^\infty
e^{-2\pi i\zeta\sigma}g(\sigma)\,d\sigma.
\]
This integral defines a holomorphic function in
$\mathbb C_-$. By Plancherel's theorem,
\[
\int_{\R}|F(\omega-i\varepsilon)|^2\,d\omega
=
\int_0^\infty e^{-4\pi\varepsilon\sigma}
|g(\sigma)|^2\,d\sigma
\le \|g\|_{L^2(\R)}^2,
\qquad \varepsilon>0.
\]
Thus $F$ belongs to the Hardy space $H^2(\mathbb C_-)$,
consisting of holomorphic functions whose $L^2$ norms
on horizontal lines are uniformly bounded.
Moreover,
\[
\|F(\cdot-i\varepsilon)-\widehat g\|_{L^2(\R)}
=
\|(e^{-2\pi\varepsilon\sigma}-1)g(\sigma)\|_{L^2(\R)}
\longrightarrow0
\qquad\text{as }\varepsilon\downarrow0.
\]
Hence $\widehat g$ is the boundary function of $F$.
Since $\widehat g$ vanishes on a set of positive measure,
the boundary uniqueness theorem for Hardy functions gives
$F\equiv0$; see \cite[Chapter~III, Section~D, and Chapter~VI]{Koosis}.
Therefore $\widehat g=0$, and Plancherel's theorem yields $g=0$.

\end{proof}

\begin{corollary}\label{C:positive}
For $u_0\in\Sm$, $Q_s(u_0)=0$ implies $u_0=0$ almost everywhere. Hence
$Q_s^{1/2}$ is a norm on $\Sm$.
\end{corollary}

\begin{proof}
Recall that
\[
g_\xi(\sigma)
=
e^{-4\pi^2i\sigma|\xi|^2}
\widehat{u_0}(\xi,-\sigma)
\mathbf1_{(0,\infty)}(\sigma).
\]
If $Q_s(u_0)=0$, the history-energy identity
\eqref{E:freq}, after the change of variable
$\omega=2\pi\tau$, gives
\[
\widehat{g_\xi}(\tau)=0
\qquad\text{for almost every }
(\xi,\tau)\in\Rd\times(0,\infty).
\]
For almost every $\xi$, the function $g_\xi$ belongs
to $L^2(\R)$ and is supported in $[0,\infty)$.
Lemma~\ref{L:onesided-Fourier} therefore implies
$g_\xi=0$. Since the exponential factor in its
definition never vanishes, we obtain
\[
\widehat{u_0}(\xi,-\sigma)=0
\qquad\text{for almost every }
(\xi,\sigma)\in\Rd\times(0,\infty).
\]
Plancherel's theorem in the spatial variables gives
$u_0=0$ almost everywhere on the past half-space.
The remaining norm properties follow from
$Q_s(u_0)^{1/2}=\|U_0\|_{L_a^2(\RN)}$
and the linearity of the Poisson lifting.

\end{proof}

\begin{definition}\label{D:history}
We denote by $\mathcal H^s$ the completion of $\Sm$ with respect to the norm
$Q_s^{1/2}$.
\end{definition}

\begin{corollary}\label{C:extends}
The oscillatory Poisson lifting
$\mathcal P_s:\Sm\to L_a^2(\RN)$ extends uniquely to an isometry
\[
\mathcal P_s:\mathcal H^s\longrightarrow L_a^2(\RN).
\]
Indeed, if $u_0^{(k)}\to u_0$ in $\mathcal H^s$, then
$\mathcal P_su_0^{(k)}$ is Cauchy in $L_a^2(\RN)$ by the definition of $Q_s$; its limit
is independent of the approximating sequence and satisfies
\[
\|\mathcal P_su_0\|_{L_a^2(\RN)}=\|u_0\|_{\mathcal H^s}.
\]
\end{corollary}
Below we use the space $\mathcal H^s$ through its isometric realization in
$L_a^2(\RN)$. 
We do not need to identify its elements with distributions on past time.

\section{Well-posedness of the nonlinear memory equation}\label{S:memory}
\label{S:wellpos}

In this section we combine the extension theory developed in this work with the estimates for the half-space evolution in \cite{GS2} to establish existence,
uniqueness and continuous dependence for the memory problem
\begin{equation}\label{E:nonlinear-memory}
\left\{
\begin{aligned}
\Le^s u(x,t)
&=
\mu |u(x,t)|^{p-1}u(x,t),
&& x\in\Rd,\quad t>0,
\\
u(x,t)
&=
u_0(x,t),
&& x\in\Rd,\quad t\leq0,
\end{aligned}
\right.
\end{equation}
where $d\ge1$ and
\[
u_0\in\mathcal H^s,\qquad \mu\in\C,\qquad p>1.
\]
By Corollary~\ref{C:extends}, the initial datum
$U_0=\mathcal P_su_0\in L_a^2(\RN)$, with $a=1-2s$, satisfies
\begin{equation}\label{E:history-isometry}
\|U_0\|_{L_a^2(\RN)}
=
\|u_0\|_{\mathcal H^s}.
\end{equation}
We use \(U_0\) as initial datum for the  singular
nonlinear Schr\"odinger evolution in the upper half-space:
\begin{equation}\label{cp02}
\begin{cases}
\p_t U - i(\Delta_x U + \Ba U) = 0, & \ \ X=(x,z)\in \mathbb{R}^{d+1}_+,\ t>0, \\
-\frac{2^{2s-1}\Gamma(s)e^{i \frac{\pi s}2}}
{\Gamma(1-s)} \displaystyle \lim_{z\to 0^+} z^a \p_z U(X,t) = \mu |U(x,0,t)|^{p-1}U(x,0,t), & \ \ x\in \Rd,\ t>0, \\
U(X,0)=U_0(X). 
\end{cases}
\end{equation}

Up to the explicit normalization in the boundary condition,
\eqref{cp02} is the nonlinear half-space problem studied in
\cite{GS2}. We first establish the strong trace and integral
formulas needed for its boundary interpretation. We then use
the boundary estimates throughout $-1<a<1$ to prove
Theorem~\ref{T:memory-boundary-full}, and conclude with
convergence to the classical nonlinear Schr\"odinger equation.

\subsection{Strong boundary traces}

The mass-preserving scaling
$U_\lambda(x,z,t)=\lambda^{(d+2-2s)/2}U(\lambda x,\lambda z,\lambda^2t)$
preserves the nonlinear boundary condition when
\[
p=1+\dfrac{4s}{d+2-2s}.
\]
 To construct solutions with strong boundary traces, we apply
the well-posedness results of \cite{GS2} to the lifted history
$U_0=\mathcal P_su_0$. When $0<s\le\frac12$, we have
$0\le a<1$, and \cite[Theorem~2.9]{GS2} applies throughout
the range
\[
1<p\le 1+\frac{4s}{d+2-2s}.
\]
Thus the upper endpoint agrees with the scaling exponent
computed above. When $\frac12<s<1$, we have $-1<a<0$,
and \cite[Theorem~2.10]{GS2} establishes well-posedness
in the smaller range
\[
1<p\le 1+\frac{2}{d+1}.
\]
This restriction comes from the use of two different time
exponents in the Strichartz estimates of
\cite[Theorem~2.5]{GS2}, on which that existence argument relies.
At the respective upper endpoints, the existence assertions
of \cite[Theorems~2.9(1) and~2.10(1)]{GS2} require small initial
data, with a threshold depending on the prescribed interval
in Theorem~2.10(1). The local existence result for arbitrary
histories at $p=p_*(s)$ will be proved below in
Theorem~\ref{T:memory-boundary-full}.
In the latter range,
``critical'' refers to the threshold of that argument, rather than
to the scaling exponent. Accordingly, we set
\begin{equation}\label{E:critical-exponents}
p_c=\begin{cases}
1+\frac{4s}{d+2-2s},&0<s\le\frac12,\\[2mm]
1+\frac{2}{d+1},&\frac12<s<1.
\end{cases}
\end{equation}
We stress that the notation $p_c$ is used only for this auxiliary strong
trace construction. Our principal existence result,
Theorem~\ref{T:memory-boundary-full}, treats the full range
$1<p\le p_*(s)=1+4s/(d+2-2s)$ for every $0<s<1$.
The compatibility of the two constructions is proved after Theorem~\ref{T:memory-boundary-full} has been established.

Here and below, given a Banach space $E$, we write $C^b_z([0,\infty);E)$ for
the space of bounded maps $z\mapsto F(\cdot,z,\cdot)$ from $[0,\infty)$ into
$E$ which are continuous up to and including $z=0$, normed by the supremum
over $z\ge0$; it is a closed subspace of $L^\infty_z E$, and membership in it
is what gives a pointwise meaning to the boundary value at $z=0$.

\begin{lemma}[Boundary trace of the free bulk evolution]\label{L:free-trace}
Let $I=[0,T]$ with $T<\infty$, let $2\le r<\infty$, and
let $q>2$ satisfy \eqref{E:admissible-positive} when
$0\le a<1$, or let $q,q_\infty$ satisfy
\eqref{E:admissible-negative} when $-1<a<0$.
Let $U_0\in L_a^2(\RN)$. If $0\le a<1$, then
\[
\mathbb T_a^\star(U_0)\in C^b_z\bigl([0,\infty);L_t^q(I;L_x^r(\Rd))\bigr);
\]
in particular $z\mapsto\mathbb T_a^\star(U_0)(\cdot,z,\cdot)$ extends
continuously to $z=0$ as an $L_t^q(I;L_x^r)$-valued map. If $-1<a<0$, the same
conclusion holds for the weighted free evolution
$k(z)\,\mathbb T_a^\star(U_0)$, in the mixed-norm framework of
\cite[Theorem~2.5]{GS2}; since $k(z)=1$ for $0\le z\le1$, this yields the same
boundary trace in $L_t^q(I;L_x^r)$.
\end{lemma}

\medskip

This is the free-evolution assertion of \cite[Theorem~8.2]{GS2}.
We give the trace argument on a finite time interval.

\begin{proof}
\smallskip
\noindent \textbf{1.} \emph{A Neumann-compatible dense class.} We first work with data in
\[
\mathcal D_a:=\operatorname{span}
\Bigl\{\varphi(x)\,\psi(z)\ \Big|\
\varphi\in\mathscr S(\Rd),\
\mathcal H_{\frac{a-1}2}(\psi)\in C_0^\infty((0,\infty))\Bigr\},
\]
where, for $\Re \nu>-1$, $\mathcal H_\nu$ is the Hankel transform
\eqref{E:hankel}. 
We use the following three properties.

(a) For each fixed $x$, the functions in $\mathcal D_a$
belong to $\mathscr S^+$ as functions of $z$. They therefore
satisfy the homogeneous Neumann condition \eqref{E:neumann},
by the observation following \eqref{E:hankeldiag}.
If $\mathcal H_\nu(\psi)=g\in
C_0^\infty((0,\infty))$, then by the inversion of \cite[Theorem~2.1]{GS1} we
have $\psi=\mathcal H_\nu(g)$, that is
$\psi(z)=\int_0^\infty(\lambda z)^{-\nu}J_\nu(\lambda z)\,g(\lambda)\,\lambda^a\,d\lambda$
with $g$ smooth and compactly supported away from the origin. Since
$w\mapsto w^{-\nu}J_\nu(w)$ is an even entire function, $\psi$ is a smooth
function of $z^2$; and repeated integration by parts in $\lambda$, using the
oscillation of $J_\nu$, shows that $\psi$ and all its derivatives decay
rapidly. Hence $\psi\in\mathscr S^+$, and by the property of that space recalled
after \eqref{E:hankeldiag} the elements of $\mathcal D_a$ satisfy
\eqref{E:neumann} and lie in $\Dom(\BaN)$.

(b) $\mathcal D_a$ is dense in $L^2_a(\RN)$. The transform $\mathcal H_\nu$ is a
surjective isometry of $L^2(\R^+,z^adz)$ onto $L^2(\R^+,\lambda^ad\lambda)$, so
the functions $\psi$ occurring above are dense in $L^2(\R^+,z^adz)$; since
$\mathscr S(\Rd)$ is dense in $L^2(\Rd)$ and finite sums of tensor products are
dense in the Hilbert space tensor product
$L^2(\Rd)\otimes L^2(\R^+,z^adz)=L^2_a(\RN)$, the claim follows.

(c) $\mathcal D_a$ is invariant under $\Delta_x$, under $\BaN$ and under the
propagator $\mathbb T^\star_a(\cdot)(\cdot,\cdot,t)$. The first two properties are clear
from \eqref{E:hankeldiag} and from $\mathscr S(\Rd)$ being invariant under
$\Delta_x$.  To verify invariance under the propagator, we use
\eqref{E:multiplier} and write
\[
e^{-it(4\pi^2|\xi|^2+\lambda^2)}
=
e^{-4\pi^2it|\xi|^2}\,e^{-it\lambda^2}.
\]
For a product $\varphi(x)\psi(z)$ in the definition of
$\mathcal D_a$, we can therefore treat the two factors
separately. Multiplication of $\widehat\varphi(\xi)$ by
$e^{-4\pi^2it|\xi|^2}$ preserves the Schwartz class:
for each fixed $t$, all derivatives of this multiplier
have at most polynomial growth. Its inverse Fourier
transform thus belongs to $\mathscr S(\Rd)$.
In the Hankel variable, multiplication of
$\mathcal H_\nu(\psi)(\lambda)$ by $e^{-it\lambda^2}$
preserves smoothness and leaves its compact support
in $(0,\infty)$ unchanged. Consequently, the propagator
takes $\varphi(x)\psi(z)$ to another product satisfying
the conditions in the definition of $\mathcal D_a$.
By linearity, $\mathcal D_a$ is invariant under the propagator. 
\smallskip

\smallskip
\noindent \textbf{2.} \emph{A conserved quantity.} We next prove the conservation identity \eqref{E:conserved}.
It gives a bound, uniform in time, for the normal derivative
of the evolution after differentiation in $x$. We will use
this bound to prove continuity in $z$, including at the
boundary $z=0$.
Let $U_0\in\mathcal D_a$,
write $U(\cdot,\cdot,t)=\mathbb T_a^\star(U_0)(\cdot,\cdot,t)$, the operator
$\mathbb T^\star_a$ being the one defined in \eqref{E:Tstar}, see also
\eqref{E:Tstar-group}, and fix
$m\in\mathbb N$ with $4m>d$, so that the Sobolev embedding gives
$\|f\|_{L^r_x}\le C(d,m)\|(1-\Delta_x)^mf\|_{L^2_x}$ for every $2\le r\le\infty$.
Set 
\[
V(t):=(1-\Delta_x)^mU(\cdot,\cdot,t),
\]
 so that $V(t)\in\mathcal D_a$ for
every $t$ by property (c) above. Since $\mathcal H_\nu$ diagonalizes
$\BaN$ by \eqref{E:hankeldiag}, it also diagonalizes the associated
quadratic form: for $\psi\in\Dom(\BaN)$,
\begin{equation}\label{E:formdiag}
\int_0^\infty|\p_z\psi|^2\,z^a\,dz
=\int_0^\infty\lambda^2\,|\mathcal H_\nu(\psi)(\lambda)|^2\,\lambda^a\,d\lambda .
\end{equation}

We now use \eqref{E:formdiag} to show that the weighted $L^2$ norm
of $\p_zV(t)$ is independent of time. Let us first write out the
Fourier--Hankel representation of $V(t)$. The operator
$(1-\Delta_x)^m$ acts on the Fourier transform in $x$ by
multiplication by $(1+4\pi^2|\xi|^2)^m$. Hence \eqref{E:multiplier}
and the identity $V(0)=(1-\Delta_x)^mU_0$ give
\[
\begin{aligned}
\mathcal H_\nu(\widehat{V(t)}(\xi,\cdot))(\lambda)
&=(1+4\pi^2|\xi|^2)^m
  e^{-it(4\pi^2|\xi|^2+\lambda^2)}
  \mathcal H_\nu(\widehat{U_0}(\xi,\cdot))(\lambda)\\
&=e^{-it(4\pi^2|\xi|^2+\lambda^2)}
  \mathcal H_\nu(\widehat{V(0)}(\xi,\cdot))(\lambda).
\end{aligned}
\]
Since $4\pi^2|\xi|^2+\lambda^2$ is real, taking absolute values
removes the exponential factor. We therefore have, for every $t$,
\begin{equation}\label{equalt}
\bigl|\mathcal H_\nu(\widehat{V(t)}(\xi,\cdot))(\lambda)\bigr|^2
=
\bigl|\mathcal H_\nu(\widehat{V(0)}(\xi,\cdot))(\lambda)\bigr|^2.
\end{equation}

To recover the norm of $\p_zV(t)$ from this equality, we first
apply Plancherel in $x$:
\[
\|\p_zV(t)\|_{L^2_a(\RN)}^2
=
\int_{\Rd}\!\int_0^\infty
\bigl|\p_z\widehat{V(t)}(\xi,z)\bigr|^2 z^a\,dz\,d\xi.
\]
For each fixed $\xi$, the function
$z\mapsto\widehat{V(t)}(\xi,z)$ belongs to $\Dom(\BaN)$,
because $V(t)\in\mathcal D_a$. We can thus apply
\eqref{E:formdiag} to the inner integral. Using \eqref{equalt}, we obtain
\begin{align}\label{E:conserved}
\|\p_zV(t)\|_{L^2_a(\RN)}^2
&=\int_{\Rd}\!\int_0^\infty \lambda^2\,
\bigl|\mathcal H_\nu(\widehat{V(t)}(\xi,\cdot))(\lambda)\bigr|^2
\lambda^a\,d\lambda\,d\xi
\\
&=\int_{\Rd}\!\int_0^\infty \lambda^2\,
\bigl|\mathcal H_\nu(\widehat{V(0)}(\xi,\cdot))(\lambda)\bigr|^2
\lambda^a\,d\lambda\,d\xi =\|\p_z(1-\Delta_x)^mU_0\|_{L^2_a(\RN)}^2.
\notag
\end{align}
The last equality follows by applying \eqref{E:formdiag} and
Plancherel in reverse at time $0$. All the integrals are finite
for data in $\mathcal D_a$; their order can be exchanged by
Tonelli's theorem, since the integrands are nonnegative.

The norm of $V(t)$ itself is conserved as well. Indeed,
Plancherel for the Fourier and Hankel transforms gives
\[
\begin{aligned}
\|V(t)\|_{L^2_a(\RN)}^2
&=\int_{\Rd}\!\int_0^\infty
\bigl|\mathcal H_\nu(\widehat{V(t)}(\xi,\cdot))(\lambda)\bigr|^2
\lambda^a\,d\lambda\,d\xi\\
&=\int_{\Rd}\!\int_0^\infty
\bigl|\mathcal H_\nu(\widehat{V(0)}(\xi,\cdot))(\lambda)\bigr|^2
\lambda^a\,d\lambda\,d\xi\\
&=\|(1-\Delta_x)^mU_0\|_{L^2_a(\RN)}^2.
\end{aligned}
\]
Thus both $\|V(t)\|_{L_a^2}$ and $\|\p_zV(t)\|_{L_a^2}$
are independent of $t$.

\smallskip
\noindent \textbf{3.} \emph{Continuity in $z$.} Indeed, for $0\le z<z'<\infty$ and every $t$, we have
\[
\|V(\cdot,z',t)-V(\cdot,z,t)\|_{L^2_x}
\le\int_z^{z'}\|\p_\zeta V(\cdot,\zeta,t)\|_{L^2_x}\,d\zeta
\le\Bigl(\int_z^{z'}\zeta^{-a}d\zeta\Bigr)^{\frac12}
\Bigl(\int_0^\infty\|\p_\zeta V(\cdot,\zeta,t)\|_{L^2_x}^2\zeta^ad\zeta\Bigr)^{\frac12},
\]
by Cauchy--Schwarz with the weight split as $\zeta^{a/2}\cdot\zeta^{-a/2}$.
Since $a<1$ the first factor equals
$\bigl((z'^{\,1-a}-z^{\,1-a})/(1-a)\bigr)^{1/2}$, which is finite for all
$0\le z<z'<\infty$, the singularity of $\zeta^{-a}$ at the origin being
integrable precisely because $a<1$; the second factor is
independent of $t$ by \eqref{E:conserved}. 

We now apply the Sobolev estimate to
$f=U(\cdot,z',t)-U(\cdot,z,t)$. Since
$V=(1-\Delta_x)^mU$, the preceding inequality gives
\[
\begin{aligned}
\|U(\cdot,z',t)-U(\cdot,z,t)\|_{L_x^r}
&\le C(d,m)\,
\|V(\cdot,z',t)-V(\cdot,z,t)\|_{L_x^2}\\
&\le C(d,m)
\left(\frac{z'^{\,1-a}-z^{\,1-a}}{1-a}\right)^{1/2}
\|\p_zV(t)\|_{L_a^2(\RN)}.
\end{aligned}
\]
By \eqref{E:conserved}, the last norm is independent of $t$.
Taking the supremum over $t$, we obtain, for a constant
$C_1=C_1(d,a,m,U_0)>0$,
\begin{equation}\label{E:zHolder}
\sup_{t\in\R}\ \|U(\cdot,z',t)-U(\cdot,z,t)\|_{L^r_x(\Rd)}
\ \le\ C_1\,\bigl|z'^{\,1-a}-z^{\,1-a}\bigr|^{\frac12},
\qquad 0\le z,z'<\infty .
\end{equation}
We also need a bound for $U$ at one fixed positive height,
which we take to be $z=1$. For $1\le\zeta\le2$, the
fundamental theorem of calculus gives
\[
V(\cdot,1,t)
=
V(\cdot,\zeta,t)
-\int_1^\zeta \p_zV(\cdot,z,t)\,dz.
\]
Taking the $L_x^2$ norm, integrating in $\zeta$ over $[1,2]$,
and applying Cauchy--Schwarz, we obtain
\[
\|V(\cdot,1,t)\|_{L_x^2}
\le
\left(\int_1^2\|V(\cdot,z,t)\|_{L_x^2}^2\,dz\right)^{1/2}
+
\left(\int_1^2\|\p_zV(\cdot,z,t)\|_{L_x^2}^2\,dz\right)^{1/2}.
\]
Since $z^a$ is bounded above and below by positive constants
on $[1,2]$, it follows that
\[
\|V(\cdot,1,t)\|_{L_x^2}
\le C_a\bigl(\|V(t)\|_{L_a^2}
+\|\p_zV(t)\|_{L_a^2}\bigr).
\]
Recalling that
$V=(1-\Delta_x)^mU$, we now apply the Sobolev estimate in $x$, to find
\[
\|U(\cdot,1,t)\|_{L_x^r}
\le C(d,m)\|V(\cdot,1,t)\|_{L_x^2}
\le C(d,a,m)\bigl(\|V(t)\|_{L_a^2}
+\|\p_zV(t)\|_{L_a^2}\bigr).
\]
Both norms on the right are conserved, so this bound
is uniform in $t$.
Combining the bound at $z=1$ with \eqref{E:zHolder}, we
conclude that, for every fixed $z\ge0$,
\[
U(\cdot,z,\cdot)\in L^\infty(\R;L^r(\Rd)).
\]
Since $1-a>0$, the right-hand side of \eqref{E:zHolder} tends to $0$ as
$z'\to z$; hence
$z\mapsto U(\cdot,z,\cdot)$ is continuous from $[0,\infty)$ into
$L^\infty_t(\R;L^r_x)$, and therefore into $L^q_t(I;L^r_x)$ as well, because
$\|\cdot\|_{L^q_t(I)}\le T^{1/q}\|\cdot\|_{L^\infty_t}$ on the finite interval
$I=[0,T]$. Note that this part of the argument uses only $-1<a<1$, and so
applies verbatim in the anomalous range.

It follows that, for $U_0\in\mathcal D_a$, the map
$z\mapsto\mathbb T_a^\star(U_0)(\cdot,z,\cdot)$ when $a\ge0$,
and its product with $k(z)$ when $a<0$, belongs to
$C^b_z\bigl([0,\infty);L^q_t(I;L^r_x)\bigr)$. Indeed,
when $0\le a<1$ the homogeneous estimate of \cite[Theorem~2.4]{GS2} bounds its
essential supremum in $z$ by $C\|U_0\|_{L^2_a}$, and for a map which is
continuous on $[0,\infty)$ the essential supremum coincides with the genuine
supremum. When $-1<a<0$ the same reasoning applied to the weighted estimate of
\cite[Theorem~2.5]{GS2} gives
$k\,\mathbb T^\star_a(U_0)\in C^b_z\bigl([0,\infty);L^q_t(I;L^r_x)\bigr)$,
the map $z\mapsto k(z)$ being continuous and positive on $(0,\infty)$ and
identically $1$ on $[0,1]$.

For general $U_0\in L_a^2(\RN)$, choose
$U_0^{(n)}\in \mathcal D_a$ with $U_0^{(n)}\to U_0$ in $L_a^2$.
When $0\le a<1$, the homogeneous estimate of \cite[Theorem~2.4]{GS2}, applied
with $F=0$ and vanishing Neumann datum to the admissible triple
$(q,r,\infty)$ of \eqref{E:admissible-positive} below, gives
\[
\operatorname*{ess\,sup}_{z\ge0}\ \|\mathbb T_a^\star(U_0^{(n)}-U_0^{(j)})(\cdot,z,\cdot)\|_{L_t^q(I;L_x^r)}
\le C\|U_0^{(n)}-U_0^{(j)}\|_{L_a^2}.
\]
Since, by the first part of the proof, each of the maps
$z\mapsto \mathbb T_a^\star(U_0^{(n)}-U_0^{(j)})(\cdot,z,\cdot)$ is continuous
on $[0,\infty)$, the essential supremum coincides with the genuine
supremum, and therefore
\begin{equation}\label{E:unifCauchy}
\sup_{z\ge0}\ \|\mathbb T_a^\star(U_0^{(n)}-U_0^{(j)})(\cdot,z,\cdot)\|_{L_t^q(I;L_x^r)}
\le C\|U_0^{(n)}-U_0^{(j)}\|_{L_a^2}.
\end{equation}
When $-1<a<0$, the corresponding estimate is the weighted one of
\cite[Theorem~2.5]{GS2}: with the two exponents $q,q_\infty$ of
\eqref{E:admissible-negative} below and the weight $k(z)=\min\{1,z^{a/2}\}$, it reads
\[
\operatorname*{ess\,sup}_{z\ge0}\ \bigl\|k(z)\,\mathbb T_a^\star(U_0^{(n)}-U_0^{(j)})(\cdot,z,\cdot)\bigr\|_{\left(L_t^{q}+L^{q_\infty}_t\right)(I;L_x^r)}
\le C\|U_0^{(n)}-U_0^{(j)}\|_{L_a^2},
\]
the target being the sum space $L^q_t(I)+L^{q_\infty}_t(I)$, which on the finite
interval $I$ coincides with $L^q_t(I)$ with equivalent norms, since
$q<q_\infty$. Arguing as above, we obtain \eqref{E:unifCauchy} in this range as
well, for the weighted maps $k\,\mathbb T_a^\star(U_0^{(n)}-U_0^{(j)})$.

In either case the approximating maps are continuous and bounded on
$[0,\infty)$ and form a uniformly Cauchy sequence with values in the Banach
space $L^q_t(I;L^r_x)$. Since $C^b_z\bigl([0,\infty);L^q_t(I;L^r_x)\bigr)$ is
complete, being a closed subspace of $L^\infty_zL^q_t(I;L^r_x)$, the limit lies
in it; and the same estimate applied to $U_0^{(n)}-U_0$ identifies that limit
with $\mathbb T_a^\star(U_0)$ when $0\le a<1$, and with
$k\,\mathbb T_a^\star(U_0)$ when $-1<a<0$. This proves the assertion, the
boundary value in the anomalous range being the unweighted one because
$k\equiv1$ on $[0,1]$.

\end{proof}

\subsection{The boundary formulation}

We first recall explicitly the part of the functional framework of
\cite{GS2} that is needed here. Set
\[
r=p+1.
\]
When \(0\leq a<1\), let \(q>2\) satisfy
\begin{equation}\label{E:admissible-positive}
\frac{2}{q}+\frac{d}{r}
=
\frac{d+a+1}{2}.
\end{equation}
On a time interval \(I=[0,T]\), the mild solutions of \eqref{cp02} belong to
\begin{equation}\label{E:GS2-positive-space}
C\bigl(I;L_a^2(\RN)\bigr)
\cap
C^b_z\bigl([0,\infty);
L_t^q(I;L_x^r(\Rd))\bigr).
\end{equation}
The bounds follow from \cite[Theorems~2.4 and~2.9]{GS2}.
Continuity in $z$ follows from Lemma~\ref{L:free-trace} for the
free term and \cite[Theorem~8.1]{GS2} for the boundary Duhamel
term, as required by \cite[Definition~2.6]{GS2}. For $a\ge0$,
\eqref{E:admissible-positive} is also the boundary admissibility
relation in that theorem. Thus
\begin{equation}\label{E:positive-trace}
U(\cdot,z,\cdot)
\longrightarrow
U(\cdot,0,\cdot)
\quad\text{in }L_t^q(I;L_x^r(\Rd))
\qquad\text{as }z\to0^+.
\end{equation}
When \(-1<a<0\), the theory of \cite{GS2} involves the weight
\[
k(z)=\min\{1,z^{a/2}\}
\]
and the exponents \(q,q_\infty\) determined by the equations:
\begin{equation}\label{E:admissible-negative}
\frac{2}{q_\infty}+\frac{d}{r}
=
\frac{d+a+1}{2},
\qquad
\frac{2}{q}+\frac{d}{r}
=
\frac{d+1}{2},
\qquad q>2.
\end{equation}
The corresponding mild solutions belong to the weighted mixed-norm
classes of \cite[Theorem~2.10]{GS2}. The boundary Duhamel term has a strong
trace in the boundary-admissible space $L_t^{q_\infty}L_x^r$ by
Theorem~8.1 of \cite{GS2}; since $q<q_\infty$ and $I$ is finite, this embeds into
$L_t^q(I;L_x^r)$. Therefore the full mild solution has a strong boundary
trace in the space needed below:
\begin{equation}\label{E:negative-trace}
U(\cdot,z,\cdot)
\longrightarrow
U(\cdot,0,\cdot)
\quad\text{in }L_t^q(I;L_x^r(\Rd))
\qquad\text{as }z\to0^+.
\end{equation}

Below, $U(x,0,t)$ denotes the strong trace in
\eqref{E:positive-trace} or \eqref{E:negative-trace}.
\begin{definition}[Mild solution with a strong boundary trace]
\label{D:memory-solution}
Let $u_0\in\mathcal H^s$, let $U_0=\mathcal P_su_0$, and let $T>0$. Let
$\mu\in\C\setminus\{0\}$ and $1<p\le p_c$, set $r=p+1$, and let $q$ be
determined by \eqref{E:admissible-positive} when $0\le a<1$, respectively by
\eqref{E:admissible-negative} when $-1<a<0$.
A mild solution on $(0,T]$ with prescribed history $u_0$ is the pair $(u_0,u)$,
where
\[
u\in L_t^q((0,T);L_x^r(\Rd))
\]
is the strong boundary trace
\[
u(x,t)=U(x,0,t),\qquad 0<t\le T,
\]
of the mild solution $U$ of the half-space problem \eqref{cp02}, in the sense
of \cite{GS2}, with initial datum $U_0$. The trace is understood in the strong
mixed-norm sense of \eqref{E:positive-trace} or \eqref{E:negative-trace}.
Uniqueness in the boundary mild class will follow from
Theorem~\ref{T:memory-boundary-full}, after the two
constructions have been identified.
\end{definition}

\begin{remark}\label{R:concrete-history}
For $u_0\in\Sm$, the pair $(u_0,u)$ defines a function on
$\Rd\times(-\infty,T]$ by
\[
\widetilde u(x,t)=
\begin{cases}
u_0(x,t),&t\le0,\\
u(x,t),&0<t\le T.
\end{cases}
\]
Definition~\ref{D:memory-solution} requires no matching trace at
$t=0$. For general elements of $\mathcal H^s$, it uses only the
lifting, without pointwise values in the past.
\end{remark}

\begin{proposition}[Boundary Duhamel term]
\label{P:boundary-Duhamel}
Let \(U\) be a mild solution of \eqref{cp02}, and set
\[
u(x,t)=U(x,0,t),\qquad t>0.
\]
Then the boundary Duhamel contribution in the mild representation of
\(U\) is
\begin{equation}\label{E:boundary-Duhamel}
\frac{\mu}{\G(s)}
\int_0^t
\frac{1}{\sigma^{1-s}}
S(\sigma)
\bigl(
|u(\cdot,t-\sigma)|^{p-1}
u(\cdot,t-\sigma)
\bigr)(x)\,d\sigma .
\end{equation}
The identity holds in the space appearing in \eqref{E:positive-trace} or
\eqref{E:negative-trace}, the boundary value being the one furnished by
Theorem~8.1 of \cite{GS2}.
\end{proposition}

\begin{proof}
The generalized Duhamel formula of \cite{GS2}, with vanishing
interior forcing, gives
\begin{equation}\label{E:bulk-Duhamel}
U
=
\mathbb T_a^\star(U_0)
+
c_s\mu\,
\Theta_a^\star
\bigl(
|U(\cdot,0,\cdot)|^{p-1}U(\cdot,0,\cdot)
\bigr),
\end{equation}
where
\[
c_s
=
-\frac{2^{1-2s}\G(1-s)}
{\G(s)e^{i\pi s/2}}.
\]
By the explicit kernel formula \eqref{Theta2}, valid for every $0<s<1$, the boundary operator is represented by
\[
\Theta_a^\star(\Phi)(x,z,t)
=
-\frac{2^{2s-1}e^{i\frac{\pi s}{2}}}{\G(1-s)}
\int_0^t
\frac{e^{i\frac{z^2}{4\sigma}}}{\sigma^{1-s}}
S(\sigma)
\bigl(\Phi(\cdot,t-\sigma)\bigr)(x)\,d\sigma.
\]

Set $b=q$ when $0\le a<1$, and $b=q_\infty$ when $-1<a<0$.
Then $2/b+d/r=(d+a+1)/2$ and $b>2$, so that Theorems~7.4 and~8.1
of \cite{GS2} apply with the boundary-admissible exponent $b$.
For $\Phi\in C_0^\infty(\Rd\times(0,T))$, extended by zero in time,
the displayed integral is absolutely convergent in $L_x^r$ for each fixed
$t>0$ and every $z\ge0$. Its value at $z=0$ is therefore
\[
-\frac{2^{2s-1}e^{i\pi s/2}}{\Gamma(1-s)}
\int_0^t\sigma^{s-1}S(\sigma)\Phi(\cdot,t-\sigma)\,d\sigma.
\]
The cited boundary estimates and density extend this identity to
$\Phi\in L_t^{b'}(I;L_x^{r'})$, with zero extension outside $I$,
and give continuity in $z$ with values in $L_t^b(I;L_x^r)$.
For rough data the integral denotes this continuous extension.
Since $r=p+1$ and $q\ge r$ in the stated ranges, with equality at $p=p_c$, the nonlinear datum satisfies
\[
\bigl\||u|^{p-1}u\bigr\|_{L_t^{b'}(I;L_x^{r'})}
\le T^{1/b'-p/q}\|u\|_{L_t^q(I;L_x^r)}^p,
\qquad 1/b'-p/q\ge0.
\]
Here $1/b'-p/q=1-(p+1)/q$ for $a\ge0$, while for $a<0$
it equals $1-(p+1)/q+(1/q-1/q_\infty)$.
When $a<0$, the embedding $L_t^{q_\infty}(I)\subset L_t^q(I)$
gives the trace in \eqref{E:negative-trace}. Taking
$\Phi=|u|^{p-1}u$ and multiplying by $c_s\mu$ now gives
\eqref{E:boundary-Duhamel}.

\end{proof}

\begin{corollary}[The boundary memory equation]\label{C:volterra}
Under the hypotheses of Proposition~\ref{P:boundary-Duhamel}, the boundary
trace $u$ of the mild solution of the memory problem satisfies, for $0<t\le T$,
the Volterra integral equation
\begin{equation}\label{E:volterra}
u(x,t)
=
\mathbb T_a^\star(\mathcal P_su_0)(x,0,t)
\;+\;
\frac{\mu}{\G(s)}
\int_0^t
\frac{1}{\sigma^{1-s}}\,
S(\sigma)\bigl(|u(\cdot,t-\sigma)|^{p-1}u(\cdot,t-\sigma)\bigr)(x)\,d\sigma ,
\end{equation}
the identity holding in $L^q_t(I;L^r_x(\Rd))$ in the sense described in
Proposition~\ref{P:boundary-Duhamel}.
\end{corollary}

\begin{proof}
Take the boundary trace of \eqref{E:bulk-Duhamel}, which exists by
Definition~\ref{D:memory-solution}, and replace the trace of the second
summand by \eqref{E:boundary-Duhamel}.

\end{proof}

\begin{remark}\label{R:volterra}
In \eqref{E:volterra}, the first term depends only on the prescribed
history, while the second contains the nonlinear evolution on $(0,t)$.
The identity follows here from
Definition~\ref{D:memory-solution}. In the formulation of
Definition~\ref{D:memory-boundary-full}, it becomes the
equivalent boundary equation used to construct solutions.
\end{remark}

\begin{remark}\label{R:potential}
For $0<s<\min\{1,d/2\}$, Theorem~\ref{T:id} identifies the kernel
of \eqref{E:boundary-Duhamel} with the fractional potential associated
with $\mathscr E_s$. The restriction is the one used to construct
$\mathscr E_s$ in \eqref{Es}; when $d=1$ it excludes $s\ge\frac12$.
The well-posedness theorem below instead uses the boundary Duhamel
formula, valid for every $0<s<1$.
\end{remark}

\begin{lemma}[Free continuation]\label{L:concrete-free-history}
Let $u_0\in\Sm$, $U_0=\mathcal P_su_0$, and let
$b(\cdot,t)=\mathbb T_a^\star(U_0)(\cdot,0,t)$ be the boundary trace
of Lemma~\ref{L:free-trace}. For $t>0$,
\begin{equation}\label{E:concrete-free-history}
b(\cdot,t)=\frac{t^s}{\Gamma(s)\Gamma(1-s)}
\int_0^\infty\frac{\sigma^{-s}}{t+\sigma}
S(t+\sigma)u_0(\cdot,-\sigma)\,d\sigma.
\end{equation}
The integral converges in $L_x^2$ and agrees with the mixed-norm
boundary trace as a distribution on $t>0$.
\end{lemma}

\begin{proof}
Choose smooth past cutoffs $u_0^{\delta,R}$ supported where
$\delta<\sigma<R$, and let $U_0^{\delta,R}=\mathcal P_su_0^{\delta,R}$.
For the datum $e^{-\varepsilon z^2}U_0^{\delta,R}$, formulas
\eqref{SaS0}, \eqref{Pa} and the group law give the free boundary trace
\[
\frac{t^s}{\Gamma(s)\Gamma(1-s)}
\int_0^\infty\frac{\sigma^{-s}}{t+\sigma+4i\varepsilon t\sigma}
S(t+\sigma)u_0^{\delta,R}(\cdot,-\sigma)\,d\sigma.
\]
Indeed, the Gaussian makes the radial integral absolutely convergent,
and
\[
\int_0^\infty z e^{-\varepsilon z^2+iAz^2/4}\,dz
=\frac{2i}{A+4i\varepsilon},\qquad A=\frac1t+\frac1\sigma,
\]
with remaining constant $-i/[2\Gamma(s)\Gamma(1-s)]$.
As $\varepsilon\downarrow0$, the bulk data converge in $L_a^2$ and
their traces converge by Lemma~\ref{L:free-trace}. The displayed
integrals converge in $C([\tau,T];L_x^2)$ for $0<\tau<T$, since
$|t+\sigma+4i\varepsilon t\sigma|\ge t+\sigma$.

To remove the past cutoffs, set $G(\sigma)=S(\sigma)u_0(\cdot,-\sigma)$.
The remainder near zero, extended by zero to negative $\sigma$, has
$L^1_\sigma L_x^2$ norm $O(\delta)$ and uniformly bounded total
variation. Its Fourier transform is bounded in $L_x^2$ by
$C\min\{\delta,|\omega|^{-1}\}$. Thus \eqref{E:freq} bounds its
squared history norm by
\[
C\int_0^\infty\omega^s\min\{\delta^2,\omega^{-2}\}\,d\omega
\le C\delta^{1-s}\longrightarrow0.
\]
The tail tends to zero by the same estimate and the rapid decay of
$G$ and $G'$. The isometry of Corollary~\ref{C:extends} and
Lemma~\ref{L:free-trace} therefore give convergence of the liftings
and their traces. Dominated convergence gives convergence of the
integrals in $C([\tau,T];L_x^2)$, proving the identity \eqref{E:concrete-free-history}.

\end{proof}

\begin{remark}[The admissible coupling constants]\label{R:admissible-mu}
The mass-conservation condition $\operatorname{Im}(c_s\mu)=0$
can be expressed directly in terms of $\mu$. Since
$\G(s),\G(1-s)>0$ for $0<s<1$, we have
\[
c_s=-A_s\,e^{-i\frac{\pi s}{2}},
\qquad
A_s:=\frac{2^{1-2s}\G(1-s)}{\G(s)}>0,
\]
so that $\arg c_s=\pi-\frac{\pi s}{2}$, and consequently
\begin{equation}\label{E:admissible-mu}
\operatorname{Im}(c_s\mu)=0
\ \Longleftrightarrow\
\mu\in e^{i\frac{\pi s}{2}}\,\R.
\end{equation}
In particular, since $e^{i\frac{\pi s}{2}}\R\cap\R=\{0\}$ for
$0<s<1$, no nonzero real $\mu$ satisfies this sufficient mass-conservation
condition.  
At $s=1$, our convention gives
\[
(\partial_t-i\Delta_x)u=\mu|u|^{p-1}u
\quad\Longleftrightarrow\quad
i\partial_tu+\Delta_xu=i\mu|u|^{p-1}u.
\]
The usual real Hamiltonian NLS coefficients therefore correspond to
$\mu\in i\R$ in the present convention, which is the limit of the
admissible lines $e^{i\pi s/2}\R$ as $s\to1^-$.
\end{remark}

\subsection{Boundary estimates and well-posedness}

We use the boundary estimates of \cite{GS2} throughout the
range $-1<a<1$. Let
\[
a=1-2s,\qquad
\kappa_a=\frac{2^{-a}}{\Gamma((a+1)/2)}
=\frac{2^{2s-1}}{\Gamma(1-s)},
\]
and denote by $\mathbb S_a(t)$ the unitary Schr\"odinger
group on $L_a^2(\RN)$ associated with the homogeneous
Neumann realization of the Weinstein operator
$\Delta_x+\Ba$.
Thus $\mathbb S_a(t)U_0$ is the free half-space evolution
with initial datum $U_0$; in the notation of
Section~\ref{S:wellpos}, it is
$\mathbb T_a^\star(U_0)(\cdot,\cdot,t)$.

For a smooth boundary source $\Phi$, the boundary Duhamel
operator of \cite{GS2} is
\begin{equation}\label{Thetastar}
\Theta_a^\star\Phi(\cdot,z,t)
=
-i\int_0^t
S_a(z,0,t-\tau)
S(t-\tau)\Phi(\cdot,\tau)\,d\tau,
\end{equation}
where, for $\sigma>0$,
\[
S_a(z,0,\sigma)
=
\kappa_a e^{-i\pi(a+1)/4}
\sigma^{-(a+1)/2}e^{iz^2/(4\sigma)}.
\]
Theorem~7.4 of \cite{GS2} applies whenever
\begin{equation}\label{E:boundary-admissibility-full}
\frac2q+\frac dr=\frac{d+a+1}{2},
\qquad 2\le r<\infty,\quad q>2.
\end{equation}
In particular,
\begin{equation}\label{E:boundary-energy-full}
\|\Theta_a^\star\Phi\|_{L^\infty((0,T);L_a^2)}
\le
C_{d,q,r}\sqrt{\kappa_a}\,
\|\Phi\|_{L^{q'}((0,T);L^{r'})},
\end{equation}
with a constant independent of $T$. The factor
$\sqrt{\kappa_a}$ follows from the explicit coefficient
in the energy calculation in the proof of that theorem.

The same estimate also gives the free boundary estimate
needed for the nonlinear problem.

\begin{lemma}\label{L:free-boundary-full}
Let $-1<a<1$, and suppose that $q,r$ satisfy
\eqref{E:boundary-admissibility-full}. For smooth initial
data, set
\[
\mathcal R_aU_0
=
\bigl(\mathbb S_a(t)U_0\bigr)\big|_{z=0}.
\]
This operator extends uniquely to a bounded linear map
\[
\mathcal R_a:
L_a^2(\RN)\longrightarrow
L^q((0,\infty);L^r(\Rd)),
\]
and
\begin{equation}\label{E:free-boundary-full}
\|\mathcal R_aU_0\|_{L^q((0,\infty);L^r)}
\le
C_{d,q,r}\sqrt{\kappa_a}\,
\|U_0\|_{L_a^2}.
\end{equation}
Moreover, the free evolution has this boundary value
in the sense that
\[
\mathbb S_a(t)U_0(\cdot,z)
\longrightarrow \mathcal R_aU_0
\quad\text{in }\mathscr D'(\Rd\times(0,\infty))
\quad\text{as }z\downarrow0.
\]
Here the slices of the free evolution are understood
through its Fourier--Hankel representation.
\end{lemma}

\begin{proof}
We first take smooth initial data and
$\Phi\in C_c^\infty(\Rd\times(0,T))$. The Duhamel formula,
the group property, and the usual complex $L^2$ pairings
give
\[
\begin{aligned}
&\int_0^T
\left\langle
\bigl(\mathbb S_a(t)U_0\bigr)(\cdot,0),
\Phi(\cdot,t)
\right\rangle_{L^2_x}\,dt\\
&\qquad=
\left\langle
U_0,\,
i\mathbb S_a(-T)\Theta_a^\star\Phi(T)
\right\rangle_{L_a^2}.
\end{aligned}
\]
Consequently, unitarity and
\eqref{E:boundary-energy-full} imply
\[
\left|
\int_0^T
\left\langle
\bigl(\mathbb S_a(t)U_0\bigr)(\cdot,0),
\Phi(\cdot,t)
\right\rangle_{L^2_x}\,dt
\right|
\le
C_{d,q,r}\sqrt{\kappa_a}\,
\|U_0\|_{L_a^2}
\|\Phi\|_{L^{q'}((0,T);L^{r'})}.
\]
Since $\Phi$ is supported away from $T$, the Duhamel
term is continuous in $L_a^2$ near $T$, so the
evaluation at $T$ is legitimate.

Duality proves the free boundary estimate on $(0,T)$.
Density then defines $\mathcal R_a$ on all of $L_a^2$.
The extensions obtained on different intervals agree.
Since the constant is independent of $T$, letting
$T\to\infty$ proves \eqref{E:free-boundary-full}.
We also verify that this extension is the distributional
boundary value of the free evolution. Following the
notation of \cite{GS2}, put
\[
\nu=\frac{a-1}{2},
\qquad
G_\nu(w)=w^{-\nu}J_\nu(w).
\]
In the Fourier-Hankel representation, the dependence
on $z$ is through $G_\nu(\lambda z)$, while the time
multiplier is
\[
e^{-it(4\pi^2|\xi|^2+\lambda^2)}.
\]
We have
\[
G_\nu(\lambda z)\longrightarrow
G_\nu(0)=\frac{2^{-\nu}}{\Gamma(\nu+1)}
\qquad\text{as }z\downarrow0,
\]
and, uniformly for $0\le z\le1$,
\[
|G_\nu(\lambda z)|
\le C_a(1+\lambda)^{\max\{-a/2,0\}}.
\]
After pairing with
$\varphi\in C_c^\infty(\Rd\times(0,\infty))$, the
remaining factor is its space-time Fourier transform
evaluated at
\[
\left(\xi,
-\frac{4\pi^2|\xi|^2+\lambda^2}{2\pi}\right).
\]
This factor decreases faster than every power of
$|\xi|+\lambda$. Since $a>-1$, dominated convergence
therefore applies in
$L^2(d\xi\,\lambda^a\,d\lambda)$.
Cauchy--Schwarz against the Fourier--Hankel transform
of $U_0$ proves convergence of the slice pairings.

For completeness, these pairings represent the actual
bulk evolution: compact spectral truncations converge
to $\mathbb S_a(t)U_0$ in $C_tL_a^2$, and their slice
pairings converge uniformly for $z$ in compact sets.
Integration against a test function in $z>0$
identifies the resulting distribution with the bulk
evolution. Finally, density identifies its boundary
limit with $\mathcal R_aU_0$.

\end{proof}

We next record the corresponding estimates for the
inhomogeneous term. Recall that
\[
c_s
=
-\frac{2^{1-2s}\Gamma(1-s)e^{-i\pi s/2}}{\Gamma(s)}
=
-\frac{e^{-i\pi s/2}}{\kappa_a\Gamma(s)}.
\]
For a smooth source $\Phi$, define
\begin{equation}\label{E:memory-duhamel-full}
\mathcal V_s\Phi(t)
=
\frac1{\Gamma(s)}
\int_0^t
(t-\tau)^{s-1}S(t-\tau)\Phi(\tau)\,d\tau.
\end{equation}
Substituting the formula \eqref{SaS0} for
$S_a(z,0,\sigma)$ into the integral \eqref{Thetastar} defining
$\Theta_a^\star$, and multiplying by $c_s$, we obtain
\[
c_s\Theta_a^\star\Phi(\cdot,z,t)
=
\frac1{\Gamma(s)}
\int_0^t
(t-\tau)^{s-1}e^{iz^2/[4(t-\tau)]}
S(t-\tau)\Phi(\cdot,\tau)\,d\tau.
\]
For $\Phi\in C_c^\infty(\Rd\times(0,T))$, dominated
convergence and the unitarity of $S(t)$ therefore give
\[
c_s\Theta_a^\star\Phi(\cdot,z,t)
\longrightarrow
\mathcal V_s\Phi(t)
\qquad\text{in }L^2(\Rd)
\quad\text{as }z\downarrow0,
\]
for every $0<t<T$. Thus $\mathcal V_s\Phi$ is the
boundary value of $c_s\Theta_a^\star\Phi$.
The dispersive estimate for $S(t)$ and
\eqref{E:boundary-admissibility-full} give
\[
\|\mathcal V_s\Phi(t)\|_{L^r}
\le
\frac{C_{d,r}}{\Gamma(s)}
\int_0^t
(t-\tau)^{-2/q}\|\Phi(\tau)\|_{L^{r'}}\,d\tau,
\]
because
\[
1-s+d\left(\frac12-\frac1r\right)=\frac2q.
\]
The Hardy--Littlewood--Sobolev inequality in time,
together with \eqref{E:boundary-energy-full}, yields
\begin{equation}\label{E:memory-forcing-full}
\begin{aligned}
&\|\mathcal V_s\Phi\|_{L^q((0,T);L^r)}
+
\sqrt{\kappa_a}\,
\|c_s\Theta_a^\star\Phi\|_{C([0,T];L_a^2)}
\\
&\qquad\le
\frac{C_{d,q,r}}{\Gamma(s)}
\|\Phi\|_{L^{q'}((0,T);L^{r'})}.
\end{aligned}
\end{equation}
The constant is independent of $T$. 
The estimate \eqref{E:memory-forcing-full} extends
both operators by density to every source
$\Phi\in L^{q'}((0,T);L^{r'}(\Rd))$.
For each such source, the function
\[
t\longmapsto c_s\Theta_a^\star\Phi(\cdot,\cdot,t)
\]
belongs to $C([0,T];L_a^2(\RN))$ and vanishes at
$t=0$. Indeed, these properties hold for smooth
sources, and \eqref{E:boundary-energy-full} ensures
that smooth approximations of $\Phi$ produce
solutions converging uniformly in $L_a^2(\RN)$
on $[0,T]$. Continuity and the zero initial value
therefore pass to the limit.

Moreover, Theorem~8.1 of \cite{GS2} identifies
$\mathcal V_s\Phi$ as the boundary value of this
half-space solution:
\[
c_s\Theta_a^\star\Phi(\cdot,z,\cdot)
\longrightarrow
\mathcal V_s\Phi
\quad\text{in }L^q((0,T);L^r(\Rd))
\quad\text{as }z\downarrow0.
\]
We now set
\[
p_*(s)=1+\frac{4s}{d+2-2s},
\qquad
1<p\le p_*(s),
\qquad r=p+1,
\]
and choose $q$ by
\[
\frac2q+\frac dr=\frac{d+2-2s}{2}.
\]
Then $q\ge p+1>2$, and
\begin{equation}\label{E:memory-time-gain-full}
\theta
:=
1-\frac{p+1}{q}
=
s-\frac{d+2-2s}{4}(p-1)
\ge0.
\end{equation}
Equality holds precisely when $p=p_*(s)$.

We first specify what we mean by a solution of the
Cauchy problem with memory \eqref{cpmem}.
In the definition below, we allow $p>1$ whenever
the exponent $q$ determined by
\[
\frac2q+\frac{d}{p+1}=\frac{d+2-2s}{2}
\]
satisfies $q>2$.
The additional restriction $p\le p_*(s)$ enters
Theorem~\ref{T:memory-boundary-full}, where we
prove existence, uniqueness and continuous dependence
for arbitrary prescribed histories in $\mathcal H^s$.

\begin{definition}\label{D:memory-boundary-full}
Let $0<s<1$, $\mu\in\mathbb C$, and $p>1$.
Set $a=1-2s$ and $r=p+1$, and suppose that the exponent
$q$ determined by
\[
\frac2q+\frac dr=\frac{d+a+1}{2}
\]
satisfies $q>2$.
Let $u_0\in\mathcal H^s$ and
$U_0=\mathcal P_su_0$. A boundary mild solution on $(0,T)$
consists of
\[
U\in C([0,T];L_a^2(\RN)),
\qquad
u\in L^q((0,T);L^r(\Rd)),
\]
with
\[
|u|^{p-1}u\in L^{q'}((0,T);L^{r'}(\Rd)),
\]
such that
\[
U(t)
=
\mathbb S_a(t)U_0
+
c_s\mu\Theta_a^\star\bigl(|u|^{p-1}u\bigr)(t)
\]
and
\[
U(\cdot,z,\cdot)\longrightarrow u
\quad\text{in }\mathscr D'(\Rd\times(0,T))
\quad\text{as }z\downarrow0.
\]
The prescribed past is $u_0$. For a concrete smooth
history, the solution on negative times is therefore
$u_0$ itself.
\end{definition}

We note that, when $1<p\le p_*(s)$, the source condition in
Definition~\ref{D:memory-boundary-full} automatically follows from
H\"older's inequality:
\[
\bigl\||u|^{p-1}u\bigr\|_{L^{q'}((0,T);L^{r'})}
\le T^{1-(p+1)/q}\|u\|_{L^q((0,T);L^r)}^p.
\]
By Lemma~\ref{L:free-boundary-full} and the trace
identity for the Duhamel term, this definition is
equivalent to the boundary equation
\begin{equation}\label{E:memory-volterra-full}
u
=
\mathcal R_aU_0
+
\mu\mathcal V_s\bigl(|u|^{p-1}u\bigr).
\end{equation}
Once $u$ is known, the extension $U$ is determined by
the Duhamel formula in
Definition~\ref{D:memory-boundary-full}.

\begin{theorem}\label{T:memory-boundary-full}
Let $0<s<1$, $\mu\in\mathbb C$, and $1<p\le p_*(s)$.
For every prescribed history $u_0\in\mathcal H^s$, there
exists $T>0$ and a unique mild solution on $(0,T)$ in the
sense of Definition~\ref{D:memory-boundary-full}.
On a common interval of existence, the solution depends
locally Lipschitz continuously on the history in
$\mathcal H^s$, with $u$ measured in
$L^q((0,T);L^{p+1}(\Rd))$ and $U$ in
$C([0,T];L_a^2(\RN))$.
\begin{itemize}
\item[(i)]
If $1<p<p_*(s)$, the interval can be chosen in terms
of $d,s,p,|\mu|$ and $\|u_0\|_{\mathcal H^s}$.
\item[(ii)]
If $p=p_*(s)$, local existence holds for every
$u_0\in\mathcal H^s$. Moreover, there exists
$\varepsilon>0$, depending on $d,s,\mu$, such that
\[
\sqrt{\kappa_a}\,\|u_0\|_{\mathcal H^s}\le\varepsilon
\]
implies global existence, with
\[
u\in L^q((0,\infty);L^{p+1}(\Rd)),
\qquad
U\in C([0,\infty);L_a^2(\RN)),
\qquad
\sup_{t\ge0}\|U(t)\|_{L_a^2}<\infty.
\]
\end{itemize}
Uniqueness holds among all mild solutions in the sense of
Definition~\ref{D:memory-boundary-full}, including at
$p=p_*(s)$.
\end{theorem}

\begin{proof}
We begin by solving the boundary equation
\eqref{E:memory-volterra-full}; once its solution $u$
is known, the Duhamel formula in
Definition~\ref{D:memory-boundary-full} will determine
the half-space solution $U$.
Let $U_0=\mathcal P_su_0$, and take $r=p+1$ and $q$
as in \eqref{E:boundary-admissibility-full}, so that
$p+1\le q<\infty$. We seek $u$ in
\[
E_T=L^q((0,T);L^r(\Rd)).
\]
We denote the free boundary evolution by
$b=\mathcal R_aU_0$ and write $N(u)=|u|^{p-1}u$
for the nonlinearity.
The isometry of Corollary~\ref{C:extends} and the
free boundary estimate \eqref{E:free-boundary-full} give
\begin{equation}\label{E:memory-free-data-full}
\|b\|_{E_T}
\le
C_{d,q,r}\sqrt{\kappa_a}\,
\|u_0\|_{\mathcal H^s}.
\end{equation}
This bound holds for every $0<T\le\infty$, with the
same constant.
Since $r'=r/p$, the standard pointwise inequality
\[
|N(u)-N(v)|
\le C_p\bigl(|u|^{p-1}+|v|^{p-1}\bigr)|u-v|
\]
and H\"older's inequality give
\begin{equation}\label{E:memory-nonlinear-full}
\|N(u)-N(v)\|_{L^{q'}((0,T);L^{r'})}
\le
C_pT^\theta
\bigl(\|u\|_{E_T}^{p-1}+\|v\|_{E_T}^{p-1}\bigr)
\|u-v\|_{E_T}.
\end{equation}
Here $\theta=1/q'-p/q=1-(p+1)/q\ge0$ is the
exponent in \eqref{E:memory-time-gain-full}.
When $p=p_*(s)$, we have $\theta=0$ and
$q=p+1$; in this case \eqref{E:memory-nonlinear-full}
also holds on $(0,\infty)$, without a time factor.

Consider now the map
\[
\mathcal T(u)=b+\mu\mathcal V_sN(u).
\]
For $\|u\|_{E_T},\|v\|_{E_T}\le\rho$,
\eqref{E:memory-forcing-full} and
\eqref{E:memory-nonlinear-full} imply
\[
\begin{aligned}
\|\mathcal T(u)\|_{E_T}
&\le \|b\|_{E_T}+C|\mu|T^\theta\rho^p,\\
\|\mathcal T(u)-\mathcal T(v)\|_{E_T}
&\le 2C|\mu|T^\theta\rho^{p-1}\|u-v\|_{E_T},
\end{aligned}
\]
where $C$ depends only on $d,s,p$. Consequently, provided that
\begin{equation}\label{E:memory-contraction-full}
\|b\|_{E_T}\le\frac{\rho}{2},
\qquad
C|\mu|T^\theta\rho^{p-1}\le\frac14,
\end{equation}
then $\mathcal T$ maps the closed ball of radius $\rho$
in $E_T$ into itself and is a contraction there
with Lipschitz constant at most $1/2$. We next analyze the cases (i) and (ii).

\smallskip
\noindent (i). Suppose first that $p<p_*(s)$. Since $\theta>0$,
the factor $T^\theta$ allows us to handle an
arbitrarily large history by shortening the time
interval. We first choose $\rho>0$, using only
the history norm and $d,s,p$, so that
\[
\rho\ge
2C_{d,q,r}\sqrt{\kappa_a}\,
\|u_0\|_{\mathcal H^s}.
\]
By \eqref{E:memory-free-data-full},
$\|b\|_{E_T}\le\rho/2$ for every $T>0$.
With this radius fixed, choose $T>0$ small enough
that $C|\mu|T^\theta\rho^{p-1}\le1/4$.
The contraction mapping theorem now gives a solution
of \eqref{E:memory-volterra-full}. Both choices use
only $d,s,p,|\mu|$ and $\|u_0\|_{\mathcal H^s}$,
as required for the subcritical existence time.

\smallskip

\noindent (ii) At the critical exponent $p=p_*(s)$, we have
$\theta=0$, and the contraction condition requires
a small radius. We therefore choose $\rho>0$ so that
$C|\mu|\rho^{p-1}\le1/4$, and then shorten the time
interval until the free term fits inside this ball.
Indeed, the time exponent is now $q=p+1<\infty$,
and \eqref{E:free-boundary-full} gives
\[
b\in L^q((0,\infty);L^r(\Rd)).
\]
Absolute continuity of the integral therefore yields
\[
\|b\|_{E_T}\longrightarrow0
\qquad\text{as }T\downarrow0.
\]
We can thus choose $T$ so that $\|b\|_{E_T}\le\rho/2$,
and \eqref{E:memory-contraction-full} gives local
existence for every history. If
\[
\sqrt{\kappa_a}\,\|u_0\|_{\mathcal H^s}
\le\frac{\rho}{2C_{d,q,r}},
\]
then \eqref{E:memory-free-data-full} gives
$\|b\|_{E_\infty}\le\rho/2$.
The same contraction on
$E_\infty=L^q((0,\infty);L^r(\Rd))$ produces a global
solution. This proves the smallness assertion with
$\varepsilon=\rho/(2C_{d,q,r})$.

Having found $u$, we return to the half-space
equation. Taking $v=0$ in
\eqref{E:memory-nonlinear-full}
shows that $N(u)\in L^{q'}((0,T);L^{r'}(\Rd))$.
We may therefore define
\[
U(t)=\mathbb S_a(t)U_0
+c_s\mu\Theta_a^\star N(u)(t).
\]
The continuity and zero initial value of the Duhamel
term established after \eqref{E:memory-forcing-full},
together with the strong continuity of $\mathbb S_a(t)$,
give $U\in C([0,T];L_a^2(\RN))$ and $U(0)=U_0$.
Lemma~\ref{L:free-boundary-full} identifies the
distributional boundary value of the free term as $b$.
For the Duhamel term, \cite[Theorem~8.1]{GS2},
applied with $\Phi=N(u)$, gives
\[
c_s\Theta_a^\star N(u)(\cdot,z,\cdot)
\longrightarrow \mathcal V_sN(u)
\quad\text{in }L^q((0,T);L^r(\Rd))
\quad\text{as }z\downarrow0.
\]
Hence
\[
U(\cdot,z,\cdot)\longrightarrow
b+\mu\mathcal V_sN(u)=u
\quad\text{in }\mathscr D'(\Rd\times(0,T))
\quad\text{as }z\downarrow0.
\]
Thus $(U,u)$ satisfies
Definition~\ref{D:memory-boundary-full}.
For the global critical solution,
\eqref{E:memory-forcing-full} and unitarity give
\[
\sup_{t\ge0}\|U(t)\|_{L_a^2}
\le
\|u_0\|_{\mathcal H^s}
+\frac{C_{d,q,r}|\mu|}{\sqrt{\kappa_a}\Gamma(s)}
\|u\|_{E_\infty}^p
<\infty.
\]
Here we apply the estimate on finite intervals and
use its independence of their length.

The contraction argument gives uniqueness within
the ball used to construct $u$. To prove uniqueness
in the class of Definition~\ref{D:memory-boundary-full},
we now take any two solutions $u,v\in E_T$ of
\eqref{E:memory-volterra-full} with the same history.
Since $q<\infty$, we can choose a
finite partition $0=t_0<\cdots<t_m=T$ whose intervals
$I=(t_j,t_{j+1})$ have length at most one and satisfy
\[
C|\mu|
\bigl(\|u\|_{L^q(I;L^r)}^{p-1}
+\|v\|_{L^q(I;L^r)}^{p-1}\bigr)
\le\frac12.
\]
Suppose that $u=v$ on $(0,t_j)$; for $j=0$ there
is no preceding interval. Subtracting the two
Volterra equations gives, for $t\in I$,
\[
(u-v)(t)
=
\frac{\mu}{\Gamma(s)}
\int_{t_j}^t
(t-\tau)^{s-1}S(t-\tau)
\bigl(N(u(\tau))-N(v(\tau))\bigr)\,d\tau,
\]
because the integrands agree for $\tau<t_j$.
After a translation in time,
\eqref{E:memory-forcing-full} and
\eqref{E:memory-nonlinear-full} apply on $I$.
Since $|I|^\theta\le1$, they yield
\[
\|u-v\|_{L^q(I;L^r)}
\le\frac12\|u-v\|_{L^q(I;L^r)}.
\]
Thus $u=v$ on $I$, and induction proves equality
on $(0,T)$. The half-space solutions then agree
by their Duhamel formulas in
Definition~\ref{D:memory-boundary-full}.
For global solutions, this argument applies on
every finite interval.

We finish by proving that the solution depends
locally Lipschitz continuously on the prescribed
history. Fix a history and carry out the construction
with $\|b\|_{E_T}\le\rho/4$ and
$C|\mu|T^\theta\rho^{p-1}\le1/4$.
For $p<p_*(s)$, this is achieved by choosing a
larger radius before choosing $T$; at $p=p_*(s)$,
it is achieved by decreasing $T$.
The bound $\|b\|_{E_T}\le\rho/4$ leaves room for
a perturbation of the history: by linearity and
\eqref{E:memory-free-data-full}, the free term remains
bounded by $\rho/2$ for all histories in a sufficiently
small $\mathcal H^s$ neighborhood. Thus we can use
the same time interval and contraction ball for
every history in this neighborhood.
Let $u_0,v_0$ be two histories in this neighborhood,
with solutions $(U,u)$ and $(V,v)$.
Subtracting \eqref{E:memory-volterra-full} and using
the contraction constant $1/2$, we obtain
\[
\|u-v\|_{E_T}
\le
C_{d,q,r}\sqrt{\kappa_a}\,
\|u_0-v_0\|_{\mathcal H^s}
+\frac12\|u-v\|_{E_T}.
\]
Consequently,
\[
\|u-v\|_{E_T}
\le
2C_{d,q,r}\sqrt{\kappa_a}\,
\|u_0-v_0\|_{\mathcal H^s}.
\]
Applying \eqref{E:memory-forcing-full} and
\eqref{E:memory-nonlinear-full} to the difference
of the Duhamel terms, and using unitarity and
Corollary~\ref{C:extends} for the initial data, gives
\[
\sqrt{\kappa_a}\,
\|U-V\|_{C([0,T];L_a^2)}
\le
C\sqrt{\kappa_a}\,\|u_0-v_0\|_{\mathcal H^s}.
\]
This proves local Lipschitz dependence on the
history for both the boundary solution in $E_T$
and the half-space solution in
$C([0,T];L_a^2(\RN))$, on a common interval of
existence.

\end{proof}

The solutions constructed through \cite{GS2} in
Definition~\ref{D:memory-solution} agree with those of
Theorem~\ref{T:memory-boundary-full}. For $0<s\le\frac12$,
the time exponents in the two formulations coincide.
For $\frac12<s<1$, let $q_0$ denote the time exponent in
Definition~\ref{D:memory-solution}, and let $q$ satisfy
\eqref{E:boundary-admissibility-full}. In the range
$1<p\le p_c$ of that definition, $q_0<q$ and $q_0\ge p+1$.
Hence, on every finite interval,
\[
\bigl\||u|^{p-1}u\bigr\|_{L^{q'}((0,T);L^{r'})}
\le T^{1/q'-p/q_0}\|u\|_{L^{q_0}((0,T);L^r)}^p,
\]
where
\[
\frac1{q'}-\frac p{q_0}
=1-\frac{p+1}{q_0}+\frac1{q_0}-\frac1q\ge0.
\]
Applying \eqref{E:free-boundary-full} and
\eqref{E:memory-forcing-full} to \eqref{E:volterra}
shows that the strong trace solution belongs to $L^q_tL^r_x$
and satisfies Definition~\ref{D:memory-boundary-full}.
Its strong trace agrees with the distributional trace.
Uniqueness in Theorem~\ref{T:memory-boundary-full}
therefore identifies the two solutions.

\begin{corollary}[Global subcritical solutions]\label{C:global-memory}
Let $0<s\le\frac12$, $1<p<p_*(s)$, and
$\operatorname{Im}(c_s\mu)=0$. Every history
$u_0\in\mathcal H^s$ admits a global boundary mild solution,
in the sense of Definition~\ref{D:memory-boundary-full}
on every finite time interval.
\end{corollary}

\begin{proof}
For $\mu\ne0$, apply \cite[Theorem~2.9(2)]{GS2} with
initial datum $U_0=\mathcal P_su_0$, zero interior forcing,
and boundary coefficient $-c_s\mu$.
Since $\|U_0\|_{L_a^2}=\|u_0\|_{\mathcal H^s}$ and
$a=1-2s\ge0$, it gives a global half-space solution whose
strong boundary trace satisfies
Definition~\ref{D:memory-solution} on each finite interval.
The preceding identification shows that this is the solution
of Theorem~\ref{T:memory-boundary-full}.
For $\mu=0$, the assertion follows directly from the free
boundary estimate \eqref{E:free-boundary-full}.

\end{proof}

We emphasize that the exponent in Theorem~\ref{T:memory-boundary-full}
satisfies
\[
\lim_{s\uparrow1}p_*(s)=1+\frac4d,
\]
the mass-critical exponent for the classical
Schr\"odinger equation. The normalization of the
history norm also has the expected limit for a fixed
Schwartz history.
Indeed, set
\[
\mathcal G_{u_0}(\omega)
=
\int_0^\infty
e^{-i\omega\sigma}
S(\sigma)u_0(\cdot,-\sigma)\,d\sigma,
\]
where the integral is taken in $L^2(\Rd)$.
By Definition~\ref{D:Qs} and the history-energy identity
\eqref{E:freq}, we have
\[
Q_s(u_0)
=
\frac{2^{1-2s}}{\Gamma(s)^2}
\int_0^\infty
\|\mathcal G_{u_0}(\omega)\|_{L^2_x}^2\ \omega^s
\,d\omega.
\]
Since $u_0$ is the restriction of a Schwartz function,
the first two time derivatives of
$S(\sigma)u_0(\cdot,-\sigma)$ are integrable in
$L^2_x$. Two integrations by parts therefore give
\[
\mathcal G_{u_0}(\omega)
=
\frac{u_0(\cdot,0)}{i\omega}
+
O_{L^2_x}(\omega^{-2})
\qquad\text{as }\omega\to+\infty.
\]
It follows that
\[
\int_0^\infty
\omega^s\|\mathcal G_{u_0}(\omega)\|_{L^2_x}^2
\,d\omega
=
\frac{\|u_0(\cdot,0)\|_{L^2_x}^2}{1-s}+O(1)
\qquad\text{as }s\uparrow1.
\]
Consequently,
\[
\lim_{s\uparrow1}2(1-s)Q_s(u_0)
=
\|u_0(\cdot,0)\|_{L^2_x}^2,
\]
and, since $\kappa_a\sim2(1-s)$,
\[
\lim_{s\uparrow1}
\sqrt{\kappa_a}\,\|u_0\|_{\mathcal H^s}
=
\|u_0(\cdot,0)\|_{L^2_x}.
\]
Thus both the critical exponent and the normalized
size of a fixed smooth history approach their
classical counterparts.

\medskip

\subsection{The classical limit}
We next establish convergence to the classical
Schr\"odinger equation as $s\uparrow1$. The prescribed
history and the exponent of the nonlinearity are kept
fixed.

\begin{theorem}[The classical limit]
\label{T:nonlinear-classical-limit}
Let $u_0\in\Sm$, $\mu\in\mathbb C$, and
\[
1<p\le 1+\frac4d.
\]
There exist $s_0<1$ and $T>0$ such that, for every
$s\in[s_0,1)$, the memory problem has a boundary mild
solution
\[
u_s\in
C([0,T];L^2(\Rd)\cap L^{p+1}(\Rd)).
\]
These solutions are unique in
$C([0,T];L^{p+1}(\Rd))$ and satisfy
\[
u_s\longrightarrow u_1
\quad\text{in }
C([0,T];L^2(\Rd)\cap L^{p+1}(\Rd))
\qquad\text{as }s\uparrow1,
\]
where $u_1$ is the classical mild solution of
\[
\begin{cases}
(\partial_t-i\Delta_x)u_1
=\mu|u_1|^{p-1}u_1,&t>0,\\
u_1(\cdot,0)=u_0(\cdot,0).
\end{cases}
\]
Let $[0,T_{\max})$ be the maximal interval of existence of
$u_1$ in
\[
C([0,T_{\max});L^2(\Rd)\cap L^{p+1}(\Rd)).
\]
For every finite $0<T<T_{\max}$, the solutions $u_s$
exist on $[0,T]$ for all $s$ sufficiently close to $1$,
and the stated convergence holds on that interval.
\end{theorem}

\begin{proof}
Set
\[
r=p+1,\qquad
\beta=d\left(\frac12-\frac1r\right).
\]
Our assumption on $p$ gives
\[
0<\beta\le\frac d{d+2}<1.
\]
Fix $s_0$ with $\beta<s_0<1$. We shall obtain estimates
uniformly for $s\in[s_0,1]$.

For $s<1$, let $b_s$ denote the free continuation of
the prescribed history. Thus, for $t>0$,
\[
b_s(t)
=
\frac{\sin(\pi s)}{\pi}\,t^s
\int_0^\infty
\frac{\sigma^{-s}}{t+\sigma}
S(t+\sigma)u_0(\cdot,-\sigma)\,d\sigma.
\]
We set
\[
b_s(0)=u_0(\cdot,0),
\qquad
b_1(t)=S(t)u_0(\cdot,0).
\]
We first prove that, for every finite $T$,
\begin{equation}\label{E:free-classical-limit}
b_s\longrightarrow b_1
\quad\text{in }C([0,T];H^1(\Rd)).
\end{equation}

To this end, write
\[
g(\sigma)=S(\sigma)u_0(\cdot,-\sigma),
\qquad
\rho_s(y)=\frac{\sin(\pi s)}{\pi}
\frac{y^{-s}}{1+y}.
\]
The function $g$ is continuous and bounded with values
in $H^1(\Rd)$. The beta integral gives
\[
\int_0^\infty\rho_s(y)\,dy=1,
\]
and the substitution $\sigma=ty$ yields
\[
b_s(t)=S(t)\int_0^\infty\rho_s(y)g(ty)\,dy.
\]
These are Bochner integrals in $H^1$. This formula
also proves continuity at $t=0$ and gives the uniform
bound
\[
\sup_{s_0\le s<1}\sup_{t\ge0}\|b_s(t)\|_{H^1}
\le
\sup_{\sigma\ge0}\|g(\sigma)\|_{H^1}.
\]

For $\delta>0$, put
\[
\omega_g(\delta)
=
\sup_{0\le\sigma\le\delta}
\|g(\sigma)-g(0)\|_{H^1},
\qquad
M_g=\sup_{\sigma\ge0}\|g(\sigma)\|_{H^1}.
\]
Splitting the integral where $ty=\delta$, we obtain
\[
\sup_{0\le t\le T}
\|b_s(t)-b_1(t)\|_{H^1}
\le
\omega_g(\delta)
+
2M_g\frac{\sin(\pi s)}{\pi s}
\left(\frac T\delta\right)^s.
\]
First let $s\uparrow1$ and then $\delta\downarrow0$.
This proves \eqref{E:free-classical-limit}.
Since $H^1(\Rd)$ embeds continuously into $L^r(\Rd)$
for the present range of $r$, we also have
\begin{equation}\label{E:free-classical-limit-Lr}
b_s\longrightarrow b_1
\quad\text{in }
C([0,T];L^2(\Rd)\cap L^r(\Rd)).
\end{equation}

For $s\in[s_0,1]$, define
\[
k_s(\sigma)=\frac{\sigma^{s-1}}{\Gamma(s)},
\qquad \sigma>0,
\]
and
\[
\mathcal D_s\Phi(t)
=
\int_0^t
k_s(t-\tau)S(t-\tau)\Phi(\tau)\,d\tau.
\]
In particular, $k_1=1$ and $\mathcal D_1$ is the
classical Schr\"odinger Duhamel operator.

The dispersive estimate gives
\[
\|S(\sigma)\Phi\|_{L^r}
\le C_{d,r}\sigma^{-\beta}\|\Phi\|_{L^{r'}}.
\]
Consequently, for
$\Phi\in C([0,T];L^{r'}(\Rd))$,
\begin{equation}\label{E:uniform-Volterra-classical}
\|\mathcal D_s\Phi\|_{C([0,T];L^r)}
\le
\frac{C_{d,r}}{\Gamma(s)}
\frac{T^{s-\beta}}{s-\beta}
\|\Phi\|_{C([0,T];L^{r'})}.
\end{equation}
Here the integral defining $\mathcal D_s\Phi$ converges
in $L^r$, since $s>\beta$. Its continuity in time
follows by writing it as
\[
\int_0^t k_s(\sigma)S(\sigma)\Phi(t-\sigma)\,d\sigma
\]
and using continuity of $\Phi$, together with the
integrable majorant $\sigma^{s-1-\beta}$.

For $T\le1$, \eqref{E:uniform-Volterra-classical}
implies
\[
\|\mathcal D_s\Phi\|_{C([0,T];L^r)}
\le
C_{d,r,s_0}T^{s_0-\beta}
\|\Phi\|_{C([0,T];L^{r'})},
\]
uniformly for $s\in[s_0,1]$. Moreover,
\begin{equation}\label{E:Volterra-operator-limit-Lr}
\begin{aligned}
&\|(\mathcal D_s-\mathcal D_1)\Phi\|_{C([0,T];L^r)}
\\
&\qquad\le
C_{d,r}
\left(
\int_0^T|k_s(\sigma)-1|\sigma^{-\beta}\,d\sigma
\right)
\|\Phi\|_{C([0,T];L^{r'})}.
\end{aligned}
\end{equation}
The integral in parentheses tends to zero by dominated
convergence, because $s_0>\beta$.

Write $N(u)=|u|^{p-1}u$. Since $r=p+1$, we have
\[
\|N(u)-N(v)\|_{L^{r'}}
\le
C_p\bigl(\|u\|_{L^r}^{p-1}
+\|v\|_{L^r}^{p-1}\bigr)\|u-v\|_{L^r}.
\]
The boundary equations are
\begin{equation}\label{E:nonlinear-classical-family}
u_s=b_s+\mu\mathcal D_sN(u_s),
\qquad s_0\le s\le1.
\end{equation}
The uniform bounds for $b_s$ and
\eqref{E:uniform-Volterra-classical} allow us to choose
a common ball in $C([0,T];L^r)$ and a common $T>0$
on which all the maps
\[
u\longmapsto b_s+\mu\mathcal D_sN(u)
\]
are contractions with Lipschitz constant at most
$1/2$. This constructs $u_s$ for $s\in[s_0,1]$.
For $s=1$, equation
\eqref{E:nonlinear-classical-family} is precisely
the classical NLS mild equation.

Subtracting the equations for $u_s$ and $u_1$, we find
\[
\begin{aligned}
u_s-u_1
={}&b_s-b_1
+\mu\mathcal D_s\bigl(N(u_s)-N(u_1)\bigr)\\
&+\mu(\mathcal D_s-\mathcal D_1)N(u_1).
\end{aligned}
\]
The contraction estimate, together with
\eqref{E:free-classical-limit-Lr} and
\eqref{E:Volterra-operator-limit-Lr}, therefore gives
\begin{equation}\label{E:nonlinear-classical-limit-Lr}
u_s\longrightarrow u_1
\quad\text{in }C([0,T];L^r).
\end{equation}
It follows as well that
\[
N(u_s)\longrightarrow N(u_1)
\quad\text{in }C([0,T];L^{r'}).
\]

We now prove convergence in $C([0,T];L^2)$. Let
\[
q_{\mathrm{cl}}=\frac2\beta.
\]
Then $q_{\mathrm{cl}}>2$ and
\[
\frac2{q_{\mathrm{cl}}}+\frac dr=\frac d2.
\]
The classical dual Strichartz estimate gives
\[
\left\|
\int_0^T S(-\tau)\Phi(\tau)\,d\tau
\right\|_{L^2}
\le
C_{d,r}
\|\Phi\|_{L^{q_{\mathrm{cl}}'}((0,T);L^{r'})}.
\]
Apply this with the source
\[
\mathbf1_{(0,t)}(\tau)k_s(t-\tau)\Phi(\tau).
\]
Using unitarity of $S(t)$, we obtain
\begin{equation}\label{E:Volterra-classical-Ltwo}
\|\mathcal D_s\Phi\|_{C([0,T];L^2)}
\le
C_{d,r}
\|k_s\|_{L^{q_{\mathrm{cl}}'}(0,T)}
\|\Phi\|_{C([0,T];L^{r'})}.
\end{equation}
The kernel belongs to this Lebesgue space whenever
$s>\beta/2$, which is ensured by $s\ge s_0>\beta$.

To justify the continuous $L^2$ representative in
\eqref{E:Volterra-classical-Ltwo}, first take smooth
sources with values in $\mathscr S(\Rd)$ and then
approximate in $C([0,T];L^{r'})$. The estimate gives
convergence in $C([0,T];L^2)$. The resulting function
agrees, as a distribution, with the $L^r$-valued
integral already defined above.

The same argument gives
\begin{equation}\label{E:Volterra-operator-limit-Ltwo}
\begin{aligned}
&\|(\mathcal D_s-\mathcal D_1)\Phi\|_{C([0,T];L^2)}
\\
&\qquad\le
C_{d,r}
\|k_s-1\|_{L^{q_{\mathrm{cl}}'}(0,T)}
\|\Phi\|_{C([0,T];L^{r'})}.
\end{aligned}
\end{equation}
Since $s_0>\beta/2$, dominated convergence implies
\[
\|k_s-1\|_{L^{q_{\mathrm{cl}}'}(0,T)}
\longrightarrow0.
\]
Applying \eqref{E:Volterra-classical-Ltwo} and
\eqref{E:Volterra-operator-limit-Ltwo} to the difference
of the two equations proves
\[
u_s\longrightarrow u_1
\quad\text{in }C([0,T];L^2).
\]
This proves the asserted convergence on the common
local interval.

Uniqueness in the whole class $C([0,T];L^r)$ follows
by subdivision. On a sufficiently short interval,
the difference estimate absorbs the nonlinear term.
Once two solutions agree up to a given time, their
earlier contributions to the Volterra equation cancel,
and the same argument applies on the next interval.

We next extend the convergence to any finite interval
on which the classical solution belongs to
$C([0,T];L^2\cap L^r)$. Put
\[
M=\sup_{0\le t\le T}\|u_1(t)\|_{L^r}.
\]
Choose $\delta>0$ so small that the Volterra operator
restricted to an interval of length at most $\delta$
has nonlinear Lipschitz constant at most $1/4$ on
the ball of radius one about $u_1$, uniformly for
$s\in[s_0,1]$. This is possible by
\eqref{E:uniform-Volterra-classical}.

Partition $[0,T]$ into finitely many such intervals.
Suppose that $u_s$ has already been constructed up to
$t_j$ and converges to $u_1$ there. On
$I_j=[t_j,t_{j+1}]$, define
\[
\begin{aligned}
\mathcal T_{s,j}(u)(t)
={}&b_s(t)
+\mu\int_0^{t_j}
k_s(t-\tau)S(t-\tau)N(u_s(\tau))\,d\tau\\
&+\mu\int_{t_j}^t
k_s(t-\tau)S(t-\tau)N(u(\tau))\,d\tau.
\end{aligned}
\]
The first integral retains the already constructed
part of the solution. Substituting $u_1$ gives
\[
\begin{aligned}
\mathcal T_{s,j}(u_1)-u_1
={}&b_s-b_1
+\mu(\mathcal D_s-\mathcal D_1)N(u_1)\\
&+\mu\int_0^{t_j}
k_s(t-\tau)S(t-\tau)
\bigl(N(u_s(\tau))-N(u_1(\tau))\bigr)\,d\tau .
\end{aligned}
\]
The first two terms tend to zero uniformly in
$L^r$ on $[0,T]$. The last tends to zero uniformly
on $I_j$ by the induction hypothesis and the
dispersive estimate. Thus, for $s$ sufficiently
close to $1$, $\mathcal T_{s,j}$ maps the ball of
radius one about $u_1$ into itself and is a
contraction there. Its fixed point matches the
previously constructed solution at $t_j$ and
converges to $u_1$ on $I_j$.

A finite induction proves existence and convergence
in $C([0,T];L^r)$. The whole-interval estimates
\eqref{E:Volterra-classical-Ltwo} and
\eqref{E:Volterra-operator-limit-Ltwo} then give
convergence in $C([0,T];L^2)$.

Finally, these solutions admit the half-space
reconstruction. For $s<1$, put $a=1-2s$ and
\[
q_s=\frac2{1-s+\beta}.
\]
Since $s>\beta$, we have $q_s>2$, and
\[
\frac2{q_s}+\frac dr=\frac{d+a+1}{2}.
\]
Moreover,
\[
N(u_s)\in C([0,T];L^{r'})
\subset L^{q_s'}((0,T);L^{r'}).
\]
The boundary forcing estimates of \cite{GS2}
therefore define
\[
U_s(t)
=
\mathbb S_a(t)\mathcal P_su_0
+
c_s\mu\Theta_a^\star N(u_s)(t)
\in C([0,T];L_a^2),
\]
whose distributional boundary trace is $u_s$.
Thus the construction remains within the extension
formulation.

\end{proof}

At the fixed power $p=1+4/d$, the preceding theorem
uses the regularity of the prescribed history.
Indeed,
\[
1+\frac4d>
1+\frac{4s}{d+2-2s}
\qquad (s<1).
\]
For $p\le p_*(s)$, the solutions constructed above
belong to the boundary Strichartz class of
Theorem~\ref{T:memory-boundary-full} and agree with
those solutions by uniqueness.

\section{Appendix}\label{S:app}

We record the Bessel identities used in the proofs. For the
classical theory, see \cite{Wa,Le}. 
Consider the Bessel equation of order $\nu\in \mathbb C$
\begin{equation}\label{besseleq}
z^2 \frac{d^2 J}{dz^2} + z \frac{dJ}{dz} + (z^2 - \nu^2)J = 0.
\end{equation}
Assume $\nu\notin\mathbb Z$. Then, two linearly independent solutions of \eqref{besseleq} are the Bessel function of the first kind and order $\nu$ 
\begin{equation}\label{besseries}
J_\nu(z) = \sum_{k=0}^\infty \frac{(-1)^k}{\G(k+1) \G(\nu+k+1)} \left(\frac z2\right)^{2k+\nu}\ \ \ \ \ \ \ \ |\arg z|<\pi,
\end{equation}
and the function 
\begin{equation}\label{negbes}
J_{-\nu}(z)=\sum_{k=0}^\infty \frac{(-1)^k}{\G(k+1) \G(k+1-\nu)} \left(\frac z2\right)^{2k-\nu}\ \ \ \ \ \ \ \ |\arg z|<\pi.
\end{equation}

From \eqref{besseries}, \eqref{negbes} we have as $z\to 0$, $|\arg z|<\pi$,
\begin{equation}\label{Js}
J_{\nu}(z) \cong \frac 1{\G(\nu+1)} \left(\frac{z}{2}\right)^{\nu},\ \ \ \ J_{-\nu}(z) \cong \frac 1{\G(1-\nu)} \left(\frac{z}{2}\right)^{-\nu}.
\end{equation}
For fixed $\nu$ and $0<\delta<\pi$, the large-argument expansion is
\begin{align}\label{jnuinfty}
J_\nu(z)&=\sqrt{\frac2{\pi z}}
\cos\left(z-\frac{\pi\nu}2-\frac\pi4\right)
+O\!\left(|z|^{-3/2}e^{|\operatorname{Im}z|}\right),
\\[-1mm]
&\hspace{8mm}|z|\to\infty,\qquad |\arg z|\le\pi-\delta.
\notag
\end{align}
We also note that the functions
\begin{equation}\label{FG}
F_\nu(z) = z^\nu J_\nu(z),\ \ \ \ \ G_\nu(z) = z^{-\nu} J_{\nu}(z),
\end{equation}
satisfy the following recurrence relations, see \cite[(5.3.5), p.103]{Le} (note that, to obtain the second relation in \eqref{recu}, we have changed $\nu$ into $-\nu$ in the second relation in (5.3.5). This is allowed since (5.3.5) is valid for arbitrary $\nu\in \C$)
\begin{equation}\label{recu}
\frac{d}{dz} F_\nu(z) = z^\nu J_{\nu-1}(z),\ \ \ \ \ \frac{d}{dz} G_{\nu}(z) = - z^{-\nu} J_{1+\nu}(z).
\end{equation}
If we let $\mu = -\nu$, then the second relation in \eqref{recu} gives
\begin{equation}\label{recubis}
\frac{d}{dz} [G_{-\nu}(z)] = \frac{d}{dz}[z^\nu J_{-\nu}(z)] = - z^\nu J_{1-\nu}(z).
\end{equation}
Consider the generalized Bessel equation
\begin{equation}\label{genbessel}
z^2 \Psi''(z) + (1 - 2\alpha) z \Psi'(z) + \left[\beta^2 \gamma^2
z^{2\gamma} + (\alpha^2 - \nu^2 \gamma^2)\right] \Psi(z) = 0.
\end{equation}
If $\Phi(z)$ is a solution to 
\eqref{besseleq}, then the function defined by the
transformation
\begin{equation}\label{besselcv}
\Psi(z) = z^\alpha \Phi(\beta z^\gamma)
\end{equation}
solves the equation \eqref{genbessel}.

Consider now the modified Bessel equation of order $\nu\in \mathbb C$
\begin{equation}\label{modbesseleq}
z^2 \frac{d^2\Phi}{dz^2} + z \frac{d\Phi}{dz} - (z^2 + \nu^2)\Phi = 0.
\end{equation}
For $\nu\notin\mathbb Z$, two linearly independent solutions of \eqref{modbesseleq} are the modified Bessel function of the first kind, 
\begin{equation}\label{Inu}
I_\nu(z) = \sum_{k=0}^\infty \frac{(z/2)^{\nu+2k}}{\G(k+1) \G(k+\nu+1)},\ \ \ \ \ \ \  \ \ |\arg(z)| < \pi,
\end{equation}
and the modified Bessel function of the third kind, or Macdonald function, which is defined by
\begin{equation}\label{Knu}
K_\nu(z) = \frac \pi{2} \frac{I_{-\nu}(z) - I_\nu(z)}{\sin \pi \nu},\ \ \ \ \ \ \ \ \ \ \  \ |\arg(z)| < \pi.
\end{equation}
Notice that 
\begin{equation}\label{same}
K_\nu(z) = K_{-\nu}(z).
\end{equation}
From \eqref{Inu} we have as $z\to 0$, $|\arg z|<\pi$,
\begin{equation}\label{Is}
I_{\nu}(z) \cong \frac 1{\G(\nu+1)} \left(\frac{z}{2}\right)^{\nu},\ \ \ \ I_{-\nu}(z) \cong \frac 1{\G(1-\nu)} \left(\frac{z}{2}\right)^{-\nu}.
\end{equation}
If $\nu> 0$, we infer from \eqref{Knu} and \eqref{Is}
\[
K_{\nu}(z) \cong 2^{\nu-1} \frac{\pi}{\sin(\pi \nu)} \frac{z^{-\nu}}{\G(1-\nu)}.
\]
Combining this asymptotic with the formula, see \cite[(1.2.2) p.3]{Le},  
\begin{equation}\label{sine}
\G(z) \G(1-z) = \frac{\pi}{\sin(\pi z)},
\end{equation} 
we infer when $\nu> 0$
\begin{equation}\label{zeroKnu}
K_{\nu}(z) \cong 2^{\nu-1} \G(\nu) z^{-\nu},\ \ \ \ \ z\to 0,\ \ |\arg z|<\pi.
\end{equation}
The following recurrence relation holds, see \cite[(5.7.9), p.110]{Le}
\begin{equation}\label{recuu}
\frac{d}{dz} [z^\nu K_\nu(z)] = - z^\nu K_{\nu-1}(z) = -  z^\nu K_{1-\nu}(z),
\end{equation}
where in the last equality we have used \eqref{same}.
If we consider the generalized modified Bessel equation
\begin{equation}\label{modgenbessel}
z^2 \Psi''(z) + (1 - 2\alpha) z \Psi'(z) + \left[(\alpha^2 - \nu^2 \gamma^2)- \beta^2 \gamma^2
z^{2\gamma}\right] \Psi(z) = 0,
\end{equation}
if the function $\Phi$ solves \eqref{modbesseleq}, then the function defined by \eqref{besselcv}
is a solution to \eqref{modgenbessel}.
For fixed $\nu$, as $|z|\to\infty$ with
$|\arg z|\le\pi/2-\delta$, $0<\delta<\pi/2$, we have
\begin{equation}\label{ab}
I_\nu(z)=\frac{e^z}{(2\pi z)^{1/2}}\left(1+O(|z|^{-1})\right).
\end{equation}
Whereas $K_\nu$ has the asymptotic formula
\begin{equation}\label{abK}
K_\nu(z) = \left(\frac{\pi}{2z}\right)^{1/2} e^{-z}  \left(1+ O(|z|^{-1})\right), \ \ \ \ \ |\arg z| \le \pi - \delta,
\end{equation}
see \cite[(5.11.9) on p.123]{Le}. Finally, we use the following integral representation 
\begin{equation}\label{intK}
K_\nu(z)=2^{-\nu-1}z^\nu
\int_0^\infty e^{-\tau}e^{-z^2/(4\tau)}
\tau^{-\nu-1}\,d\tau,\qquad |\arg z|<\frac{\pi}{4}.
\end{equation}
see \cite[(5.10.25), p.119]{Le}.

\subsection{Some oscillatory integral identities}\label{S:osc2}

For the oscillatory integrals below, complex powers and square
roots are taken on the principal branch:
\[
w^\alpha=\exp(\alpha \operatorname{Log} w),
\qquad
\operatorname{Log} w=\log|w|+i\operatorname{Arg} w,
\qquad
-\pi<\operatorname{Arg} w<\pi.
\]
We begin with a basic formula.
 
 \begin{lemma}\label{L:aux}
 Let $A\in \R\setminus\{0\}$ and $0<s<1$. Then
 \begin{equation}\label{As}
 -\frac{s}{\G(1-s)} \int_0^\infty \frac{1}{\sigma^{1+s}} [e^{-i \sigma A} - 1] d\sigma = (i A)^s = |A|^s e^{i\frac{\pi s}{2} \operatorname{sgn}A}.
 \end{equation}
 \end{lemma}
 
 \begin{proof}
We recall the following well-known integral:
\[
-\frac{s}{\G(1-s)} \int_0^\infty \frac{1}{\sigma^{1+s}} [e^{-\sigma L } - 1] d\sigma = L^s,\ \ \ \ \ \ \Re L>0.
\]
Applying this formula with $L = \ve+iA$, $\ve>0$, and letting $\ve\to 0^+$ we reach the desired conclusion \eqref{As}.

 \end{proof}
 
We will also need the following classical oscillatory integral.

\begin{lemma}\label{L:cis}
Let $b>0$ and $0<\Re \mu<1$. Then
\begin{equation}\label{ei}
\int_0^\infty y^{\mu-1} e^{\pm i by} dy = \frac{\G(\mu)}{b^\mu} e^{\pm i\frac{\mu \pi}{2}}.
\end{equation} 
\end{lemma}

\begin{proof}
The proof of \eqref{ei} directly follows from:
\begin{equation}\label{sincosy}
\int_0^\infty y^{\mu-1} \sin(by) dy = \frac{\G(\mu)}{b^\mu} \sin \frac{\mu \pi}{2},\ \ \ \ \ \ \int_0^\infty y^{\mu-1} \cos(by) dy = \frac{\G(\mu)}{b^\mu} \cos \frac{\mu \pi}{2},
\end{equation}
see \cite[formulas 3.761.4 \& 3.761.9, p.420-1]{GR}.

\end{proof}

The following identity follows from Lemma~\ref{L:cis}.

\begin{lemma}\label{L:Abel}
Let $0<s<1, t>0$, and $z>0$. Then the following oscillatory identity holds:
\begin{equation}\label{abel}
\int_0^t
\frac{e^{i \frac{z^2}{4\sigma}}}{(t-\sigma)^{1-s}\sigma^{1+s}}\,d\sigma =
4^s\,\G(s)\,e^{i\frac{\pi s}{2}}\,t^{s-1}\,z^{-2s}
e^{i \frac{z^2}{4t}}.
\end{equation}
\end{lemma}

\begin{proof}
Since the factor $\sigma^{-{(1+s)}}$ is not integrable at $\sigma=0$, we interpret the integral $I(t,z)$ in the left-hand side of \eqref{abel}
as a generalized Riemann integral in which we first make the change of variable $\sigma = t u$, and then $\tau = u^{-1}$, obtaining
\begin{align*}
I(t,z) & = \underset{\ve\to 0^+}{\lim} \int_{\ve}^t \frac{e^{i\frac{z^2}{4\sigma}}}
{(t-\sigma)^{1-s}\sigma^{1+s}}\,d\sigma = \frac{1}{t}\underset{\ve\to 0^+}{\lim}  \int_{\ve/t}^1\frac{e^{i\frac{z^2}{4t u}}}
{(1-u)^{1-s} u^{1+s}}\,du
\\
& = \frac{1}{t} \underset{\ve\to 0^+}{\lim} \int_1^{t/\ve}\frac{e^{i\frac{z^2}{4t }\tau}}
{(\tau-1)^{1-s}}\,d\tau = \frac{e^{i\frac{z^2}{4t }}}{t} \underset{\ve\to 0^+}{\lim} \int_0^{t/\ve -1} y^{s-1}e^{i\frac{z^2}{4t } y}\,dy.
\end{align*}
If we now apply \eqref{ei} with $\mu = s$, we reach the desired conclusion.

\end{proof}


\section*{Declarations}

\smallskip
\noindent\textbf{Competing interests.} The authors have no competing
interests to declare that are relevant to the content of this article.

\smallskip

\noindent \textbf{Use of AI:} We acknowledge the assistance of ChatGPT (OpenAI) with
the organization and revision of the manuscript and
the review of mathematical arguments. The research ideas
and mathematical approach are exclusively our own, and we take full
responsibility for the mathematical content and final text.

\bibliographystyle{amsplain}

\end{document}